\documentclass[11pt,letterpaper]{amsart}
\usepackage{graphicx} 
\usepackage{amsmath, amssymb, amsthm, amsfonts}
\usepackage{enumerate}
\usepackage{fancyhdr, fullpage}
\usepackage{array, xcolor, mathtools}
\PassOptionsToPackage{hyphens}{url}\usepackage{hyperref}
\usepackage{colortbl}
\usepackage{tikz}
\usepackage{stmaryrd}
\usepackage{relsize}
\usetikzlibrary{arrows, matrix, positioning}
\usepackage[
    backend=biber,
    style= alphabetic,
    maxbibnames=9,
  ]{biblatex}
\theoremstyle{plain} 
\usepackage{etoolbox}
\usepackage{comment}
\usepackage{hyperref}
\makeatother
\usepackage{enumerate}
\theoremstyle{plain}
\newtheorem{thrm}{Theorem}[section]
\usepackage{float}
\theoremstyle{definition}
\newtheorem{defn}[thrm]{Definition}
\newtheorem{exmp}[thrm]{Example}
\newtheorem{cor}[thrm]{Corollary} 
\newtheorem{lem}[thrm]{Lemma} 
\newtheorem{prop}[thrm]{Proposition}
\usepackage{multicol}

\newtheorem*{thrm*}{Theorem}

\newcommand{\heart}{\ensuremath\heartsuit}
\allowdisplaybreaks

\DeclareMathOperator{\ad}{ad}
\DeclareMathOperator{\Ad}{Ad}
\DeclareMathOperator{\Adm}{Adm}

\DeclareMathOperator{\ani}{ani}

\DeclareMathOperator{\Bun}{Bun}
\DeclareMathOperator{\charac}{char}

\DeclareMathOperator{\Cox}{Cox}

\DeclareMathOperator{\End}{End}
\DeclareMathOperator{\eva}{ev}
\DeclareMathOperator{\Fl}{Fl}

\DeclareMathOperator{\GL}{GL}
\DeclareMathOperator{\Gr}{Gr}

\DeclareMathOperator{\Hom}{Hom}
\DeclareMathOperator{\id}{id}

\DeclareMathOperator{\Lie}{Lie}

\DeclareMathOperator{\para}{par}
\DeclareMathOperator{\Pic}{Pic}

\DeclareMathOperator{\reg}{reg}

\DeclareMathOperator{\RHom}{RHom}
\DeclareMathOperator{\rk}{rk}
\DeclareMathOperator{\rs}{rs}

\DeclareMathOperator{\SL}{SL}

\DeclareMathOperator{\Spec}{Spec}

\DeclareMathOperator{\supp}{supp}
\DeclareMathOperator{\Supp}{Supp}

\DeclareMathOperator{\Tr}{Tr}
\DeclareMathOperator{\val}{val}
\DeclareFontFamily{U}{mathx}{}
\DeclareFontShape{U}{mathx}{m}{n}{<-> mathx10}{}
\title{Parabolic Multiplicative Affine Springer Fibers}
\author{Marielle Ong}
\begin{document}

\maketitle
\begin{abstract}
We introduce parabolic multiplicative affine Springer fibers, which resemble the admissible union of affine Deligne Lusztig varieties in the affine flag variety. We also study their global counterparts called parabolic multiplicative Hitchin fibers. Their associated fibration is a global analogue of the Grothendieck simultaneous resolution for monoids. Using this fibration, we show that the parabolic multiplicative affine Springer fibers are equidimensional and find their dimension.  
\end{abstract}
\tableofcontents
\section{Introduction}
Multiplicative affine Springer fibers are certain group-theoretic generalizations of affine Springer fibers in which the Lie algebra $\mathfrak{g}(\mathcal{O})$ in the definition is replaced by an affine Schubert cell $G(\mathcal{O})t^\lambda G(\mathcal{O})$ or an affine Schubert variety $\overline{G(\mathcal{O})t^\lambda G(\mathcal{O})}$ associated to a dominant coweight $\lambda$. As a result of this replacement, they parametrize multiplicative Higgs bundles over the formal disk with a trivialization over the punctured disk. These bundles are like linear Higgs bundles except that their Higgs fields are now valued in the so-called Vinberg monoid instead of the Lie algebra. The global analogues of multiplicative affine Springer fibers are multiplicative Hitchin fibers. These assemble to form the multiplicative Hitchin fibration, which is induced by the adjoint quotient of the Vinberg monoid. In this paper, we develop a parabolic version of multiplicative affine Springer fibers and multiplicative Hitchin fibers. We also study their dimensions and equidimensionality. 
\subsection{Background} 
Multiplicative affine Springer fibers first appeared in the context \cite{KV10} of developing a linear (as opposed to $\sigma$-linear), group-theoretic version of Katz's Hodge-Newton decomposition of $F$-isocrystals \cite{Kat79}. They were then used by Frenkel-Ng\^{o} \cite{FN11} to give a geometric interpretation of orbital integrals of spherical Hecke functions. This particular work provided a way to view multiplicative Higgs bundles in terms of the Vinberg monoid. Originally, Markman-Hurtubise \cite{HM03} introduced a multiplicative Higgs bundle on a smooth algebraic curve $X$ as a pair $(E, \varphi)$ consisting of a $G$-bundle $E$ over $X$ and a section $\varphi$ 
of the adjoint group bundle $E \wedge^G G$ over the complement $X \setminus \supp(D)$, where $D$ is a fixed divisor. Furthermore, the singularities of the meromorphic section $\varphi$ at each point $x \in \supp(D)$ is controlled by a dominant coweight $\lambda_x$. This definition of a multiplicative Higgs bundle has been useful in the realm of gauge theory and mathematical physics. Charbonneau-Hurtubise \cite{CH09} and Smith \cite{Smi15} showed that stable multiplicative Higgs bundles on $X$ correspond to singular Hermitian-Einstein monopoles on $S^1 \times X$. When $X = \mathbb{C}$, Elliot-Pestun \cite{EP19} upgraded this correspondence to that of hyperk\"{a}hler manifolds and analyzed their twistor structures. Unfortunately, this definition does not lend itself neatly to the language of mapping stacks, which is how the theory of linear Higgs bundles and the Hitchin fibration is phrased in \cite{Ngo08}. Linear Higgs bundles arise as maps to the stack $[(\mathfrak{g} \wedge^{\mathbb{G}_m} L)/G]$, where $L$ is a $\mathbb{G}_m$-torsor. To create a group-theoretic analogue, we should consider maps to the stack $[G/G]$ with some added twist. However, there is no natural way of twisting $[G/G]$ by a $T$-torsor and $[G/G]$ may not have non-trivial sections. To overcome these obstacles, we consider maps to the stack $[(V_G \wedge^T L)/G]$, where $V_G$ is the Vinberg monoid of $G$ and $L$ is a $T$-torsor.

Since \cite{FN11}, there has been an ongoing program to replicate the theory of Higgs bundles for the multiplicative case from the viewpoint of the Vinberg monoid. Bouthier \cite{Bou12, Bou14} and Chi \cite{Chi22} studied geometric aspects of mutliplicative affine Springer fibers and Hitchin fibers. In particular, Chi showed that the fibers were equidimensional and established their dimension formula using the multiplicative Hitchin fibration. Most recently, Wang \cite{Wang23} proved the Fundamental Lemma for the spherical Hecke algebra of an unramified reductive group while Gallego-Garcia-Prada \cite{GGP24} studied multiplicative Higgs bundles with involutions. This article is a continuation of this program and it seeks to develop a parabolic analogue of multiplicative affine Springer fibers and Hitchin fibers.

\subsection{Main results} The goal of this paper is to introduce parabolic multiplicative affine Springer fibers and parabolic multiplicative Hitchin fibers. Let $G$ be a split connected reductive group over an algebraically closed field $k$ with Weyl group $W$ and extended affine Weyl group $\widetilde{W}.$ 
\begin{defn}
\emph{Parabolic multiplicative affine Springer fibers} of a regular semisimple element $\gamma \in G(k(\!(t)\!)))$ and a dominant coweight $\lambda$ are the sets
\begin{align*}
X^{\lambda, \para}_\gamma &= \left\{g \in G(k(\!(t)\!))/I: g^{-1}\gamma g \in \bigcup_{\mu \in W(\lambda)} It^\mu I\right\}, \\
X^{\leq \lambda, \para}_\gamma &= \left\{g \in G(k(\!(t)\!))/I: g^{-1}\gamma g \in \bigcup_{w \in \Adm(\lambda)} IwI\right\},
\end{align*}
where $I$ is the Iwahori subgroup of $G(k[\![t]\!])$ and $$\Adm(\lambda) = \{w \in \widetilde{W}: w \leq t^{x(\lambda)} \text{ for some } w \in W\}$$ is the $\lambda$-admissible subset of the extended affine Weyl group $\widetilde{W}$. 
\end{defn}
When $\charac k = p,$ the closed version $X^{\leq \lambda, \para}_\gamma$ is the admissible union of affine Deligne Lusztig varieties in the affine flag variety without the Frobenius-twisted conjugation. The admissible union arises as the group-theoretic model for the Newton stratum inside the special fiber of Shimura varieties with Iwahori level structure. This work presents an interpretation of the admissible union via parabolic multiplicative Higgs bundles.

We establish their non-emptiness criterion, which turns out to be the same as the non-emptiness criterion for multiplicative affine Springer fibers. 
\begin{lem}[Lemma \ref{nonemppar}] The natural projection $G(k(\!(t)\!))/I \rightarrow G(k(\!(t)\!))/G(k[\![t]\!])$ induces a surjective map $X^{ \lambda, \para}_\gamma \rightarrow X^{\lambda}_\gamma$. Furthermore, $X^{\lambda, \para}_\gamma$ is non-empty if and only if $X^\lambda_\gamma$ is non-empty. 
\end{lem}
Their ind-scheme structure and description as a moduli space of parabolic multiplicative Higgs bundles on the formal disk is given in Definition \ref{pmasfdefn}. In order to describe parabolic multiplicative affine Springer fibers in terms of the Vinberg monoid, we introduce the notion of an \textit{Iwahori submonoid}, which lives in $V_G(k[\![t]\!]).$

We then move on to the global analogues and define \emph{parabolic multiplicative Higgs bundles}. These are quadruples $(x, E, \varphi, E^B_x)$ consisting of a point $x$ on $X$, a multiplicative Higgs bundle $(E, \varphi)$ and a $B$-reduction $E^B_x$ of $E$ such that the Higgs field $\varphi$ is compatible with the $B$-reduction at $x$. We denote the moduli stack of these quadruples as $\mathcal{M}^{\para}_X.$ The Vinberg monoid $V_G$ admits a Grothendieck simultaneous resolution, which induces the \emph{parabolic mulitplicative Hitchin fibration}
\begin{align*} h^{\para}_X: \mathcal{M}^{\para}_X \rightarrow \mathcal{A}_X \times X, \quad (x, E, \varphi, E^B_x) \mapsto (h_X(E, \varphi), x),
\end{align*}
where $h_X: \mathcal{M}_X \rightarrow \mathcal{A}_X$ is the multiplicative Hitchin fibration from the moduli stack $\mathcal{M}_X$ of multiplicative Higgs bundles over $X$ to its affine base $\mathcal{A}_X$. Unlike the linear case, the moduli stack $\mathcal{M}_X$ is not smooth. Its open subset $\mathcal{M}^\heart_X$ admits a local model of singularities whose fibers are given by the product of quotients of affine Schubert varieties. The parabolic Hitchin stack $\mathcal{M}^{\para}_X$ is not smooth either. We will use deformation theoretic methods similar to \cite[\S 7]{Wang23} to show that $\mathcal{M}^{\para, \heart}_X$ admits a local model $[\widetilde{Q}_X]_G$ of singularities, which, fiberwise, is a product of quotients of admissible unions. 
\begin{thrm}[Theorem \ref{localmodel}]
Let $m = (L, x, E, \varphi, E^B_x) \in \mathcal{M}^{\para, \heart}_X.$ If $L$ is very $(G, \delta_a + \rk(G))$-ample where $a = h_X(m)$, then the local evaluation map
$$\eva^{\para}: \mathcal{M}^{\para, \heart}_X \rightarrow [\widetilde{Q}_X]_G,$$
is formally smooth at $m$. Meanwhile, the truncated version
$\eva^{\para}_N: \mathcal{M}^{\para, \heart}_X \rightarrow [\widetilde{Q}_X]_{G,N}$
is smooth at $m$. 
\end{thrm}

Another problem explored in this paper is the dimension formula and equidimensionality of parabolic multiplicative affine Springer fibers. With an algebraic-combinatorial approach, \cite{He23} defined affine Lusztig varieties, which are non-Frobenius-twisted analogs of affine Deligne Lusztig varieties in the affine flag variety. He then showed that the difference between the dimensions of affine Lusztig varieties and affine Deligne Lusztig varieties is the dimension of an affine Springer fiber. We cannot use his method as a general dimension formula for the admissible union remains unknown. It is only known for certain cases of $\gamma$ and $\lambda$ by \cite{GHKR10}, \cite{HY21} and \cite{Sad22}.

Instead, our strategy will be similar to \cite{Chi22}. We will first compute the dimension of parabolic multiplicative affine Springer fibers for unramified or split group elements. In particular, we will find that these can be related to the dimensions of Mirkovic-Vilonen cycles in the affine flag variety associated to translation elements. We will then use the Product formula (Theorem \ref{paraprod}) and the local constancy of parabolic multiplicative affine Springer fibers (Theorem \ref{localconstancy}) to compare fibers of unramified group elements with general elements. We conclude the following theorem,
\begin{thrm}[Propositions \ref{paradim} and \ref{parcor}]
Under certain ampleness conditions on the boundary divisor, the parabolic multiplicative Hitchin fiber is equidimensional. Its corresponding parabolic multiplicative affine Springer fiber is equidimensional and it shares the same dimension as a multiplicative affine Springer fiber. 
\end{thrm}

\subsection{Outline}
In Section 2, we will briefly review the geometry of the Vinberg monoid. In its later sections, we develop some basic properties about Vinberg Borel submonoids, Iwahori submonoids and a monoid-version of the Grothendieck simultaneous resolution. Section 3 will also review the theory and results pertaining to multiplicative affine Springer fibers while Section 4 covers the multiplicative Hitchin fibration. These sections are a summary of the work completed in \cite{Bou12}, \cite{Chi22} and \cite{Wang23}. 

Section 5 is devoted to introducing parabolic multiplicative affine Springer fibers and proving its non-emptiness pattern. In Section 6, we define the parabolic multiplicative Hitchin fibration in the style of \cite{Yun11} and establish some basic geometric properties. Since Section 6.6 is the most technical and computationally-involved, we have broken it up into multiple subsections. In Section 7, we prove our main theorem on the dimension formula and equidimensionality. 

\subsection{Notation} Let $k$ be an algebraically closed field, $F = k((t))$ and $\mathcal{O} = k[[t]]$. 
Let $G$ be a split connected reductive group over $k$ with split maximal torus $T$, fixed Borel group $B$ containing $T$ and rank $r$.
Assume that $\charac k = 0$ or $\charac k > 0$ does not divide the order of the Weyl group $W $ of $G$. We fix the following notation:
\begin{itemize}
\item center $Z$ of $G$ and adjoint group $G_{\ad} = G/Z$
\item the stabilizer $G_x$ of an element $x$ in a $G$-set $S$ 
\item Weyl group $W = N_G(T)/T$ with longest element $w_0 \in W$
\item weight lattice $X^*(T) = \Hom(T, \mathbb{G}_m)$ and coweight lattice $X_*(T) = \Hom(\mathbb{G}_m, T)$
\item set of dominant weights $X^*(T)^+$ and set of dominant coweights $X_*(T)^+$
\item natural pairing $\langle \cdot, \cdot \rangle: X^*(T) \times X_*(T) \rightarrow \mathbb{Z}$
\item set of roots $\Phi$ and set of coroots $\Phi^\vee$
\item set of positive roots $\Phi^+$ and set of simple roots $\Delta$
\item root lattice $\Phi_\mathbb{Z}$ and coroot lattice $\Phi^\vee_\mathbb{Z}$
\item $G^{\reg}$ is the set of regular elements of $G$
\item $G^{\rs}$ is the set of regular semisimple elements of $G$
\item the affine Grassmanian $\Gr_G$ and the affine flag variety $\Fl_G$ of $G$
\item a smooth, projective, connected curve $X$ of genus $g$ over $k$ with its set $|X|$ of closed points
\item If $E$ is a principal $G$-bundle and $F$ is a $G$-set, then we denote $E \wedge^G F$ as the associated bundle. 
\item Given maps $X \rightarrow Z$ and $Y \rightarrow Z$ in the category of schemes, we denote $X \times_Z Y$ as the pullback.

\end{itemize}

\noindent \textit{Acknowledgements} I would like to thank Jingren Chi for providing much helpful guidance and direction. I would also like to especially thank Griffin Wang for his careful reading of the first draft, his valuable technical comments, and for patiently answering my questions. I would like to thank Ron Donagi, Ng\^{o} Bao Ch\^{a}u, Xuhua He, Tony Pantev and Xinwen Zhu for the enlightening conversations, as well as Benedict Morrissey for suggesting this topic. Much thanks to Emmett Lennen and Michael Gr\"{o}echenig for their comments on the first draft. This paper was written as part of my Ph.D. thesis at the University of Pennsylvania. 
\section{The Vinberg monoid}
\subsection{Constructions}
Let $G$ be a semisimple, simply connected group over $k$ of rank $r$ with center $Z$. Fix a maximal torus $T$ and a Borel subgroup $B$ containing it. 
\begin{defn}
The \emph{enhanced group} of $G$ is given by the quotient $$G_+ = (T \times G)/Z,$$ where $Z$ embeds anti-diagonally into $T \times G$. That is, if $z \in Z$ and $(t, g) \in T \times G$, then $z\cdot (t, g) = (z^{-1}t, zg).$
\end{defn}
It is a reductive group with center $Z_+ \cong T$, derived subgroup $G$ and maximal torus $T_+$ that is contained in a Borel subgroup $B_+$.

For $1 \leq i \leq r$, let $\alpha_i$ be the simple roots, $\omega_i$ be the fundamental weights and $\rho_i: G \rightarrow \GL(V_{\omega_i})$ be the irreducible representations of $G$ with highest weight $\omega_i$. We can extend each simple root $\alpha_i$ and representation $\rho_i$ to those of $G_+$ as follows:
\\
\begin{equation*}
\begin{aligned}[c]
\alpha^+_i: G_+ &\rightarrow \mathbb{G}_m,\\
(t, g) &\mapsto \alpha_i(t), \qquad 
\end{aligned}
\begin{aligned}[c]
\rho^+_i: G_+ & \rightarrow \GL(V_{\omega_i}), \\ (t, g) &\mapsto \omega_i(t)\rho_i(g). \end{aligned} 
\end{equation*}
\\
These are the components of the closed embedding, 
$$(\alpha^+, \rho^+) := (\alpha^+_1, ..., \alpha^+_r, \rho^+_1, ..., \rho^+_r) : G_+ \hookrightarrow \mathbb{G}^r_m \times \prod^r_{i=1} \GL(V_{\omega_i}).$$
\begin{defn}
The \emph{Vinberg monoid $V_G$} of $G$ is the normalization of the closure of the image of $(\alpha^+, \rho^+)$ in 
$$\mathbb{A}^r \times \prod^r_{i=1} \End(V_{\omega_i}).$$
\end{defn}
It is an affine algebraic monoid with unit group $G_+$. The inverse image of   
$$\mathbb{A}^r \times \prod^r_{i=1} (\End(V_{\omega_i})\setminus \{0\})$$ in $V_G$ forms a smooth, dense, open subvariety $V^0_G$ called the \emph{non-degenerate locus}. 
\subsection{The abelianization}
The group $(G \times G)$ acts on $G_+$ by left and right multiplication while $T$ acts on $G_+$ by multiplication onto the first factor. Both of these actions extend to actions on $V_G$ and they commute with each other. 
\begin{defn}
The \emph{abelianization of $V_G$} is the GIT quotient $A_G = V_G\sslash(G \times G)\cong \mathbb{A}^r$ along with a monoid homomorphism $\alpha_G: V_G \rightarrow A_G$. It is a commutative monoid whose unit group is the torus $T_{\ad}$. 
\end{defn}
The abelianization map $\alpha_G$ is $(G \times G)$-invariant and $T$-equivariant and it extends the natural group homomorphism
$$G_+ = (T \times G)/Z \rightarrow T_{\ad} = T/Z, \quad (t, g) \mapsto t \pmod Z.$$
The abelianization map of $V_G$ is flat with reduced and irreducible fibers. In fact, it is universal amongst all such abelianizations \cite[Theorem 5]{Vin95}. The abelianization $\alpha_G$ can be regarded as a contraction of the $(G \times G)$-action on $\alpha_G^{-1}(1) = G$ to the $(G \times G)$-action on the asymptotic monoid $\alpha_G^{-1}(0).$ The latter action is special, meaning that the stabilizer of any point in $\alpha_G^{-1}(0)$ contains a maximal unipotent subgroup of $G \times G.$ 

The map $T \rightarrow T \times G$, which sends $t \mapsto (t, t^{-1})$, descends to a map $d: T_{\ad} \rightarrow G_+.$ Since the isomorphism $T_{\ad} \cong d(T_{\ad})$ extends to an isomorphism $A_G \cong \overline{d(T_{\ad})} \subseteq V_G$, there is a section,
$$s_G: A_G \rightarrow V_G,$$ of $\alpha_G$. The non-degenerate locus $V^0_G$ meets each fiber of $\alpha_G$ in the open $(G \times G)$-orbit of that fiber.
\begin{exmp} \label{sl2}
If $G = \SL_2(\mathbb{C})$, then $T \cong \mathbb{G}_m$ and $G_+ \cong \GL_2(\mathbb{C})$. The representation-theoretic ingredients are the simple root and the fundamental weight
$$\alpha_1 \begin{pmatrix}
t & 0\\ 0& t^{-1}
\end{pmatrix} = t^2, \quad \omega_1\begin{pmatrix}
t & 0\\0 & t^{-1}
\end{pmatrix} = t,$$
as well as the standard representation $\rho_1 = \id$ of $G$. The Vinberg monoid $V_G$ is the normalization of the closure of the image of
\begin{align*}
(\alpha^+, \rho^+): G_+ &\rightarrow \mathbb{G}_m \times \GL_2(\mathbb{C}) \\
\left( \begin{pmatrix}
t & 0\\0 & t^{-1}
\end{pmatrix}, g \right) &\mapsto  \left(t^2, tg\right).
 \end{align*}
in $\mathbb{A}^1 \times \End(\mathbb{C}^2)$. More explicitly, it is given by 
$$V_G = \{(t, A) \in \mathbb{A}^1 \times \End(\mathbb{C}^2) : t = \det(A)\} \cong \End(\mathbb{C}^2),$$ 
and it has non-degenerate locus $V_G^0 \cong \End(\mathbb{C}^2) \setminus \{0\}$. The $(G \times G)$-action is given by the left and right multiplication while the $T$-action is given by scalar multiplication. The abelianization and section maps are 
\begin{align*}
\alpha_G: V_G  &\rightarrow \overline{T}_+ = \mathbb{A}^1 \\
 \begin{pmatrix}
a & b \\c & d
\end{pmatrix} &\mapsto \det \begin{pmatrix}
a & b \\c & d
\end{pmatrix}, \\ 
s_G: \begin{pmatrix}
1 & 0 \\0 & a
\end{pmatrix} &\mapsfrom a.
\end{align*}
\end{exmp}
Unfortunately, the Vinberg monoids for groups of higher rank are way more complicated.
\begin{exmp}
Let $G = \SL_3(\mathbb{C})$ with center $Z = \langle \lambda I : \lambda^3 = 1 \rangle$ and torus 
$$T = \left\{
t = \begin{pmatrix}
t_1 & 0 & 0 \\0 & t_2 & 0 \\ 0 & 0 &(t_1t_2)^{-1}
\end{pmatrix}: t_i \in \mathbb{C}\right\}.$$
Its simple roots are given by $\alpha_1(t) = t_1t_2^{-1}$ and $\alpha_2(t) = t_1t_2^2$. 
There are also two fundamental representations of $G$; notably $\rho_1 = \id$ with fundamental weight $\omega_1(t) = t_1$, and $\rho_2: g \mapsto (g^{-1})^T$ with fundamental weight $\omega_2(t) = t_1t_2$. We extend these representations to $G_+$ as follows:
\begin{align*}
\rho^+_1: (t,g) &\mapsto \omega_1(t)\rho_1(g) = t_1t_2g \\
\rho^+_2: (t, g) &\mapsto \omega_2(t)\rho_2(g) = t_1(g^{-1})^T. 
\end{align*}
The Vinberg monoid, $V_G$ is the normalization of the closure of the image of
\begin{align*}
(\alpha^+, \rho^+): G_+ &\rightarrow \mathbb{G}^2_m \times \GL_3(\mathbb{C}) \times \GL_3(\mathbb{C}), \\
\left( \begin{pmatrix}
t_1 & 0 & 0\\0 & t_2 & 0 \\ 0 & 0 & (t_1t_2)^{-1}
\end{pmatrix}, g \right) &\mapsto (t_1t_2^{-1}, t_1t_2^2, t_1t_2g, t_1(g^{-1})^T). 
 \end{align*}
It can be realized as the set
$$\left\{ (x, y, A_1, A_2) \in \mathbb{A}^2 \times \End(\mathbb{C}^3) \times \End(\mathbb{C}^3): A_1^TA_2 = A_1A^T_2 = xyI, \Lambda^2 A_1 = xA_2, \Lambda^2 A_2 = yA_1\right\}.$$
Its abelianization is given by $\alpha_G: V_G \rightarrow A_G$ in which $(x, y, A_1, A_2) \mapsto (x, y)$. 
\end{exmp}
\subsection{Adjoint quotient} Let $G$ act on $V_G$ by the adjoint or conjugation action. When restricted to $G_+$, this action is given by $$h((t, g)Z)h^{-1} = (t, hgh^{-1})Z$$ for all $h \in G$ and $(t, g)Z \in G_+$. 
\begin{defn}
The \emph{extended Steinberg base} is the GIT quotient $\mathfrak{C}_+ = V_G \sslash \Ad(G)$. It comes with a canonical quotient map $\chi_+: V_G \rightarrow \mathfrak{C}_+$. 
\end{defn}
The abelianization map $\alpha_G$ factors through the adjoint quotient of $V_G$, making the following diagram commute, i.e. 
 \begin{center}
\begin{tikzpicture}[node distance=2cm, auto]
\node (A) {$V_G$};
\node (C) [below of = A] {$\mathfrak{C}_+$};
\node (D) [right of = C] {$A_G$};
\draw[->, left] (A) to node {$\chi_+$} (C);
\draw[->] (C) to node{$\beta_G$} (D);  
\draw[->] (A) to node{$\alpha_G$} (D);  
\end{tikzpicture}
\end{center}
There is an analogue of the Chevalley isomorphism theorem for monoids. 
\begin{prop}[\cite{Bou12}, Theorems 1.5 and 1.6]
The restriction $k[V_G]^G \rightarrow k[\overline{T}_+]^W$ is an isomorphism of $k$-algebras. Moreover, the isomorphism
\begin{align*}
T_+ \sslash W &\xrightarrow{\simeq} \ \mathbb{G}^r_m \times \mathbb{A}^r, \\
(t, g) &\mapsto (\alpha_1(t), ..., \alpha_r(t), \Tr(\rho^+_1(t, g)), ..., \Tr(\rho^+_r(t, g))).
\end{align*}
extends to an isomorphism $\mathfrak{C}_+ = A_G \times (G\sslash \Ad(G)) \cong A_G \times \mathbb{A}^r \cong \mathbb{A}^{2r}$. 
The canonical quotient $\overline{T}_+ \rightarrow  \mathfrak{C}_+$ is finite, flat and generically Galois \'{e}tale with Galois group $W$.
\end{prop}
Under the adjoint action of $G$ on $V_G$, we define the following:
\begin{defn}
 An element $\gamma \in V_G$ is \begin{itemize}
 \item \emph{regular} if the dimension of its stabilizer is at its minimum, i.e. $\dim G_\gamma = r$,
\item \emph{semisimple} if $\gamma \in \overline{T}_+$,
\item \emph{nilpotent} if $\gamma^k = 0$ for some $k \in \mathbb{N}.$
 \end{itemize}  
Let $V_G^{\reg}$ (resp. $V_G^{\rs}$, resp. $\mathcal{N}$) be the the set of regular (resp. regular semisimple, resp. nilpotent) elements of $V_G$.
\end{defn}
The adjoint quotient $\chi_+$ admits multiple quasi-sections that are indexed by \emph{Coxeter elements} of the Weyl group, i.e. those that can be expressed as a product of simple reflections; each occurring exactly once. 
\begin{defn}
For each Coxeter element $w$, let $\varepsilon^w$ be the Steinberg quasi-section of the adjoint quotient $\chi: G \rightarrow T \sslash W$ of $G$. The \emph{extended Steinberg quasi-section} $\varepsilon^w_+$ is given by
\begin{align*} 
\varepsilon^w_+: \mathfrak{C}_+ \cong A_G \times T\sslash W &\rightarrow V_G \\
(a, b) &\mapsto s_G(a)\varepsilon^w(b).
\end{align*}
\end{defn}
As a quasi-section to $\chi_+$, the composition $\varepsilon^w_+ \circ \chi_+$ is an automorphism. The extended Steinberg base $\mathfrak{C}_+$ has two important divisors. Recall that there is a discriminant function on $T$ given by 
$D(t) = \prod_{\alpha \in \Phi} (1- \alpha(t))$ that descends to a regular function on the Chevalley quotient $T\sslash W$. Similarly, we can define a discriminant function on $T_+$ given by 
$$D_+(t_1, t_2) = 2\rho(t_1) D(t_2),$$
where $\rho(t)$ is the half sum of positive roots. 
\begin{defn}
The discriminant function $D_+$ on $T_+$ extends to the \emph{extended discriminant function} on  $\overline{T}_+$, which descends to a regular function on $\mathfrak{C}_+$. Its zero locus defines an effective divisor $\mathfrak{D}_+ \subseteq \mathfrak{C}_+$ called the \emph{discriminant divisor}. Denote its complement in $\mathfrak{C}_+$ by $\mathfrak{C}^{\rs}_+$. 
\end{defn}
The preimage of $\mathfrak{C}^{\rs}_+$ under $\chi_+$ is $V_G^{\rs}.$ 
\begin{defn}
The \emph{numerical boundary divisor} $\mathfrak{B}_+$ is the divisor on $\mathfrak{C}_+$ given by the complement of $\mathfrak{C}^\times_+ := A_G^\times \times (G/\Ad(G)) = T_{\ad} \times (G/\Ad(G))$ in $\mathfrak{C}_+.$ 
\end{defn}
This divisor is cut out by the product of $r$ coordinate functions and hence is a principal Weil divisor. 
\begin{prop}[\cite{Chi22}, Lemma 2.4.2] The discriminant divisor $\mathfrak{D}_+$ intersects the boundary divisor $\mathfrak{B}_+$ properly, i.e. the codimension of $\mathfrak{D}^\times_+ = \mathfrak{D}_+ \cap \mathfrak{B}_+$ in $\mathfrak{C}_+$ is at least 2.  
\end{prop}

 \subsection{Regular centralizers} 
Recall that $G$ acts on $V_G$ by the adjoint action. Define the centralizer group scheme $I$ over $V_G$ as 
$$I = \{(g, \gamma) \in G \times V_G: \Ad_g(\gamma) = \gamma\}.$$
It is smooth and commutative over the regular locus $V^{\reg}_G.$
\begin{prop}[\cite{Chi22}, Lemma 2.3.1] \label{jscheme}
There exists a unique smooth commutative group scheme $J$ over $\mathfrak{C}_+$ such that we have a $G$-equivariant isomorphism $(\chi_+^{\reg})^*J \cong I|_{V^{\reg}_G}.$ This isomorphism extends uniquely to a homomorphism $\chi_+^*J \rightarrow I.$
\end{prop}
Unlike the Lie algebra case, $[V^{\reg}_G/G]$ is not a single $BJ$-gerbe but it is a finite union of them.
\begin{prop}[\cite{Chi22}, Proposition 2.4.3]
The classifying stack $BJ$ acts on $[V_G/G].$ For each $w \in \Cox(W, S),$ the morphism $[\chi_+]: [V^w_G/G] \rightarrow \mathfrak{C}_+$ induced by $\chi_+$ is a $BJ$-gerbe that is neutralized by $\varepsilon^w_+$.
\end{prop} 
\subsection{Vinberg Borel submonoid} 
For a Borel subgroup $B$ of $G$, we call the closure $V_B$ of $B_+$ in $V_G$ the \emph{Borel submonoid of $V_G$}. It is solvable with unit group $B_+.$ Thanks to the following lemma, many properties of Borel subgroups hold for Borel submonoids.
\begin{lem}[\cite{Put82}, Lemma 1.2] \label{putcha} Let $M$ be a reductive monoid with unit group $G$ and subset $S$.  
If $a, b \in G$, then $\overline{aSb} = a\overline{S}b$.
\end{lem}
\begin{proof}
Since $a, b \in G$, then $a\overline{S}b$ is closed. Since $aSb \subseteq a\overline{S}b$, $\overline{aSb} \subseteq a\overline{S}b$. 
Similarly, $S \subseteq a^{-1}(\overline{aSb})b^{-1}$. Hence, $\overline{S} \subseteq a^{-1}(\overline{aSb})b^{-1}$ and $a\overline{S}b \subseteq \overline{aSb}$. 
Both inclusions imply equality. 
\end{proof}
\begin{prop} \label{borelsub} Let $B$ and $B'$ be Borel subgroups of $G$. 
\begin{enumerate}[(i)]
\item If $B' = gBg^{-1}$ for some $g \in G$, then $V_{B'}= V_{gBg^{-1}} = gV_Bg^{-1}$. Also, $N_G(V_B) = V_B.$
\item Let $\mathfrak{B}(V_G)$ be the set of Borel submonoids of $V_G$. Then, the following map is a bijection.
$$\varphi: G/B \rightarrow  \mathfrak{B}(V_G), \quad gB \mapsto gV_Bg^{-1}.$$
\end{enumerate}
\end{prop}
\begin{proof}
(i) follows from Lemma \ref{putcha} because $V_{B'} = V_{gBg} = gV_{B}g^{-1}$. Also, $N_G(V_B) \subseteq N_G(B_+)$ since $B_+ \subseteq V_B$. Conversely, if $g \in N_G(B_+)$, then $gV_Bg^{-1} = V_{gBg^{-1}} = V_B$. Hence, $g \in N_G(V_B)$ which proves equality. 
For (ii), the map $\varphi$ is well-defined because if $gB = hB$, then $h^{-1}gB = B$ and 
$$\phi(h^{-1}gB) = h^{-1}gV_Bg^{-1}h = V_B, \quad \Longrightarrow \quad \phi(gB) = gV_Bg^{-1} = hV_B h^{-1} = \phi(hB).$$
This map is surjective because any two Borel submonoids are conjugated to each other. It is also injective because if $gV_Bg^{-1} = hV_Bh^{-1}$, then $gh^{-1} \in N_G(V_B) = \overline{B}$. Thus, $gh^{-1}V_B = V_B$ and $gV_B = hV_B$.
\end{proof}

\subsection{Grothendieck simultaneous resolution} 
There is a closed subscheme of $V_G \times G/B$ given by 
$\widetilde{V}_G = \{(x, B) : V_G \times G/B: x \in V_B \}.$ Let $B$ act on $G \times V_B$ by 
$$b \cdot (g, x) = (gb^{-1}, bxb^{-1}),$$
and let $G \times_B V_B$ be the orbit space. Then, there is a $G$-equivariant isomorphism 
\begin{equation} \label{isom} G \times_B V_B \xrightarrow{\cong} \widetilde{V}_G, \quad (g, x) \mapsto (gxg^{-1}, gB/B).
\end{equation}
With this in mind, we can consider the commutative diagram 
\begin{equation} \label{CartDiag} 
\begin{tikzpicture}[node distance=2cm, auto]
\node (A) {$\widetilde{V}_G$};
\node (B) [right of = A] {$\overline{T}_+$};
\node (C) [below of = A] {$V_{G}$};
\node (D) [below of = B] {$\mathfrak{C}_+$};
\draw[->, above] (A) to node {$p$}(B);
\draw[->, left] (A) to node {$\pi_G$} (C);
\draw[->] (B) to node {$q$} (D);
\draw[->] (C) to node {$\chi_+$} (D);
\end{tikzpicture}
\end{equation}
Here, $q$ is finite flat, generically Galois and \'{e}tale with Galois group $W$ while $\pi_G$ is the projection onto the first factor. Since $V_B\sslash \Ad(T) \cong \overline{T}_+$ by \cite[Proposition 2.4.19]{Wang23}, the map $p$ is induced by the GIT quotient, i.e. 
$$\widetilde{V}_G \cong G \times_B V_B \rightarrow \overline{T}_+ \cong V_B\sslash \Ad(T).$$
This diagram is Cartesian over the loci $V^{\rs}_G$ and $G_+^{\reg}$. However unlike the Lie algebra case, it is not Cartesian over the regular locus $V^{\reg}_G.$  Since $G/B$ is compact, $\pi_G$ is proper and each fiber $\pi_G^{-1}(x)$ is closed.  

The Grothendieck simultaneous resolution induces the following commutative diagram of quotient stacks
\begin{equation} \label{paraCart}
\begin{tikzpicture}[node distance=2.5cm, auto]
\node (A) {$[\widetilde{V}_G/G] = [V_B/B]$};
\node (B) [right of = A] {$\overline{T}_+$};
\node (C) [below of = A] {$[V_G/G]$};
\node (D) [below of = B] {$\mathfrak{C}_+$};
\draw[->, below] (A) to node {}(B);
\draw[->, left] (A) to node {} (C);
\draw[->] (B) to node {} (D);
\draw[->] (C) to node {} (D);
\end{tikzpicture}
\end{equation}
\subsection{Restricted monoids}
\subsubsection{Cartan decomposition}
For each $\lambda_{\ad} \in X_*(T_{\ad})^+$, define the affine scheme $V^{\lambda_{\ad}}_G$ over $\Spec \mathcal{O}$ by the Cartesian diagram,
\begin{equation} \label{arc}
\begin{tikzpicture}[every text node part/.style={align=center}, node distance=2.5
cm]
\node (A) at (0,0)  {$V^{\lambda_{\ad}}_G$};
\node (B) [right of = A] {$V_G \times T_{\ad}$};
\node (C) [below of = A] {$\Spec \mathcal{O}$};
\node (D) [right of = C] {$A_G$};
\draw[->, below] (A) to node {\hspace{-1.6cm} $\mathlarger{\mathlarger{\mathlarger{\lrcorner}}}$}(B);
\draw[->] (A) to node {} (C);
\draw[->, right] (B) to node {$(g, z) \mapsto \alpha_G(g)z$} (D);
\draw[->, above] (C) to node{$t^{-w_0(\lambda_{\ad})}$} (D);  
\end{tikzpicture}
\end{equation} Here, the bottom arrow is defined by 
$$\left(t^{\langle -w_0(\lambda_{\ad}),\alpha_1 \rangle}, \cdots, t^{\langle -w_0(\lambda_{\ad}), \alpha_r \rangle}\right) \in A_G(\mathcal{O}).$$
Replacing $V_G$ with $V^0_G$ in this diagram gives rise to $V^{\lambda_{\ad}, 0}_G.$ 

The schemes $V^{\lambda_{\ad}}_G(\mathcal{O})$ and $V^{\lambda_{\ad}, 0}_G(\mathcal{O})$ admit a decomposition that is similar to the Cartan decomposition of loop groups, namely
\begin{align*}
V^{\lambda_{\ad}}_G(\mathcal{O}) &=  \bigcup_{\substack{\mu \in X_*(T_{\ad})^+ \\ \mu \leq \lambda}} G_+(\mathcal{O})t^{(-w_0(\lambda_{\ad}), \mu)} G_+(\mathcal{O}), \\
V^{\lambda_{\ad},0}_G(\mathcal{O}) & = \bigcup_{w \in W} G_+(\mathcal{O})t^{(-w_0(\lambda_{\ad}), w\lambda_{\ad}w^{-1})} G_+(\mathcal{O}) \\
&=  \bigcup_{w \in W_+} G_+(\mathcal{O})wt^{(-w_0(\lambda_{\ad}) \lambda_{\ad})} w^{-1}G_+(\mathcal{O}) \\ 
&= G_+(\mathcal{O})t^{(-w_0(\lambda_{\ad}), \lambda_{\ad})} G_+(\mathcal{O}).
\end{align*}
Here, $\widetilde{W}_+ = W \ltimes X_*(T_+)$ is the extended affine Weyl group of $G_+$ in which $W$ acts on the second factor of $X_*(T_+)$ only. For each element $g \in V^{\lambda_{\ad}}_G(\mathcal{O}),$ the matrices $\rho^+_{i} (g) \in \End(V_{\omega_i})$ have entries in $\mathcal{O}$ for all $1 \leq i \leq r.$ If $g$ belongs to $V^{\lambda_{\ad},0}_G(\mathcal{O}),$ then the matrices have non-zero reduction mod $t$.
\begin{exmp}
Let $G = \SL_2(\mathbb{C}).$ By Example \ref{sl2}, its associated Vinberg monoid is $V_G = \End(\mathbb{C}^2)$ with unit group $G_+ = \GL_2(\mathbb{C})$. Its abelianization is $A_G = \mathbb{A}^1$ with unit group $\mathbb{G}_m.$ For a dominant coweight $\lambda_{\ad} = n \in X_*(T_{\ad})^+ \cong \mathbb{Z}_{\geq 0}$ of $T_{\ad}$, the scheme $V^{\lambda_{\ad}}_G = V^n_G$ is given by the following Cartesian diagram:
\begin{center}
\begin{tikzpicture}[every text node part/.style={align=center}, node distance=2.5cm]
\node (A) at (0,0)  {$V^n_G$};
\node (B) [right of = A] {$\End(\mathbb{C}^2) \times \mathbb{G}_m$};
\node (C) [below of = A] {$\Spec (\mathbb{C}[\![t]\!])$};
\node (D) [right of = C] {$\mathbb{A}^1$};
\draw[->, below] (A) to node {\hspace{-1cm} $\mathlarger{\mathlarger{\mathlarger{\lrcorner}}}$}(B);
\draw[->] (A) to node {} (C);
\draw[->, right] (B) to node {$(g, z) \mapsto \det(g)z$} (D);
\draw[->, above] (C) to node{$t^{n}$} (D);  
\end{tikzpicture}
\end{center}
Here, $V^n_G$ contains $2 \times 2$ matrices with entries in $\mathcal{O} = \mathbb{C}[\![t]\!]$ with determinant $t^n$. Its non-degenerate locus $V^{n, 0}_G$ contains matrices in $V^n_G$ that are non-zero and singular reduction mod $t$. For instance, $$\begin{pmatrix}
    1/(1-t^n) & 1 \\
    1-t^n & 1-t^n
\end{pmatrix} = \begin{pmatrix}
    \sum^\infty_{k=0} t^{nk} & 1 \\
    1-t^n & 1-t^n
\end{pmatrix}$$ has determinant $t^n$. Its reduction mod $t$, which is given by $\begin{pmatrix}
1 & 1\\
1 & 1
\end{pmatrix}$, has determinant 0.
\end{exmp}
Let $(\cdot)_{\ad}: G_+ \rightarrow G_{\ad}$ be the natural projection map onto $G_{\ad}.$ 
\begin{lem}[\cite{Chi22}, Lemma 2.5.1] \label{adjointcrit}
An element $g_+ \in G_+(F)$ belongs to $V^{\lambda_{\ad}}_G(\mathcal{O})$ (resp. $V^{\lambda_{\ad},0}_G(\mathcal{O})$) if and only if $\alpha_G(\gamma_+) \in t^{-w_0(\lambda_{\ad})}T_{\ad}(\mathcal{O})$ and 
$$(g_+)_{\ad} \in \overline{G_{\ad}(\mathcal{O})t^{\lambda_{\ad}}G_{\ad}(\mathcal{O})}, \quad (\text{resp. } (g_+)_{\ad} \in G_{\ad}(\mathcal{O})t^{\lambda_{\ad}}G_{\ad}(\mathcal{O})).$$
\end{lem}

\subsubsection{Iwahori submonoid}
We will now define the analogue of the Iwahori subgroup for monoids. Consider the restriction of the diagram in \eqref{arc} to $V_B$, i.e. 
\begin{equation}
\begin{tikzpicture}[every text node part/.style={align=center}, node distance=2.5cm]
\node (A) at (0,0)  {$V^{\lambda_{\ad}}_B$};
\node (B) [right of = A] {$V_B \times T^{\ad}$};
\node (C) [below of = A] {$\Spec \mathcal{O}$};
\node (D) [right of = C] {$A_G$};
\draw[->, below] (A) to node {\hspace{-1.6cm} $\mathlarger{\mathlarger{\mathlarger{\lrcorner}}}$}(B);
\draw[->] (A) to node {} (C);
\draw[->, right] (B) to node {$(g, z) \mapsto \alpha_G(g)z$} (D);
\draw[->, above] (C) to node{$t^{-w_0(\lambda_{\ad})}$} (D);  
\end{tikzpicture}
\end{equation}

Replacing $V_B$ with $V^0_B$ gives rise to $V^{\lambda_{\ad}, 0}_B.$ Note that $$V^{\lambda_{\ad}}_B(\mathcal{O}) = V^{\lambda_{\ad}}_G(\mathcal{O}) \cap V_B(\mathcal{O}), \quad  V^{\lambda_{\ad}, 0}_B(\mathcal{O}) = V^{\lambda_{\ad}, 0}_G(\mathcal{O}) \cap V_B(\mathcal{O}).$$ Denote the Iwahori subgroup of $G_+(\mathcal{O})$ by $I_+$. The loop group $G_+(F)$ admits a Bruhat decomposition $G_+(F) \cong I_+\setminus \widetilde{W}_+ /I_+$.

Next, let $I^{\lambda_{\ad}}$ (resp. $I^{\lambda_{\ad}, 0}$) be the pre-image of $V^{\lambda_{\ad}}_B(k)$  (resp. $V^{\lambda_{\ad},0}_B(k)$) under the reduction-mod-$t$ map $$\eva_0: V^{\lambda_{\ad}}_G(\mathcal{O}) \rightarrow V^{\lambda_{\ad}}_G(k), \quad t \mapsto 0.$$
We will now work out its Bruhat decomposition. Recall from \cite[Lemma 2.5.1]{Chi22} that for $g$ to belong to $V^{\lambda_{\ad}, 0}_G(\mathcal{O}),$ the matrices $\rho_i^+(g) \in \End(V_{\omega_i})$ must have entries in $\mathcal{O}$ and non-zero reduction mod $t$ for all $1 \leq i \leq r.$. This amounts to $V^{\lambda_{\ad}, 0}_G(\mathcal{O})$ having the following Cartan decomposition
$$V^{\lambda_{\ad}, 0}_G(\mathcal{O})  = \bigcup_{w \in W} G_+(\mathcal{O})t^{(-w_0(\lambda_{\ad}), w\lambda_{\ad}w^{-1})} G_+(\mathcal{O}).$$
Now, for $g$ to belong to $I^{\lambda_{\ad}, 0}$, it must belong to $V^{\lambda_{\ad},0}_G$ and additionally, the matrices $\rho_i^+(g)$ must be upper-triangular reduction mod $t$. In particular, $I^{\lambda_{\ad},0}$ is stable under the action of $I_+ \times I_+.$ Hence, $I^{\lambda_{\ad},0}$ is a union of $I_+$-double cosets in $G_+(F)$ indexed by $\widetilde{W}_+$ and by the Bruhat decomposition, 
$$I^{\lambda_{\ad}, 0}  = \bigcup_{w \in W} I_+ t^{(-w_0(\lambda_{\ad}), w\lambda_{\ad}w^{-1})} I_+.$$
Since $V^0_B$ is open dense in $V_B$, the (Zariski) closure of $I^{\lambda_{\ad}, 0}$ in $G_+(F)$ is $I^{\lambda_{\ad}}$ by construction. Since $I^{\lambda_{\ad}}$ is stable under the action of $I_+ \times I_+$, it is also a union of $I_+$-double cosets and hence, it is given by the Bruhat closure of $I^{\lambda_{\ad}, 0}$, i.e. 
\begin{equation} \label{Bruhat}
I^{\lambda_{\ad}} = \bigcup_{\mu \in W(\lambda_{\ad})} \bigcup_{\substack{(t^\mu, w) \in \widetilde{W}_+ \\ (t^\mu, w) \leq t^{(-w_0(\lambda_{\ad}), \mu)}}}I_+wI_+. \end{equation} 
We can write the indexing set of the union in \eqref{Bruhat} in simpler terms using the notion of the admissible set. 
\begin{defn}
Let $\lambda \in X_*(T)^+$ be a dominant coweight. Then, the \emph{$\lambda$-admissible set} is the set $\Adm(\lambda) = \{w \in \widetilde{W}: w \leq t^{x( \lambda)} \text{ for some }x \in W\}$. 
\end{defn}
It is a finite subset of $\widetilde{W}$ whose maximal elements are $t^\mu$ for $\mu \in W(\lambda)$.
\begin{exmp}[$G = \GL_2(\mathbb{C})$]
When $G = \GL_2(\mathbb{C}),$ its Weyl group is $W = S_2$, its cocharacter lattice is $X_*(T)= \mathbb{Z}^2$ and so, its extended affine Weyl group is $\widetilde{W} = S_2 \ltimes \mathbb{Z}^2.$ Choose a dominant coweight $\lambda = (1, 0).$ Then, 
$$\Adm(\lambda) = \{t^{(1,0)}, t^{(0, 1)}, st^{(1, 0)}\} \subseteq \widetilde{W} = S_2 \ltimes \mathbb{Z}^2,$$
where $s \in S_2$ is the non-trivial element in $S^2$ and $t^{(1, 0)}$ and $t^{(0, 1)}$ are translation elements. 
\end{exmp}
Other examples of the admissible set for $G = \GL_3(k)$ and $G = \mathrm{GSp}_4(k)$ can be found in \cite[Figures 1 and 2]{Hai04}. By \cite[\S 2.3.5]{Wang23}, the coweight lattice $X_*(T_+)$ of $T_+$ has the following form:  
$$X_*(T_+) = \{(\lambda, \mu) \in X_*(T_{\ad}) \times X_*(T_{\ad}): \lambda + \mu \in X_*(T)\}.$$
In our case, $(-w_0(\lambda_{\ad}), \lambda_{\ad})$ is a dominant coweight of $T_+$ because $\lambda_{\ad}$ is a dominant coweight. Since $W$ acts trivially on the first factor $X_*(T_{\ad})$, the first factor of $X_*(T_+)$ inherits the trivial partial order. Hence, $\Adm(-w_0(\lambda_{\ad}), \lambda_{\ad})$ is given by 
\begin{align*}
\Adm(-w_0(\lambda_{\ad}), \lambda_{\ad}) = \{ (t^{-w_0(\lambda_{\ad})}, w) \in \widetilde{W}_+: (t^{-w_0(\lambda_{\ad})}, w) \leq t^{(-w_0(\lambda_{\ad}), x(\lambda_{\ad}))} \text{ for some }x \in W\}.
\end{align*}
This allows us to rewrite the above as 
\begin{align*}
I^{\lambda_{\ad}} &= \bigcup_{\Adm(-w_0(\lambda_{\ad}), \lambda_{\ad})}I_+wI_+. \end{align*}
Denote the Iwahori subgroup of $G_{\ad}(\mathcal{O})$ by $I_{\ad}$. From the exposition above, we obtain an Iwahori version of Lemma \ref{adjointcrit}.
\begin{lem}
An element $g_+ \in G_+(F)$ belongs to $I^{\lambda_{\ad}}$ (resp. $I^{\lambda_{\ad},0}$) iff $\alpha_G(\gamma)\in t^{-w_0(\lambda_{\ad})}T_{\ad}(\mathcal{O})$ and 
$$(g_+)_{\ad} \in \bigcup_{w \in \Adm(\lambda_{\ad})} I_{\ad}wI_{\ad}, \quad \left( \text{resp. } (g_+)_{\ad} \in \bigcup_{\mu\in W(\lambda_{\ad})} I_{\ad} t^\mu I_{\ad}\right).$$
\end{lem}
\subsubsection{Grothendieck simultaneous resolution}
In future sections, we will be globalizing these arc space constructions to a curve. For this, we require a Grothendieck simultaneous resolution that incorporates cocharacter-valued poles. For each $\lambda_{\ad} \in X_*(T_{\ad})_+,$ define the affine scheme $\widetilde{V}^{\lambda_{\ad}}_G$ over $\Spec \mathcal{O}$ by the following Cartesian diagram.
\begin{align*} 
\begin{tikzpicture}
\node (A) at (0, 0) {$V^{\lambda_{\ad}}_G$};
\node (B) at (3, 0) {$V_G \times T_{\ad}$};
\node (C) at (0, -2) {$\Spec \mathcal{O}$};
\node (D) at (3, -2){$A_G$};
\node (E) at (0, 2) {$\widetilde{V}^{\lambda_{\ad}}_{G}$};
\node (F) at (3, 2) {$\widetilde{V}_G \times T^{\ad}$};
\draw[->, below] (A) to node {\hspace{-2cm} $\mathlarger{\mathlarger{\mathlarger{\lrcorner}}}$}(B);
\draw[->, right] (B) to node {$(g, z) \mapsto \alpha_G(g)z$} (D);
\draw[->, above] (C) to node {$t^{-w_0(\lambda_{\ad})}$} (D);
\draw[->] (A) to node {} (C);
\draw[->] (E) to node {} (A);
\draw[->, right] (F) to node {$\pi_G \times \id$} (B);
\draw[->, below] (E) to node {\hspace{-2cm} $\mathlarger{\mathlarger{\mathlarger{\lrcorner}}}$}(F);
\end{tikzpicture}
\end{align*}
Here, $\pi_G$ is the Grothendieck simultaneous resolution of $V_G$ from Equation \eqref{CartDiag} and $\eva_0$ is the reduction-mod-$t$ map. Fix a Borel submonoid $V_B$. More concretely, $\widetilde{V}^{\lambda_{\ad}}_G(\mathcal{O})$ is defined as 
\begin{align*}
\widetilde{V}^{\lambda_{\ad}}_G(\mathcal{O}) &= \{(g, B) \in V^{\lambda_{\ad}}_G(\mathcal{O}) \times G/B: g \in \eva_0^{-1}(V^{\lambda_{\ad}}_B(k)) \} \\
 &= \{(g, B) \in V^{\lambda_{\ad}}_G(\mathcal{O}) \times G/B: g \in I^{\lambda_{\ad}} \}.
\end{align*}
We can define $\widetilde{V}^{\lambda_{\ad}, 0}_G$ similarly by replacing $V_G$ with $V_G^\circ.$ Recall that there is an isomorphism $G \times_B V_B \cong \widetilde{V}_G$ given by $(g, x) \mapsto (gxg^{-1}, gB/B).$ By essentially the same map, we have the isomorphism
$\widetilde{V}^{\lambda_{\ad}}_G(\mathcal{O}) \cong G \times_B I^{\lambda_{\ad}}.$
\section{Multiplicative affine Springer fibers}
\subsection{Definitions}
Multiplicative affine Springer fibers are group-theoretic versions of affine Springer fibers. They were first introduced by \cite{FN11} to give a geometric interpretation of orbital integrals of spherical Hecke functions. Let $\gamma \in G^{\rs}(F)$ be a regular semisimple element and $\lambda \in X_*(T)^+$ be a dominant coweight. Then, \emph{multiplicative affine Springer fibers} or \emph{Kottwitz-Viehmann varieties} are given by the sets
\begin{align} \label{masfo}
X^\lambda_\gamma &= \{g \in \Gr_G(k): \Ad_{g^{-1}}(\gamma) \in G(\mathcal{O}) t^\lambda G(\mathcal{O})\}, \\ \label{masfc}
X^{\leq \lambda}_\gamma &= \{g \in \Gr_G(k): \Ad_{g^{-1}}(\gamma) \in \overline{G(\mathcal{O}) t^\lambda G(\mathcal{O})}\}. 
\end{align} 
These are non-Frobenius-twisted analogs of affine Deligne Lusztig varieties in the affine Grassmannian.
\subsection{Non-emptiness}
The non-emptiness pattern was first studied by \cite[\S 2]{KV10} and it relies on the notions of the Newton and Kottwitz homomorphisms. Let $\mathbb{D} = \Spec(F[\mathbb{Q}])$ be the diagonalizable group over $F$ with character group $\mathbb{Q}.$ Each $\gamma \in G^{\rs}(F)$ has an associated Newton homomorphism $\nu_\gamma: \mathbb{D} \rightarrow G$ that is described as follows. For any $G$-representation $V$, the element $\gamma \in G(F)$ acts on $V$ as an automorphism. As a result, $V$ admits a slope decomposition with respect to $\gamma$, i.e.  
$$V = \bigoplus_{a \in \mathbb{Q}} V_a,$$
where $V_a$ is the direct sum of all generalized eigenspaces associated to the eigenvalue with valuation $a$. Given an $F$-algebra $R$ and $d \in \mathbb{D}(R) = \Hom(\mathbb{Q}, R^\times)$, the image $\nu_\gamma(d) \in G(R)$ is characterised as the automorphism that acts on $V \otimes R$ as multiplication-by-$d(a)$ on each summand $V_a \otimes R.$ Since $\gamma$ is regular semisimple, $\nu_\gamma \in \Hom(\mathbb{D}, G) = X_*(T).$

Now, we define the Kottwitz homomorphism. Let $p_G: X_*(T) \rightarrow \pi_1(G) = X_*(T)/\Phi^\vee_\mathbb{Z}$ be the natural projection. The \emph{Kottwitz homomorphism} $\kappa_G$ is the composition
$$\kappa_G: G(F) \xrightarrow{r} X_*(T) \xrightarrow{p_G} \pi_1(G)$$
where $r(\gamma) = \mu \in X_*(T)$ such that $\gamma \in G(\mathcal{O}) t^\mu G(\mathcal{O}).$ There are natural surjections 
\begin{align*}
G_+ \rightarrow G_{\ad}, \quad G  \rightarrow G_{\ad}, \quad X_*(T) \rightarrow X_*(T_{\ad}).
\end{align*} 
We will denote both of these maps as $(\cdot)_{\ad}.$
\begin{prop}[\cite{Chi22}, Lemma 3.1.5] \label{newelem}
Suppose that $\kappa_G(\gamma) = p_G(\lambda)$. There exists an element $\gamma_\lambda \in G_+(F)$ such that 
\begin{enumerate}[(i)]
    \item $(\gamma_\lambda)_{\ad} = \gamma_{\ad}$ in $G_{\ad}(F)$. 
    \item $\alpha_G(\gamma_\lambda) =  t^{-w_0(\lambda_{\ad})} \in T_{\ad}(\mathcal{O}) \cap A_G(F).$
\end{enumerate}
This element is uniquely defined up to the multiplication by an element of $Z(F).$
\end{prop}
We have the following non-emptiness criterion. 
\begin{lem}[\cite{Chi22}, Proposition 3.1.6] \label{nonemp}
The following are equiavalent:
\begin{enumerate}[(i)]
\item $X^\lambda_\gamma$ is non-empty.
\item $X^{\leq \lambda}_\gamma$ is non-empty.
\item $\kappa_G(\gamma) = p_G(\lambda)$ and $\nu_\gamma \leq_{\mathbb{Q}} \lambda$, i.e. $\lambda - \nu_\gamma$ is a linear combination of simple coroots with non-negative coefficients in $\mathbb{Q}.$
\item $\kappa_G(\gamma) = p_G(\lambda)$ and $\chi_+(\gamma_\lambda) \in \mathfrak{C}_+(\mathcal{O}).$
\end{enumerate}
\end{lem}
Condition (iii) is essentially the non-emptiness pattern of affine Deligne Lusztig varieties in the affine Grassmanian except in that case, $\kappa_G$ maps from the set of Frobenius-twisted conjugacy classes.
\subsection{Ind-scheme structure}
Like affine Springer fibers, multiplicative affine Springer fibers admit an ind-scheme structure, which provides a geometric interpetation. With $v \in X$, let $\mathcal{O}_v$ be the completed local ring of $X$ at $v$ and $F_v$ be its fraction field. Then, $X_v = \Spec(\mathcal{O}_v) \cong k[[t]]$ and $X^\bullet_v = \Spec(F_v) = k((t))$ is the formal disc and formal punctured disc at $v$ respectively. Let $\gamma_+ \in G_+(F_v)$ with image $\chi_+(\gamma_+) = a \in \mathfrak{C}_+(\mathcal{O}_v)$. We denote $X_v \widehat{\times} S$ as the $v$-adic completion of $X_v \times S$. 

\begin{defn} \label{MASFunctor}
Let $M_v(\gamma_+)$ be the functor that maps a scheme $S$ to the set of isomorphism classes of pairs $(h, \iota)$, where $h$ fits into the commutative diagram,
\begin{center} 
\begin{tikzpicture}[node distance=2cm]
\node (A) {$X_v \widehat{\times} S$};
\node (B) [right of = A] {$[V_G/G]$};
\node (C) [below of = A] {$X_v$};
\node (D) [right of = C] {$\mathfrak{C}_+$};
\draw[->, above] (A) to node {$h$} (B);
\draw[->, right] (B) to node {$\chi_+$} (D);
\draw[->,above] (C) to node {$a$} (D);
\draw[->] (A) to node {} (C);
\end{tikzpicture}
\end{center}
and $\iota$ is an isomorphism between $h|_{X^\bullet_{v} \widehat{\times} S}$ and the induced map 
$X_{v}^\bullet \widehat{\times} S \xrightarrow{\gamma_+} V_G \rightarrow [V_G/G].$ By replacing $V_G$ with $V_G^\circ$ and $V_G^{\reg}$, we obtain subfunctors $M_v(\gamma_+)^\circ$ and $M_v(\gamma_+)^{\reg}$. 
Since the isomorphism classes of $M_v(\gamma_+)$ and $M_v(\gamma_+)^\circ$ only depend on $a = \chi_+(\gamma_+)$, we will also denote these as $M_v(a)$ and $M_v(a)^\circ.$
\end{defn}
For $\gamma \in G^{\rs}(F_v)$ and $\lambda \in X_*(T)^+,$ suppose that $X^{\lambda}_{\gamma}$ is non-empty. By Proposition \ref{nonemp}, there exists an element $\gamma_\lambda \in G_+(F_v)$ such that 
$$X^{\lambda}_{\gamma} \cong  M_v(\gamma_\lambda)^\circ, \quad X^{\leq \lambda}_{\gamma} \cong M_v(\gamma_\lambda).$$
This leads to a moduli description of multiplicative affine Springer fibers. The functor $M_v(\gamma_+)$ maps a scheme $S$ to the set of isomorphism classes of triples $(E, \varphi, \sigma)$ consisting of:
\begin{itemize}
    \item a $G$-bundle $E$ on the $X_v \widehat{\times} S$ 
    \item $\varphi \in H^0(X_v, E \times_G V^{\lambda_{\ad},0}_G)$, 
    \item a trivialization $\sigma: (E^0, \gamma_\lambda) \cong (E, \varphi)|_{X^{\bullet}_v \widehat{\times} S}$ over the punctured disc $X^{\bullet}_v \widehat{\times} S$, where $E^0$ is the trivial bundle.
\end{itemize}
The pair $(E, \varphi)$ is called a \emph{multiplicative Higgs bundle} on $X_v$ and $\varphi$ is called a \emph{multiplicative Higgs field}.
\begin{prop} \label{connectiontoVin}
If $\kappa_G(\gamma) = p_G(\gamma)$, then
equations \eqref{masfc} and \eqref{masfo} can be rewritten as
\begin{align*}
    X^\lambda_\gamma &\cong \{g \in \Gr_G: \Ad_{g^{-1}}(\gamma_\lambda) \in V^{\lambda_{\ad},0}_G(\mathcal{O})\}, \\
     X^{\leq \lambda}_\gamma &\cong\{g \in \Gr_G: \Ad_{g^{-1}}(\gamma_\lambda) \in V^{\lambda_{\ad}}_G(\mathcal{O})\}.
\end{align*}
\end{prop}
\begin{proof}
We will only prove this for $X^{\leq \lambda}_\gamma$. The proof for $X^{\lambda}_\gamma$ follows suit. Suppose $g \in X^{\leq \lambda}_\gamma.$ Then, $g^{-1}\gamma_\lambda g \in G(\mathcal{O})t^\mu G(\mathcal{O})$ for some $\mu \in X_*(T)_{\ad}^+$ satisfying $\mu \leq \lambda_{\ad}.$ Furthermore, $\kappa_G(\gamma) = p_G(\mu) =  p_G(\lambda)$ by definition of $\kappa_G$ and $p_G.$ By Lemma \ref{newelem}, there exists $\gamma_\lambda \in G_+(F)$ such that 
$$(g^{-1}\gamma_\lambda g)_{\ad} \in G_{\ad}(\mathcal{O})t^{\lambda_{\ad}}G_{\ad}(\mathcal{O}), \quad \text{and} \quad \alpha_G(g^{-1}\gamma_\lambda g) = \alpha_G \in t^{-w_0(\lambda_{\ad})}T_{\ad}(\mathcal{O}).$$
By Lemma \ref{adjointcrit}, $g^{-1}\gamma_\lambda g \in V^{\lambda_{\ad}}_G(\mathcal{O}).$

Conversely, suppose that $g^{-1}\gamma_\lambda g \in V^{\lambda_{\ad}}_G(\mathcal{O})$. Then, $g^{-1}\gamma_\lambda g \in G_+(\mathcal{O})t^{(-w_0(\lambda_{\ad}), \mu)}G_+(\mathcal{O})$ for some $\mu \in X_*(T_{\ad})^+$ satisfying $\mu \leq \lambda_{\ad}$. This surjects onto $g^{-1}\gamma_{\ad}g \in G_{\ad}(\mathcal{O})t^{\mu}G_{\ad}(\mathcal{O})$. Choose a lift $\widetilde{\gamma} \in G(F)$ of $\gamma_{\ad} \in G_{\ad}(F)$ such that $\kappa_G(\widetilde{\gamma}) = \kappa_G(\gamma) = p_G(\lambda).$ Then, $g \in X^{\lambda}_{\widetilde{\gamma}}$. We need to show that $\widetilde{\gamma}$ is conjugate to $\gamma$ in $G(F).$ Consider the transporter of $\widetilde{\gamma}$ and $\gamma$, i.e. 
$$\mathrm{Tran}(\widetilde{\gamma}, \gamma) = \{g \in G: g\widetilde{\gamma}g^{-1} = \gamma\}.$$
It is a torsor under the torus $G_{\widetilde{\gamma}}$ over $F$, the function field of a smooth curve over an algebraically closed field. By Steinberg's theorem, this torsor is trivial and $\mathrm{Tran}(\widetilde{\gamma}, \gamma)$ contains an $F$-point. Since $\widetilde{\gamma}$ and $\gamma$ are conjugate, $g \in X^\lambda_\gamma$. \end{proof}
\subsection{Symmetries}
For some $v \in X$ and $a \in \mathfrak{C}_+(\mathcal{O}_v) \cap \mathfrak{C}_+^\times(F_v)^{\rs}$, we can define a commutative group scheme $J_a$ over $\Spec(\mathcal{O}_v)$ as the following pullback:
\begin{center} 
\begin{tikzpicture}[node distance=2cm]
\node (A) {$J_a$};
\node (B) [right of = A] {$J$};
\node (C) [below of = A] {$\Spec(\mathcal{O}_v)$};
\node (D) [right of = C] {$\mathfrak{C}_+$};
\draw[->, below] (A) to node {\hspace{-1.3cm} $\mathlarger{\mathlarger{\mathlarger{\lrcorner}}}$} (B);
\draw[->, right] (B) to node {} (D);
\draw[->,above] (C) to node {$a$} (D);
\draw[->] (A) to node {} (C);
\end{tikzpicture}
\end{center}
where $J$ is the commutative group scheme from \ref{jscheme}.
\begin{defn} The \emph{local Picard functor} is the affine Grassmannian $P_v(a) = \Gr_{J_a}$.
\end{defn}
It parametrizes $J_a$-torsors on the formal disc with a trivialization over the punctured disc. There is a natural action of $P_v(a)$ on $M_v(a)$ in the following way. An $S$-point of $M_v(a)$ is the tuple $(E, \phi, \sigma)$ consisting of a $G$-torsor over $X_v \widehat{\times} S,$ a $G$-equivariant map $\phi: E \rightarrow V_G$ and a trivialization $\sigma$ over $X_v^\bullet \widehat{\times} S.$ An $S$-point of $P_v(a)$ is the pair $(E_J, \sigma_J)$ consisting of a $J_a$-torsor $E_J$ over $X_v \widehat{\times} S$ and a trivialization $\sigma_J$ over $X_v^\bullet \widehat{\times} S.$ 

Since $J_a$ acts on the fibers of $\phi: E \rightarrow V_G$ via the homomorphism $\chi_+^*J \rightarrow I$ over $V_G$, we can produce a new bundle $E' := E \times^{{J_a}}_{\phi, V_G} E_J$. This comes with a $G$-equivariant map $\phi': E' \rightarrow V_G$. The trivialization $\sigma_J$ induces the isomorphism $E' \cong E$ over $X_v^\bullet \widehat{\times} S.$ We can define $\sigma'$ as the composition of this isomorphism with $\sigma.$ So, $(E_J, \sigma_J)$ acts on $(E, \phi, \sigma)$ to produce $(E', \phi', \sigma').$
\subsection{Dimension formula}
The dimension formula of multiplicative affine Springer fibers was the main result of \cite{Chi22}. He first computed the dimension of the regular locus. Since $X^{\lambda, \reg}_\gamma$ is not dense in $X^\lambda_\gamma$ and there might be irreducible components disjoint from $X^{\lambda, \reg}_\gamma$, we cannot readily conclude that $\dim X^{\lambda, \reg}_\gamma = \dim X^{\lambda}_\gamma$ like in the Lie algebra case. So instead, \cite{Chi22} used the multiplicative Hitchin fibration and the Product formula to prove the following results on the dimension and equidimensionality of the fibers. 
\begin{defn} Let $\gamma \in G(F)^{\rs}.$ Its \emph{discriminant valuation} is defined as 
$$d(\gamma) = \val \det( \id - \ad_\gamma: \mathfrak{g}(F)/\mathfrak{g}_\gamma(F) \rightarrow \mathfrak{g}(F)/\mathfrak{g}_\gamma(F)),$$
where $\mathfrak{g}_\gamma$ is the centralizer of $\gamma$ under $\ad_\gamma$. Meanwhile, the \emph{extended discriminant valuation} of $a \in \mathfrak{C}_+(\mathcal{O}) \cap \mathfrak{C}^\times_+(F)^{\rs}$ is 
$d_+(a) := \val(a^*\mathfrak{D}_+).$ If $a = \chi_+(\gamma_\lambda)$, then we will refer to $d_+(a)$ as 
$$d_+(a) = d_+(\gamma) :=  \langle 2 \rho, \lambda \rangle + d(\gamma).$$
\end{defn}
\begin{defn}
For $\gamma \in G^{\rs}(F),$ we define $c(\gamma) := \rk(G) - \rk_F(G_\gamma).$ Here, $\rk_F(G_\gamma)$ is the dimension of the maximal $F$-split subtorus of the centralizer $G_\gamma.$
An element $\gamma \in G(F)^{\rs}$ is \emph{unramified} if $c(\gamma) = 0.$
\end{defn}

\begin{thrm}[\cite{Chi22}, Theorem 1.2.1]
The $k$-schemes $X^\lambda_\gamma$ and $X^{\leq \lambda}_\gamma$ with $a = \chi_+(\gamma_\lambda)$ are locally of finite type and are equidimensional of dimension
\begin{align*} \dim X^\lambda_\gamma = \dim X^{\leq \lambda}_\gamma  = \langle \rho, \lambda \rangle + \frac{d(\gamma) - c(\gamma)}{2} = \frac{d_+(\gamma)-c(\gamma)}{2}.
\end{align*}
\end{thrm}
\begin{defn} \label{localdelta} We define the \emph{local $\delta$-invariant of $a$} as 
$\delta_v(a) = \dim P_v(a).$
\end{defn}
We will not explore this aspect here but the local $\delta$-invariant admits a geometric description via local Neron models and cameral covers \cite[Corollary 3.3.4]{Chi22}.
\section{Multiplicative Hitchin fibration}
\subsection{Boundary divisors}
Let $X$ be a smooth, projective, connected curve over $k$.
\begin{defn}
A \emph{boundary divisor} is a $X_*(T_{\ad})^+$-valued divisor $$\lambda = \sum_{x \in |X|} \lambda_x x,$$
where $\lambda_x$ is a non-zero dominant coweight in $X_*(T_{\ad})^+$ for finitely many $x \in X$.
\end{defn}
The mapping stack $\underline{\Hom}(X, [A_G/T_{\ad}])$ sends every $k$-scheme $S$ to the set of pairs $(E, \varphi)$ consisting of a $T_{\ad}$-torsor $E$ over $X \times S$ and a section $\varphi \in H^0(X \times S, E \wedge^{T_{\ad}} A_G)$.

Consider the open substack $\mathcal{B}^{\ad}_X$ of $\underline{\Hom}(X, [A_G/T_{\ad}])$ such that for every $k$-scheme $S$, the set $\mathcal{B}^{\ad}_X(S)$ contains pairs $(E, \varphi)$ satisfying the following: 
for any geometric point $s \in S$, the image of the generic point of $X \times \{s\}$ under $\varphi$ is contained in $E \wedge^{T_{\ad}} T_{\ad}$. In other words, each $k$-point of $\mathcal{B}^{\ad}_X$ corresponds to a map $h_{(E, \phi)}: X \rightarrow [A_G/T_{\ad}]$ such that $h^{-1}_{(E, \varphi)}[T_{\ad}/T_{\ad}]$ is a non-empty subset of $X$.
\begin{prop}[\cite{BNS16}, Lemma 3.4] There is a canonical bijection between the set of boundary divisors on $X$ and $\mathcal{B}^{\ad}_X$.
\end{prop}
\subsection{Global constructions}
The translation action of $Z_+=T$ on $V_G$ induces the following maps of quotient stacks:
$$[V_G/(T \times G)] \rightarrow [\mathfrak{C}_+/T] \rightarrow [A_G/T] \rightarrow BT.$$
Meanwhile, the quotient $T \rightarrow T_{\ad}$ induces the map $\Bun_T \rightarrow \Bun_{T_{\ad}}.$ This, along with the above, gives rise to the following maps between mapping stacks:
\begin{center} 
\begin{tikzpicture} \label{mHdiag}
\node (A) at (0, 0) {$\underline{\Hom}(X, [V_G/(T \times G)])$};
\node (B) at (5, 0) {$\underline{\Hom}(X, [\mathfrak{C}_+/T] )$};
\node (C) at (9, 0) {$\underline{\Hom}(X, [A_G/T])$};
\node (D) at (12, 0) {$\Bun_T$};
\node (E) at (9, -2) {$\underline{\Hom}(X, [A_G/T_{\ad}])$};
\node (F) at (12, -2) {$\Bun_{T_{\ad}}$};
\draw[->] (A) to node {} (B);
\draw[->] (B) to node {} (C);
\draw[->] (C) to node {} (D);
\draw[->] (C) to node {} (E);
\draw[->] (E) to node {} (F);
\draw[->] (D) to node {} (F);
\end{tikzpicture}
\end{center}
Let $\mathcal{B}_X$ be the preimage of  $\mathcal{B}^{\ad}_X$ of boundary divisors under the map $\Bun_T \rightarrow \Bun_{T_{\ad}}$. Since $\mathcal{B}_X \rightarrow \mathcal{B}^{\ad}_X$ is a $Z$-gerbe over its image, $\mathcal{B}_X$ is a proper Deligne-Mumford stack, locally of finite type. Let $\mathcal{M}_X \subseteq \underline{\Hom}(X, [V_G/(T \times G)])$ and $\mathcal{A}_X \subseteq \underline{\Hom}(X, [\mathfrak{C}_+/T] )$ be the corresponding preimages  of $\mathcal{B}_X$. They all fit into the diagram:
\begin{equation} 
\begin{tikzpicture}[node distance=2.4cm]
\node (A) {$\mathcal{M}_X$};
\node (B) [right of = A] {$\mathcal{A}_X$};
\node (C) [right of = B] {$\mathcal{B}_X$};
\node (D) [right of = C] {$\Bun_T$};
\node (E) [below of = C] {$\mathcal{B}^{\ad}_X$};
\node (F) [below of = D] {$\Bun_{T_{\ad}}$};
\draw[->, below] (A) to node {} (B);
\draw[->] (B) to node {} (C);
\draw[->] (C) to node {} (D);
\draw[->] (C) to node {} (E);
\draw[->] (E) to node {} (F);
\draw[->] (D) to node {} (F);
\end{tikzpicture}
\end{equation}
If $b \in \mathcal{B}_X$, denote the fiber of $\mathcal{A}_X$ over $b$ as $\mathcal{A}_b.$
\begin{defn}
The \emph{multiplicative Hitchin fibration} is the map $$h_X: \mathcal{M}_X \rightarrow \mathcal{A}_X$$ with total stack $\mathcal{M}_X$ and base $\mathcal{A}_X$. Fibers of $h_X$ over $a \in \mathcal{A}_X$ are called \emph{multiplicative Hitchin fibers} and they are denoted by $\mathcal{M}_a$.
\end{defn}
Choose a $Z_+$-torsor $L \in \Bun_T$ on $X$. Denote $V_G^L$, $A_G^L$ and $\mathfrak{C}_+^L$ as the twists of $V_G$, $A_G$ and $\mathfrak{C}_+$ by $L$. 
For example, $V^L_G = V_G \wedge^{Z_+} L = V_G \wedge^T L$. 
The fiber of the top row of \eqref{mHdiag} over $L \in \Bun_T$ is a sequence of maps
$$ \underline{\Hom}(X, V^L_G) \rightarrow \underline{\Hom}(X, \mathfrak{C}^L_+) \rightarrow  \underline{\Hom}(X, A_G^L).$$
The substack $\mathcal{M}_{X, L}$ classifies \emph{($L$-twisted) multiplicative Higgs bundles}  $(E, \varphi)$ consisting of a $G$-torsor $E$ on $X$ and a \emph{multiplicative Higgs field} $\varphi \in H^0(X, E \wedge^G V^L_G).$
For each $\omega \in X^*(T)$, denote $\omega(L)$ as the line bundle on $X$ obtained by pushing $L$ along $\omega$. Recall that $\mathfrak{C}_+ = A_G \times (G \sslash \Ad(G))$. Then, 
$$\underline{\Hom}(X, A_G^L) = \bigoplus^r_{i=1} H^0(X, \alpha_i(L)), \quad \underline{\Hom}(X, \mathfrak{C}^L_+)= \underline{\Hom}(X, A_G^L) \oplus \left(\bigoplus^r_{i=1} H^0(X, \omega_i(L))\right).$$
The open subspace $\mathcal{B}_{X, L}$ of $\underline{\Hom}(X, A_G^L)$ contain sections $(b_1, ..., b_r)$ where $b_i \neq 0$ for all $1 \leq i \leq r$. 
\subsection{Symmetries}
In Proposition \ref{jscheme}, we have a commutative group scheme $J$ over $\mathfrak{C}_+$. Since the $Z_+$-action on $\mathfrak{C}_+$ lifts to an action of $J$, we have the map
$[J/Z_+] \rightarrow [\mathfrak{C}_+/Z_+]$. Define $J_X$ as the commutative group scheme given by the pullback
\begin{center} 
\begin{tikzpicture}
\node (A) at (0,0) {$J_X$};
\node (B) at (2.5, 0) {$[J/Z_+]$};
\node (C) at (0, -2) {$X \times \mathcal{A}_X$};
\node (D) at (2.5, -2) {$[\mathfrak{C}_+/Z_+]$};
\draw[->, below] (A) to node {\hspace{-1.4cm} $\mathlarger{\mathlarger{\mathlarger{\lrcorner}}}$} (B);
\draw[->] (B) to node {} (D);
\draw[->, above] (C) to node {$\eva$} (D);
\draw[->] (A) to node {} (C);
\end{tikzpicture}
\end{center}
\begin{defn}
The \emph{relative Picard stack} is the smooth group stack 
$$p_X: \mathcal{P}_X = \Pic(J_X/X \times \mathcal{A}_X/\mathcal{A}_X) \rightarrow \mathcal{A}_X.$$ It classifies $J_X$-torsors over $X$ relative to $\mathcal{A}_X$ and naturally acts on $\mathcal{M}_X$ induced by the homomorphism $\chi^*_+I \rightarrow J.$
\end{defn}
If $a \in \mathcal{A}_X,$ denote $\mathcal{P}_a$ as the fiber of $\mathcal{P}_X$ of $a$. Let $\mathcal{M}^{\reg}_X$ be the substack of $\mathcal{M}_X$ classifying multiplicative Higgs bundles whose Higgs fields have images contained in $[V^{\reg}_G/G]$. 
\begin{defn}
The \emph{universal cameral cover} $\widetilde{\pi}: \widetilde{X} \rightarrow X \times \mathcal{A}_X$ is defined as the pullback
\begin{center} 
\begin{tikzpicture}
\node (A) at (0, 0) {$\widetilde{X}$};
\node (B) at (2.5, 0) {$[\overline{T}_+/Z_+]$};
\node (C) at (0, -2) {$X \times \mathcal{A}_X$};
\node (D) at (2.5, -2) {$[\mathfrak{C}_+/Z_+]$};
\draw[->, below] (A) to node {\hspace{-1.4cm} $\mathlarger{\mathlarger{\mathlarger{\lrcorner}}}$} (B);
\draw[->] (B) to node {} (D);
\draw[->, above] (C) to node {$\eva$} (D);
\draw[->] (A) to node {} (C);
\end{tikzpicture}
\end{center}
The fiber of $\widetilde{\pi}$ over $a \in \mathcal{A}_X(S)$ is called \emph{the cameral curve} $\widetilde{X}_a.$
\end{defn}
\begin{defn} \label{globaldelta}
For any $a \in \mathcal{A}^\heart_X(k),$ the \emph{global $\delta$-invariant of $a$} is defined as 
$$\delta_a = \sum_{v \in X \setminus U} \delta_v(a).$$
\end{defn}
Like its local counterpart, the global $\delta$-invariant can be described via Neron models and cameral curves. 
\subsection{Important loci in $\mathcal{A}_X$}
Recall that there is a discriminant divisor $\mathfrak{D}_+$ and a numerical boundary divisor $\mathfrak{B}_+$ in $\mathfrak{C}_+.$
\begin{defn}
The \emph{universal discriminant divisor} $\mathfrak{D}_X$ is the inverse image of $[\mathfrak{D}_+/Z_+]$ in $X \times \mathcal{A}_X$ via the evaluation map $\eva: X \times \mathcal{A}_X \rightarrow [\mathfrak{C}_+/Z_+].$ We also define $\mathcal{B}_X$ as the inverse image of $[\mathfrak{B}_+/Z_+]$. The fibers of $\mathfrak{D}_X$ and $\mathcal{B}_X$ over $a \in \mathcal{A}_X$ are called the \emph{discriminant divisor} $\mathfrak{D}_a$ and the \emph{boundary divisor} $\mathcal{B}_a.$ We also denote the pre-image of $[\mathfrak{C}^\times_+/Z_+]$ as $\mathfrak{C}_X^\times.$
\end{defn}
Next, we will consider a series of important subsets within $\mathcal{A}_X.$
\begin{defn}
The \emph{generically regular semisimple locus} $\mathcal{A}^\heart_X$ is the open locus in $\mathcal{A}_X$ such that $\mathfrak{D}_a$ is either empty or an effective divisor. It contains sections $a \in \mathfrak{C}_+$ that generically lie in the open subset $\mathfrak{C}^{\rs}_+$, i.e. $a(X) \not \subset \mathfrak{D}_+.$
\end{defn}
For $a \in \mathcal{A}^\heart_X$, the cameral curve $\widetilde{X}_a$ is reduced. Also, if $b = \sum_{v \in X(k)} \lambda_v v$ is a boundary divisor, then $\deg(\mathfrak{D}_a) = \sum_{v \in X(k)} \langle 2 \rho, \lambda_v \rangle$ \cite[Lemma 6.3.4]{Wang23}.
\begin{defn} The \emph{anisotropic locus} of $\mathcal{A}^\heart_X$
is given by the open subset $$\mathcal{A}^{\ani}_X = \{ a \in \mathcal{A}^\heart_X : \pi_0(\mathcal{P}_a) \text{ is finite.} \}.$$
\end{defn}
We denote the restrictions of $\mathcal{M}_X$ to $\mathcal{A}^\heart_X$ and $\mathcal{A}^{\ani}_X$ by $\mathcal{M}^\heart_X$ and $\mathcal{M}^{\ani}_X$ respectively.
\begin{defn}
For every positive integer $N$, a $Z_+$-torsor $L$ on $X$ is called \emph{very $(G, N)$-ample} if 
$$\deg(\omega(L)) > 2r(2g-1)+2+N$$
for every fundamental weight $\omega$ of $G$. If $N = 0$, then we say that $L$ is \emph{very $G$-ample}.
\\\\
A point $a \in \mathcal{A}_X(k)$ or $b \in \mathcal{B}_X(k)$ is \emph{very $(G, N)$-ample} if its associated bundle $L$ is. Let $\mathcal{A}_{\gg N} \subseteq \mathcal{A}_X$ and $\mathcal{B}_{\gg N} \subset \mathcal{B}_X$ be the set of very $(G, N)$-ample points. If $N = 0$, then we denote $\mathcal{A}_{\gg} := \mathcal{A}_{\gg 0}$ and $\mathcal{B}_{\gg} := \mathcal{B}_{\gg 0}$.
\end{defn}
\subsection{Product formula}
The relationship between multiplicative affine Springer fibers and multiplicative Hitchin fibers is captured neatly by the Product formula.  Let $a \in \mathcal{A}^{\ani}_X(k)$ and $U_a = a^{-1}([\mathfrak{C}_+^{rs}/Z^+]) \subset X$. For each $v \in X$, we can consider the restriction $a_v = a|_{\Spec(\mathcal{O}_v)}$ satisfying
$$\beta_G(a_v) \in t^{-w_0(\lambda_{\ad})}T_{\ad}(\mathcal{O}_v) \subset A_G(\mathcal{O}_v) \cap T_{\ad}(F_v)$$ for some $\lambda_{\ad} \in X_*(T_{\ad})^+$.
\begin{prop}[\cite{Wang23}, Proposition 6.9.1]
There is a homeomorphism of stacks,
$$\mathcal{P}_a \times^{\prod_{v \in X \setminus U_a} P_v(a_v)} \prod_{v \in X \setminus U_a} M_v(a) \rightarrow \mathcal{M}_a.$$
\end{prop}
This was first established in \cite[Theorem 4.2.10]{Chi22}, who states that the proof is the same as the Lie algebra case \cite[Theorem 4.13.1]{Ngo08}. In \cite[\S 6.9]{Wang23}, the author spells out the proof in greater detail and proves a version for 2-stacks.
 
\subsection{Global affine Schubert Scheme}
For a scheme $S$, fix an element $b = (L, \theta) \in \mathcal{B}^{\ad}_X(S)$ consisting of a $T_{\ad}$-torsor $L$ over $X \times S$ and a section $\theta$ of $A_G^L$ that is generically contained in $T_{\ad}^L$ over every geometric point $s \in S$. Recall that the numerical divisor $\mathfrak{B}_+$ is a principal divisor in $A_G$ cut out by the product of all simple roots:
$$\Pi_\Delta = \prod^r_{i=1} e^{\alpha_i}.$$
Its pullback $\theta^*\Pi_\Delta$ defines a Cartier divisor $\mathfrak{B}_b$ on $X \times S$ relative to $S$. Let $\widehat{X}_{\mathfrak{B}_b}$ be the formal completion of $X \times S$ at $\mathfrak{B}_b.$ Define the following arc space over $\mathcal{B}^{\ad}_X$,
\begin{align}
L^+_{\mathcal{B}^{\ad}_X}G(S) &= \{(b, g): b \in \mathcal{B}^{\ad}_X(S), g \in G(\widehat{X}_{\mathfrak{B}_b})\} \label{globalgroup}.
\end{align}

We will now define the arc space $L^+_{\mathcal{B}^{\ad}_X}(V_G/A_G)$ locally. For any $b$, choose locally over $S$ a trivialization of the $T_{\ad}$-torsor $L$ over the disc $\widehat{X}_{\mathfrak{B}_b}$. It then lifts to a $T$-torsor $\widetilde{L}$ over $\widehat{X}_{\mathfrak{B}_b}$. Now, we can form the torsor $V_G^{\widetilde{L}}$ and consider the Cartesian diagram:
\begin{equation} \label{localcart}
\begin{tikzpicture}[node distance=2cm]
\node (A) {$V^L_{G, b}$};
\node (B) [right of = A] {$V_G^{\widetilde{L}}$};
\node (C) [below of = A] {$\widehat{X}_{\mathfrak{B}_b}$};
\node (D) [right of = C] {$A_G^L$};
\draw[->, below] (A) to node {\hspace{-1.3cm} $\mathlarger{\mathlarger{\mathlarger{\lrcorner}}}$} (B);
\draw[->] (B) to node {} (D);
\draw[->, above] (C) to node {$\theta$} (D);
\draw[->] (A) to node {} (C);
\end{tikzpicture}
\end{equation}
If $b$ corresponds to the boundary divisor $\sum_{v \in X(k)} \lambda_v v$, then $L^+_b(V^L_{G,b})(k)$ over $b$ is non-canonically isomorphic to $$\prod_{v \in X(k)} \left( \bigcup_{\substack{\mu \in X_*(T^{\ad})^+ \\ \mu \leq \lambda_v}} G(\mathcal{O}_v)t^{\langle -w_0(\lambda_v), \mu\rangle}_v G(\mathcal{O}_v)\right).$$
Define the following arc space over $\mathcal{B}^{\ad}_X,$
$$L^+_{\mathcal{B}^{\ad}_X}(V_G/A_G)(S) = \{(b, g): b \in \mathcal{B}^{\ad}_X(S), g \in V^L_{G, b}(\widehat{X}_{\mathfrak{B}_b})\}.$$
\begin{defn}
The \emph{global affine Schubert scheme associated with $V_G$} is the sheaf  
$$Q^{\ad}_X := L^+_{\mathcal{B}^{\ad}_X}(V_G/A_G)/L^+_{\mathcal{B}^{\ad}_X}G.$$
Note that it does not depend on the local trivialization of $L$ because the ambiguity involved in lifting $L$ to a $T$-torsor is cancelled out after taking the quotient by $L^+_{\mathcal{B}^{\ad}_X}G$.
\end{defn}
Its fiber over $b \in \mathcal{B}^{\ad}_X$ is 
$$\prod_{v \in X(k)} \Gr^{\leq \lambda_v}_{G_{\ad}}(k) = \prod_{v \in X(k)} \left( \bigcup_{\substack{\mu_v \in X_*(T^{\ad})_+ \\ \mu_v \leq \lambda_v}} G_{\ad}(\mathcal{O}_v)t^{\mu_v}_v G_{\ad}(\mathcal{O}_v)\right).$$
Replacing $V_G$ with $V^0_G$ in the above definition gives the sheaf
$$Q^{\ad, 0}_X := L^+_{\mathcal{B}^{\ad}_X}(V^0_G/A_G)/L^+_{\mathcal{B}^{\ad}_X}G$$
whose fibers over $b \in \mathcal{B}^{\ad}_X$ are 
$$\prod_{v \in X(k)} \Gr^{\lambda_v}_{G_{\ad}}(k) = \prod_{v \in X(k)} G_{\ad}(\mathcal{O}_v)t^{\lambda_v}_v G_{\ad}(\mathcal{O}_v).$$
In particular, if $\lambda_v = 0$ for some $v$, then $\Gr^{\leq \lambda_v}_{G_{\ad}} = \Gr^{\lambda_v}_{G_{\ad}} = G_{\ad}(\mathcal{O}_v)$.
Let $Q_X \rightarrow \mathcal{B}_X$ be the pull-back of $Q^{\ad}_X \rightarrow \mathcal{B}^{\ad}_X$ to $\mathcal{B}_X$. 
\begin{defn}\label{globhecke}
We call the pro-algebraic stack $[Q_X]_G := [L^+_{\mathcal{B}_X}G\setminus Q_X]$ the \emph{global Hecke stack}.
\end{defn}
More concretely, over the point $b \in \mathcal{B}_X$ with boundary divisor
$$\lambda_b = \sum_{v \in X(k)} \lambda_v v,$$
the fiber of $[Q_X]_G$ is isomorphic to 
$$\prod_{v \in X(k)} [L^+_b G\setminus \Gr_{G_{\ad}}^{\leq -w_0(\lambda_{\ad, v})}] := \prod_{v \in X(k)} \left[L^+_b G\setminus \left(\overline{G_{\ad}(O_{v})t^{-w_0(\lambda_{\ad, v})}_v G_{\ad}(O_{v})}\right)\right].$$
We can also introduce truncated versions. For each $b \in \mathcal{B}^{\ad}_X$, let $I_{\mathfrak{B}_b}$ be the ideal in $O_{X \times S}$ defining the divisor $\mathfrak{B}_b.$ Furthermore for each $N \in \mathbb{N},$ consider the infinitesimal neighbourhoods $D_N$ defined by $I_{\mathfrak{B}_b}^N$. Define the $N$th jet group
$$L^+_{\mathcal{B}^{\ad}_X, N} G(S) = \{(b, g): b \in \mathcal{B}^{\ad}_X(S), g \in G(D_N)\}.$$
The action of $L^+_{\mathcal{B}^{\ad}_X} G$ on $Q^{\ad}_X$ factors through jet group $L^+_{\mathcal{B}^{\ad}_X, N} G$.
Let $Q_X \rightarrow \mathcal{B}_X$ be the pull-back of $Q^{\ad}_X \rightarrow \mathcal{B}^{\ad}_X$ to $\mathcal{B}_X$. 
\begin{defn} \label{globheckeN}
 For $N \in \mathbb{N}$, we call the quotient stack $[Q_X]_{G, N} := [L^+_{\mathcal{B}_{X, N}}G \setminus Q_X]$ the \emph{$N$-truncated global Hecke stack}.
\end{defn}
\subsection{Local model of singularities}
Unlike the Lie algebra case, the total multiplicative Hitchin stack $\mathcal{M}_X$ is not smooth. 
To capture the local behaviour near its singularities, we build local models of $\mathcal{M}_X$ using global affine Schubert cells. By \cite[Theorem 6.10.2]{Wang23}, we can only do so when the global obstruction to deforming a point in $\mathcal{M}_X$ is completely determined by its local obstructions. 

Let $I \rightarrow V_G$ be the universal centralizer with fibers $I|_{\gamma} = G_\gamma$. Then, $I \rightarrow V_G$ is $(G \times T)$-equivariant and it descends to the quotient stack $[V_G/(G \times T)].$ Let $I_{(E, \varphi)}$ be the pullback of $I$ along the map $(E, \varphi) \in \Hom(X, [V_G^L/G])$. The group $I_{(E, \varphi)}$ is not flat in general but by \cite{BLR12}, there exists a unique smooth group scheme $I^{\mathrm{sm}}_{(E, \varphi)}$ over $X$ such that for any other smooth scheme $S$ over $X$, 
$$\Hom_X(S, I_{(E, \varphi)}) \cong \Hom_X(S, I^{\mathrm{sm}}_{(E, \varphi)}).$$
The natural action of $Z$ on $I^{\mathrm{sm}}_{(E, \varphi)}$ induces an action of the Lie algebra $\mathfrak{z}_G = \Lie(Z)$ on $\Lie( I^{\mathrm{sm}}_{(E, \varphi)})$.
\begin{thrm}[\cite{Wang23}, Theorem 6.10.2] Let $m = (L, E, \varphi) \in \mathcal{M}^\heart_X(k)$ be a point and $a = h(m)$. If 
$$H^1(X, (\Lie(I^{\mathrm{sm}}_{(E, \varphi)})/\mathfrak{z}_G)^*) = 0,$$
then for large enough $N$, the local evaluation map $\eva: \mathcal{M}_X \rightarrow [Q_X]$ is formally smooth at $m \in \mathcal{M}^\heart_X.$ Meanwhile, $\eva_N: \mathcal{M}_X \rightarrow [Q_X]_N$ is smooth at $m$. In particular, this is true if $L$ is very $(G, \delta(a))$-ample.
\end{thrm}
Let $b \in \mathcal{B}_X$ have associated boundary divisor $\lambda = \sum_{v \in X(k)} \lambda_v v$. Then, the local evaluation map sends a point $(L, E, \varphi) \in \mathcal{M}_X$ to the product of trivializations of $\varphi$ over formal neighbourhoods of $v \in X(k)$, that is $(\gamma_v)_{v \in X(k)} \in \prod_{v \in X(k)} \Gr_{G_{\ad}}^{\leq \lambda_v}$, modulo the action by $L^+_bG.$ When $\lambda_v = 0$, then $\gamma_v \in G_{\ad}(\mathcal{O}_v).$
\section{Parabolic multiplicative affine Springer fibers}
\subsection{Definitions}
We will now introduce the parabolic analogues of multiplicative affine Springer fibers.
\begin{defn} \label{PMASF}
Let $\gamma \in G^{\rs}(F)$ be a regular semisimple element and $\lambda \in X_*(T)^+$ be a dominant coweight. The \emph{parabolic Kottwitz-Viehmann varieties} or \emph{parabolic mulitplicative affine Springer fibers} are given by the sets
\begin{align} \label{pmasfo}
X^{\lambda, \para}_\gamma &= \left\{g \in \Fl_G(k): \Ad_{g^{-1}}(\gamma) \in \bigcup_{\mu \in W(\lambda)} It^{\mu}I \right \}, \\
\label{pmasfc}
X^{\leq \lambda, \para}_\gamma &=  \left \{g \in \Fl_G(k): \Ad_{g^{-1}}(\gamma) \in \bigcup_{w \in \Adm(\lambda)} IwI \right\}.
\end{align}
\end{defn} 
Equation \eqref{pmasfc} is a non-Frobenius-twisted analog of the union of affine Deligne Lusztig varieties indexed by an admissible set. Under the projection $\Fl_G \rightarrow \Gr_G,$ the pre-image of a Schubert cell $G(\mathcal{O}) t^\lambda G(\mathcal{O})$ is the union $\bigcup_{w \in Wt^\lambda W} IwI,$ which contains the union $\bigcup_{\mu \in W(\lambda)} It^\mu I.$ Hence, $X^{\lambda, \para}_\gamma$ is contained in the pre-image of $X^\lambda_\gamma$ under this projection. 
\subsection{Non-emptiness}
We will now establish the non-emptiness criterion of parabolic multiplicative affine Springer fibers. We first prove some basic properties.
\begin{lem} \label{nonemppar}
The natural projection $\pi: \Fl_G \rightarrow \Gr_G$ induces surjections $X^{\leq \lambda,\para}_\gamma \rightarrow X^{\leq \lambda}_\gamma$ and $X^{\lambda,\para}_\gamma \rightarrow X^{\lambda}_\gamma$. 
\end{lem}
\begin{proof}
First, we will show that the affine Schubert variety can be rewritten as 
\begin{equation} \label{affineadmiss} \overline{G(\mathcal{O})t^\lambda G(\mathcal{O})} = \bigcup_{w \in \Adm(\lambda)}G(\mathcal{O})wG(\mathcal{O}).
\end{equation}
By definition, any element $w$ of $\Adm(\lambda) =\{w\in \widetilde{W}: w \leq t^{x(\lambda)} \text{ for some }x \in W\}$
can be written as $w = at^\mu \in \widetilde{W} = W \ltimes X_*(T)$. Then, 
$$WwW= Wat^\mu W = Wt^\mu W = Wt^{\mu_+}W,$$
where $\mu_+$ is the dominant $W$-conjugate of $\mu$. Since $w \in \Adm(\lambda)$,
$$Wt^{\mu_+}W = WwW \leq  Wt^{x(\lambda)}W = W t^\lambda W,$$
which occurs iff $\mu_+ \leq \lambda$ \cite[Lemma 4.2]{Hai01}. Thus,
$G(\mathcal{O})wG(\mathcal{O}) = G(\mathcal{O})t^{\mu_+} G(\mathcal{O}) \subseteq \overline{G(\mathcal{O})t^\lambda G(\mathcal{O})}.$
Since this is true for any $w \in \Adm(\lambda)$, we have that
$$\bigcup_{w \in \Adm(\lambda)}G(\mathcal{O})wG(\mathcal{O}) \subseteq \overline{G(\mathcal{O})t^\lambda G(\mathcal{O})}.$$
Conversely, if $\mu \in X_*(T)^+$ satisfies $\mu \leq\lambda$, then $t^\mu \leq t^\lambda$ by \cite{Rap05} and so, $t^\mu \in \Adm(\lambda)$. Thus,  $\overline{G(\mathcal{O})t^\lambda G(\mathcal{O})} \subseteq \bigcup_{w \in \Adm(\lambda)}G(\mathcal{O})wG(\mathcal{O}).$ 

To show surjectivity, we follow \cite[\S 6.3]{He14}. Let $P_J$ be the standard parahoric subgroup associated with the set $J \subset \widetilde{S}.$ Then, set $\Adm^J(\lambda) = W_J\Adm(\lambda)W_J$ and 
$$Y_{J, \lambda} = \bigcup_{w \in \Adm(\lambda)} P_JwP_J = \bigcup_{w \in \Adm^J(\lambda)}IwI.$$
By \cite[Theorem 6.1]{He14}, $\prescript{J}{}{\widetilde{W}} \cap \Adm^J(\lambda) = \prescript{J}{}{\widetilde{W}} \cap \Adm(\lambda)$. Hence,
$$Y_{J,\lambda} = \bigcup_{w \in \prescript{J}{}{\widetilde{W}} \cap \Adm^J(\lambda)} P_J \cdot (IwI) = \bigcup_{w \in \widetilde{W} \cap \Adm(\lambda)} P_J \cdot  (IwI) \subset  P_J \cdot  Y_{\emptyset, \lambda} = P_J \cdot \left( \bigcup_{w \in \Adm(\lambda)} IwI\right).$$
Since $I \subset P_J$, we also have the reverse inequality $Y_{J,\lambda} \supset P_J \cdot Y_{\emptyset, \lambda}$. Combining the two gives the equality $Y_{J,\lambda} = P_J \cdot Y_{\emptyset, \lambda}$. Setting $J = S\subset \widetilde{S}$, we have that $P_J = G(\mathcal{O})$ and 
$$\overline{G(\mathcal{O})t^\lambda G(\mathcal{O})}  = \bigcup_{w \in \Adm(\lambda)}G(\mathcal{O}) w G(\mathcal{O}) = Y_{S,\lambda} = G(\mathcal{O})\cdot \left(\bigcup_{w \in\Adm(\lambda)} IwI \right) = G(\mathcal{O}) \cdot_\sigma Y_{\emptyset, \lambda}.$$
Let $\pi: \Fl_G \rightarrow \Gr_G$ be the  natural projection map. Then, 
\begin{align*}
\pi(X^{\leq \lambda, \para}_\gamma) &= \left\{g \in \Gr_G: \Ad_{g^{-1}}(\gamma) \in Y_{S, \lambda} = G(\mathcal{O})\cdot \left(\bigcup_{w \in\Adm(\lambda)} IwI \right) \right\} \\
&= \left\{g \in \Gr_G: \Ad_{g^{-1}}(\gamma) \in Y_{S, \lambda} = \bigcup_{w \in\Adm(\lambda)} G(\mathcal{O})wG(\mathcal{O})  \right\} = X^{\leq \lambda}_\gamma.
\end{align*}
\end{proof}
Since the maps $X^{\leq \lambda,\para}_\gamma \rightarrow X^{\leq \lambda}_\gamma$ and $X^{\lambda,\para}_\gamma \rightarrow X^{\lambda}_\gamma$ are surjective, we have the following corollary.
\begin{cor} \label{pmasfnonemp}
The parabolic multiplicative affine Springer fibers $X^{\leq \lambda,\para}_\gamma$ and $X^{\lambda, \para}_\gamma$ are non-empty if and only if $\kappa_G(\gamma) = p_G(\lambda)$ and $\nu_\gamma  \leq \lambda.$
\end{cor}
\begin{proof}
If the parabolic multiplicative affine Springer fibers are non-empty, then one can use $\pi$ to project down one of its non-trival elements to a non-trivial element in the multiplicative affine Springer fibers. Since multiplicative affine Springer fibers are non-empty, $\kappa_G(\gamma) = p_G(\lambda)$ and $\nu_\gamma  \leq \lambda$ by Lemma \ref{nonemp}. Conversely, this condition implies the non-emptiness of the multiplicative affine Springer fibers. Since $\pi_G: \Fl_G \rightarrow \Gr_G$ is surjective, there exists a non-trivial element in the parabolic multiplicative affine Springer fibers that maps to a non-trivial element in the multiplicative affine Springer fibers.
\end{proof}
\subsection{Ind-scheme structure}
Let $v \in X$ be a point, $X_v = \Spec(\mathcal{O}_v)$ be the formal disc and $X^\bullet_v = \Spec(F_v)$ be the punctured disc at $v.$ Let $\gamma_+ \in G_+(F_v)$ with image $\chi_+(\gamma_+) = a \in \mathfrak{C}_+(\mathcal{O}_v).$ 
\begin{defn} \label{pmasfdefn}
The \emph{parabolic multiplicative affine Springer fiber $M^{\para}_v(\gamma_+)$} is the functor that maps a scheme $S$ to the set of isomorphism classes of pairs $(h, \iota)$ that fit into the following commutative diagram:
\begin{center} 
\begin{tikzpicture}[node distance=2cm]
\node (A) {$X_v \widehat{\times} S$};
\node (B) [right of = A] {$[\widetilde{V}_G/G]$};
\node (C) [below of = A] {$X_v$};
\node (D) [right of = C] {};
\node (E) [right of = B] {$[V_G/G]$};
\node (F) [right of = D] {$\mathfrak{C}_+$};
\draw[->, above] (A) to node {$h$} (B);
\draw[->, right] (E) to node {$\chi_+$} (F);
\draw[->,above] (C) to node {$a$} (F);
\draw[->] (A) to node {} (C);
\draw[->, above] (B) to node {$\pi_G$} (E);
\end{tikzpicture}
\end{center}
Here, $\pi_G$ is induced by the Grothendieck simultaneous resolution and $\iota$ is an isomorphism between $(\pi_G \circ h)|_{X^\bullet_{v, S}}$ and the induced map 
$X_{v, S}^\bullet \xrightarrow{\gamma_+} V_G \rightarrow [V_G/G].$ By replacing $V_G$ with $V_G^\circ$ and $V_G^{\reg}$, we obtain subfunctors $M^{\para}_v(\gamma_+)^\circ$ and $M^{\para}_v(\gamma_+)^{\reg}.$ Since the isomorphism classes of $M^{\para}_v(\gamma_+)$ and $M^{\para}_v(\gamma_+)^\circ$ only depend on $a = \chi_+(\gamma_+)$, we will also denote these as $M^{\para}_v(a)$ and $M^{\para}_v(a)^\circ.$
\end{defn}
For $\gamma \in G^{\rs}(F_v)$ and $\lambda \in X_*(T)^+,$ suppose that $X^{\lambda, \para}_{\gamma}$ is non-empty. By Proposition \ref{nonemp}, there exists an element $\gamma_\lambda \in G_+(F_v)$ such that 
$$X^{\lambda, \para}_{\gamma} \cong  M^{\para}_v(\gamma_\lambda)^\circ, \quad X^{\leq \lambda, \para}_{\gamma} \cong M^{\para}_v(\gamma_\lambda).$$
\begin{defn}[Moduli description] \label{indscheme}
The parabolic multiplicative affine Springer fiber $M^{\para}_v(\gamma_+)$ is the functor that maps a scheme $S$ to the set of quadruples $(E, \varphi, E^B_v, \sigma)$, where
\begin{itemize}
    \item $E$ is a $G$-torsor on the $v$-adic completion $X_v \widehat{\times} S$ of $X_v \times S,$
    \item $\varphi \in H^0(X_v \widehat{\times} S, E \times_G V_G)$,
    \item $E^B_x$ is a $B$-reduction along $\{v\} \times S$,
    \item $\sigma$ is a trivialization $(E, \varphi)|_{X^\bullet_v \widehat{\times} S} \cong (E^0, \gamma_+)$ over the punctured disc, where $E^0$ is the trivial bundle. 
\end{itemize}
Similarly, replacing $V_G$ with $V^0_G$ in the above definition gives rise to $M^{\para}_v(\gamma_+)^\circ.$
\end{defn}

\section{Parabolic multiplicative Hitchin fibration}
We will now introduce the parabolic multiplicative Hitchin fibration and prove its geometric properties. 
\subsection{Global constructions}
Since the Grothendieck-simultaneous resolution \eqref{CartDiag} is $G$-equivariant, we can form the following Cartesian diagram: 
\begin{equation} \label{parab}
\begin{tikzpicture}
\node (A) at (0, 0) {$\underline{\Hom}(X, [\widetilde{V}_G/(T \times G)])$};
\node (B) at (5, 0) {$[\widetilde{V}_G/(T \times G)]$};
\node (C) at (0, -2) {$\underline{\Hom}(X, [V_G/(T \times G)]) \times X$};
\node (D) at (5, -2) {$[V_G/(T \times G)]$};
\draw[->, below] (A) to node {} (B);
\draw[->, left] (A) to node {$\pi_M$} (C);
\draw[->, right] (B) to node {$\pi_G$} (D);
\draw[->, above] (C) to node {$\eva$} (D);
\end{tikzpicture}
\end{equation}
Let $\mathcal{M}^{\para}_X$ be the pre-image of $\mathcal{M}_X \times X \subseteq \underline{\Hom}(X, [V_G/(T \times G)]) \times X$. 
\begin{defn} \label{pmhfib}
The \emph{parabolic multiplicative Hitchin fibration} with total stack $\mathcal{M}^{\para}_X$ is the map
$h^{\para}_X = (h_X \times \id_X) \circ \pi_M : \mathcal{M}^{\para}_X \rightarrow \mathcal{A}_X \times X.$ Fibers of $h^{\para}_X$ over $(a, x) \in \mathcal{A}_X \times X$ are called \emph{parabolic multiplicative Hitchin fibers} and they are denoted by $\mathcal{M}^{\para}_{a, x}$.
\end{defn}
We gain a more concrete description when we twist all of the stacks by a $Z_+$-torsor $L$. Let $\mathcal{M}^{\para}_{X, L}$ be the functor that sends a scheme $S$ to the groupoid of quadruples $(x, E, \varphi, E^B_x)$, where
\begin{itemize}
\item $x: S \rightarrow X$ with graph $\Gamma(x)$.
\item $(E, \varphi) \in \mathcal{M}_{X, L}(S)$ is a multiplicative Higgs bundle
\item $E^B_x$ is a $B$-reduction of $E$ along $\Gamma(x)$
\end{itemize}
such that $\varphi$ is compatible with $E^B_x$, i.e.
$$\varphi |_{\Gamma(x)} \in H^0(\Gamma(x), (E^B_x \wedge^G V_B) \otimes_{\mathcal{O}_S} x^* L).$$
and the parabolic multiplicative Hitchin fibration is given by 
\begin{align*}
h^{\para}_{X, L}: \mathcal{M}^{\para}_{X, L} &\rightarrow \mathcal{A}_{X, L} \times X,\\
(x, E, \varphi, E^B_x) &\mapsto (h_X(E, \varphi), x).
\end{align*}
The left vertical arrow $\pi_M$ in \eqref{parab} is obtained by forgetting the $B$-reduction. When restricted to the loci $\mathcal{A}^\heart_X, \mathcal{A}^{\diamond}_X$ or $\mathcal{A}^{\ani}_X$, we denote the restriction of $\mathcal{M}^{\para}_X$ as $\mathcal{M}^{\para, \heart}_X, \mathcal{M}^{\para, \diamond}_X$ and $\mathcal{M}^{\para, \ani}_X$ respectively.
\begin{defn}
Define the \emph{enhanced multiplcative Hitchin base} or \emph{universal cameral cover} by the Cartesian diagram
\begin{center} 
\begin{tikzpicture}[node distance=2cm, auto]
\node (A) {$\widetilde{\mathcal{A}}_X$};
\node (B) [right of = A] {$[\overline{T}_+/T]$};
\node (C) [below of = A] {$\mathcal{A}_X \times X$};
\node (D) [below of = B] {$[\mathfrak{C}_+/T]$};
\draw[->, below] (A) to node {\hspace{-1.1cm} $\mathlarger{\mathlarger{\mathlarger{\lrcorner}}}$}(B);
\draw[->, left] (A) to node {$q$} (C);
\draw[->] (B) to node {} (D);
\draw[->] (C) to node {$\eva$} (D);
\end{tikzpicture}
\end{center}
\end{defn}
For each $a \in \mathcal{A}_X$, the projection $X_a = q^{-1}(\{a\} \times X) \rightarrow X$ from the cameral curve is a ramified $W$-cover. The commutative diagrams \eqref{paraCart} and \eqref{parab} induce \emph{the enhanced parabolic multiplicative Hitchin fibration} $\widetilde{h}^{\para}_X: \mathcal{M}^{\para}_X \rightarrow \widetilde{\mathcal{A}}_X.$
The various stacks we introduced fit into the following commutative diagram:
\begin{equation}  \label{paraboliccartesian}
\begin{tikzpicture}[node distance=2cm, auto]
\node (A) {};
\node (B) [right of = A] {$\mathcal{M}_{X, L} \times_{\mathcal{A}_X} \widetilde{\mathcal{A}}_{X, L}$};
\node (C) [right of = B] {};
\node (D) [right of = C] {$\mathcal{M}_{X, L} \times X$};
\node (E) [below of = B] {};
\node (F) [below of = C] {$\widetilde{\mathcal{A}}_{X, L}$};
\node (G) [below of = D] {};
\node (H) [right of = G] {$\mathcal{A}_{X, L} \times X$};
\node(I) [below of = E] {$[V^L_G/G] \times_{\mathfrak{C}_+} \overline{T}^+$};
\node(J) [below of = G] {$[V^L_G/G]$};
\node (K) [below of = I] {};
\node (L) [right of = K] {$\overline{T}_+^L$};
\node (M) [below of = H] {};
\node (N) [below of = M] {$\mathfrak{C}_+^L$};
\node (P) [left of = I] {};
\node (O) [left of = P] {$[V^L_B/B]$};
\node (Q) [left of = A] {$\mathcal{M}^{\para}_{X, L}$};
\draw[->] (Q) to node {$\nu_M$}(B);
\draw[->] (B) to node {}(D);
\draw[->] (D) to node {}(H);
\draw[->] (B) to node {}(F);
\draw[->] (D) to node {}(J);
\draw[->] (F) to node {}(H);
\draw[->] (B) to node {} (I);
\draw[->] (I) to node {} (J);
\draw[->] (I) to node {} (L);
\draw[->] (L) to node {} (N);
\draw[->] (J) to node {} (N);
\draw[->] (F) to node {} (L);
\draw[->] (H) to node {} (N);
\draw[->] (Q) to node {} (O);
\draw[->] (O) to node {$\nu$} (I);
\end{tikzpicture}
\end{equation}
The three squares formed by the horizontal and vertical arrows are Cartesian. 
\subsection{Symmetries}
Parabolic multiplicative Hitchin fibers also admit an action by the Picard stacks. We first show that there is a $BJ$-action on $[\widetilde{V}_G/G] = [V_B/B]$.
\begin{lem}
For any $(x, gB) \in \widetilde{V}_G^{\reg}(k)$, $I_x \subset \Ad_g(B)$.
\end{lem}
\begin{proof}
Recall that $G$ is a subgroup of $G_+$ by inclusion of the second coordinate: $g \mapsto (e, g)$. Hence, the $\Ad_G$-action on $V_G$ is given by $h \cdot (t, g) = (e, h) \cdot (t, g) =(t, hgh^{-1}).$ The proof is essentially the same as \cite[Lemma 2.4.3]{Ngo08}.

Let $x \in V^{rs}_G$ be a regular semisimple element. Then, $G_x = T$ and $(G_+)_x = (G_x)_+ = T_+$. So, there are $|W|$-many Borel subgroups $B$ such that $(G_+)_x = T_+ \subseteq B_+$. \cite[Lemma 2.2.14]{Chi22} showed that the centralizer $(G_+)_\gamma$ of a semisimple element $\gamma = te_{I,J} \in V_G$ is the subgroup of $L_M$, where $M = (I \cap J^0) \cup J^c$ generated by $T_+, L_{J^c}$ and $(L_{I \cap J^0})_t$. So $I = \emptyset$ and $J = \Delta$ results in $M = \emptyset$ and $(G_+)_\gamma = T_+$. Taking closures, this means that there are $|W|$-many Borel subgroups $B$ such that $x \in V_{(G_+)_x} = \overline{T}_+ \subseteq V_B$. Since the fiber $\pi_G^{-1}(x)$ is discrete, the $I_x = G_x$-action on $\pi_G^{-1}(x)$ is trivial and the result follows. 

Consider a group scheme $H$ over $\widetilde{V}_G^{\reg}$ in which the fiber $H_{(x, gB)}$ over $(x, gB) \in \widetilde{V}_G^{\reg}$ is a subgroup of the centralizer $I_x$ of $x$ containing elements $h$ such that $h \in gBg^{-1}.$ By construction, it is a closed subgroup scheme of $(\pi_G^{\reg})^*I$ that coincides with $(\pi_G^{\reg})^*I|_{V^{\reg}_G}$ over $\widetilde{V}^{rs}_G$. Since $V_G$ is irreducible, the open subsets $V^{\reg}_G$ and $V^{rs}_G$ are dense in $V_G$. Since $(\pi_G^{\reg})^*I$ is flat over $\widetilde{V}_G^{\reg}$, the group subscheme $H$ must be equal to $(\pi_G^{\reg})^*I$. 
\end{proof}
By Proposition \ref{jscheme}, there is a $G$-equivariant isomorphism $(\chi_+^{\reg})^*J \cong I|_{V^{\reg}_G}$ that extends uniquely to a homomorphism $j_G: \chi^*_+J \rightarrow I$. Let $J_B = \chi_+^*J$ be the pullback of $J$ over $\mathfrak{C}_+$ to $V_B$. Let $I_B$ be the universal centralizer group scheme of the adjoint action of $B$ on $V_B$. 
\begin{lem} \label{borelsym}
There is a natural $j_B: J_B \rightarrow I_B$ such that $j_G|_{V_B}$ factors as $J_B \xrightarrow{j_B} I_B \rightarrow I_G|_{V_B}$. Therefore, the stack $BJ$ acts on the stack $[V_B/B]$ over $\mathfrak{C}_+$.
\end{lem}
\begin{proof}
The proof is essentially the same as \cite[Lemma 2.3.1]{Yun11}. For $x \in V_B^{\reg}$, we have the canonical identification $J_x = I_{G, x} = I_{B, x}$. This gives the desired map $j_B$ over $V_B^{\reg}$. Since $J_B$ is a commutative smooth group scheme over $V_B$, we know that it extends uniquely to a homomorphism $J_B \rightarrow I_B$ over $B_+ \cup V^{\reg}_B$ thanks to \cite[Lemma 2.3.1]{Yun11} and \cite[Lemma 2.3.1]{Chi22}. Since the complement of $B_+ \cup V^{\reg}_B$ has codimension at least 2, it extends further to the whole space $V_B$. 
\end{proof}
Let $J_T$ be the pull-back of $J$ along $\overline{T}_+ \rightarrow \mathfrak{C}_+$. 
\begin{lem} \label{Taction}
There is a $W$-equivariant homomorphism $j_T: J_T \rightarrow T \times \overline{T}_+$ of groups schemes over $\overline{T}_+$, which is an isomorphism over $\overline{T}_+^{rs}$. Here, the $W$-structure on $J_T$ is given by its left action on $\overline{T}_+$ while the $W$-structure on $T \times \overline{T}_+$ is given by the diagonal left action. 
\end{lem}
\begin{proof}
Like \cite[Lemma 2.2.2]{Yun11}, this is due to \cite[Proposition 2.4.7]{Chi22} and \cite[Proposition 11]{Bou14}.
\end{proof}
These lemmas imply that there is a $BJ$-action on $[V_B^L/B]$ that is compatible with its action on $[V_G^L/G]$ and preserves the morphism $[V_B^L/B] \rightarrow \overline{T}^L_+$.
\begin{lem} \label{Paction} The Picard stack $\mathcal{P}_X$ acts on $\mathcal{M}^{\para}_{X, L}$. The action is compatible with its action on $\mathcal{M}_{X, L}$ and preserves the enhanced parabolic Hitchin fibration $\widetilde{h}^{\para}_X :\mathcal{M}^{\para}_{X, L} \rightarrow\widetilde{\mathcal{A}}_{X, L}$. That is, when  $\widetilde{\mathcal{P}}_X = \widetilde{\mathcal{A}}_{X, L} \times_{\mathcal{A}_{X, L}} \mathcal{P}_X$ is viewed as a Picard stack over $\widetilde{\mathcal{A}}_{X, L}$, it acts on $\mathcal{M}^{\para}_{X, L}$ over $\widetilde{\mathcal{A}}_{X, L}$.
\end{lem}
\subsection{Geometric properties}
By \cite[Prop 4.2.9]{Chi22} and \cite[Prop 9.1.1]{Wang23}, $\mathcal{M}_X$ is locally of finite type and $\mathcal{M}^{\ani}_X \rightarrow \mathcal{A}^{\ani}_X$ is a relative Deligne Mumford stack of finite type. This is also true for $\mathcal{P}_X$. 
\begin{lem}
$\mathcal{M}^{\para}_X$ is locally of finite type.
\end{lem}
\begin{proof}
The fibers of the natural morphism $\mathcal{M}^{\para}_X \rightarrow \Bun_G \times \Bun_B \times \Bun_{Z_+} \times X$ contain multiplicative Higgs fields $\varphi \in H^0(X, E \wedge^G V_G)$ that are compatible with the $B$-reduction at a point. So, each fiber is the section space of a fixed \'{e}tale-locally trivial fiber bundle with affine fibers. Since $X$ is projective, this section space must be of finite type. Since $ \Bun_G \times \Bun_B \times \Bun_{Z_+} \times X$ is locally of finite type, so is $\mathcal{M}^{\para}_X$. 
\end{proof}
Recall that a morphism $f:X \rightarrow Y$ of algebraic spaces is called schematic if for every scheme $U$ and $y \in Y(U)$, the fiber product $U \times_{y,Y, f} X$ is a scheme.
\begin{lem} $\mathcal{M}^{\para, \ani}_X$ is a Deligne-Mumford stack.
\end{lem}
\begin{proof}
Since $\pi$ is proper, it is of finite type. Since $\pi$ is schematic and of finite type, so is $\pi_M$. Hence, $\mathcal{M}^{\para, \ani}_X$ is Deligne-Mumford.
\end{proof}

\begin{prop}
The restrictions of $h^{\para}_X$ and $\widetilde{h}^{\para}_X$ to $\mathcal{A}^{\ani}_X$ are proper.
\end{prop}
\begin{proof}
By \cite[Proposition 8.1.2]{Wang23}, $h_X|_{\mathcal{A}^{\ani}_X}$ is proper. Since $\mathcal{M}^{\para}_X$ fits into \eqref{paraboliccartesian},
$\mathcal{M}^{\para, \ani}_X$ is obtained by base change from proper maps $[V_B^L/B] \rightarrow [V_G^L/G]$ and $[V_B^L/B] \rightarrow [V_G^L/G] \times_{\mathfrak{C}_+^L} \overline{T}^L_+$. Hence, $\mathcal{M}^{\para, \ani}_X$ is proper over  $\mathcal{M}^{\ani}_X \times X$ and $\mathcal{M}^{\ani}_X \times_{\mathcal{A}^{\ani}_X} \widetilde{\mathcal{A}}^{\ani}_X$. 
\end{proof}
\subsection{Product formula}
We would like to reproduce the Product formula but for parabolic multiplicative affine Springer fibers and parabolic multiplicative Hitchin fibers. Let $a \in \mathcal{A}^\heart_X(k)$. For each $v \in X$, consider the restriction $a_v = a|_{\Spec(\mathcal{O}_v)}$ satisfying
$$\beta_G(a_v) \in t^{-w_0(\lambda_{\ad})}T_{\ad}(\mathcal{O}_v) \subset A_G(\mathcal{O}_v) \cap T_{\ad}(F_v)$$ for some $\lambda_{\ad} \in X_*(T_{\ad})^+$. 
\begin{prop}[Product formula] \label{paraprod}
Let $(a, x) \in \mathcal{A}^{\ani}_X(k) \times X(k)$ and let $U_a = a^{-1}([\mathfrak{C}_+^{rs}/Z^+])$ be the dense open subset of $X$. We have a homeomorphism of stacks:
$$\mathcal{P}_a \times^{P_x(a_x) \times P'} \left(M^{\para}_x(a_x) \times M' \right) \rightarrow \mathcal{M}^{\para}_{a, x},$$
where 
\begin{align*}
P' = \prod_{y \in X \setminus (U_a \cup \{x\})} P_y(a_y), \quad M' = \prod_{y \in X \setminus (U_a \cup \{x\})} M_y(a_y).
\end{align*}
Since $M' = P'$ by \cite[Theorem 5.2.1]{Chi22}, we have that
$$ \dim M^{\para}_x(a_x) - \dim P_x(a_x) = \dim \mathcal{M}^{\para}_{a, x} - \dim \mathcal{P}_a.$$
\end{prop}
\begin{proof}
The proof follows in the same way as the Lie algebra case \cite[Proposition 2.4.1]{Yun11} and the non-parabolic case \cite[Proposition 4.13.1]{Ngo08} and \cite[\S 6.9]{Wang23}.
\end{proof}
\subsection{Global admissible union}
We would like a parabolic version of the global affine Schubert scheme in order to construct a local model of singularities for $\mathcal{M}^{\para}_X.$ For each $b = (L, \theta) \in \mathcal{B}^{\ad}_X(S),$ let $\mathfrak{B}_b$ be the pull-back of the numerical divisor $\mathfrak{B}_+$ along $\theta$ and let $\widehat{X}_{\mathfrak{B}_b}$ be the formal completion of $X \times S$ at $\mathfrak{B}_b.$  
\subsubsection{Global Iwahori subgroup} Recall that by Bruhat-Tits theory \cite[Example 1.2.9]{zhu16}, there is a group scheme $\mathcal{G}_I \rightarrow \Spec(\mathcal{O})$ called the \emph{Iwahori group scheme of $G$} such that $\mathcal{G}_I(\mathcal{O})$ is the Iwahori subgroup $I$ of $G(\mathcal{O})$. 
For each $b \in \mathcal{B}^{\ad}_X(S),$ we can construct a group scheme $I \rightarrow X \times S$ such that $I|_{(X \times S) \setminus \mathfrak{B}_b}$ is the trivial group scheme $G \times ((X \times S) \setminus \mathfrak{B}_b)$ and for every closed point $(x, s) \in \mathfrak{B}_b,$ the restriction $I|_{\Spec(\mathcal{O}_{x, s})}$ is an Iwahori group scheme. By letting $b$ vary, we can construct a global Iwahori subgroup
$$I_{\mathcal{B}^{\ad}_X}(S) = \{(b,g): b \in \mathcal{B}^{\ad}_X(S), g \in I(\widehat{X}_{\mathfrak{B}_b})\}.$$
 
\subsubsection{Global Iwahori submonoid}
Recall that for each $b = (L, \theta)\in \mathcal{B}^{\ad}_X(S)$, the construction of $V^L_{G, b}$ from \eqref{localcart} involves choosing a trivialization of $L$ over the disc $\widehat{X}_{\mathfrak{B}_b}$ and a lift of $L$ to a $T$-torsor $\widetilde{L}$ over $\widehat{X}_{\mathfrak{B}_b}$. So, define $I^L_b \rightarrow \widehat{X}_{\mathfrak{B}_b}$ as the scheme such that for each closed point $(x, s) \in \widehat{X}_{\mathfrak{B}_b},$ the restriction $I^L_b|_{\Spec(\mathcal{O}_{x, s})}$ is an Iwahori submonoid of $V^L_{G, b}(\mathcal{O}_{x, s})$. We can assemble this into a family, 
$$I_{\mathcal{B}^{\ad}_X}^{L}(S) = \{(b, g): b \in \mathcal{B}^{\ad}_X(S), g \in I^L_b(\widehat{X}_{\mathfrak{B}_b})\}.$$
\subsubsection{Global Grothendieck simultaneous resolution}
Define
$$\widetilde{V}^{\widetilde{L}}_G = \{(g, B) \in V^{\widetilde{L}}_G \times G/B: g_x \in V^{\widetilde{L}}_{B, x}, x \in \widehat{X}_{\mathfrak{B}_b}\}\subseteq V^{\widetilde{L}}_G \times G/B.$$
The condition $g_x \in V^{\widetilde{L}}_{B, x}$ means that when restricted to $x \in \widehat{X}_{\mathfrak{B}_b}$, the element $g_x$ belongs to the $B$-reduction $V^{\widetilde{L}}_{B, x}$ of the fiber $V^{\widetilde{L}}_{G,x}.$ Now, consider the following Cartesian diagram: 
\begin{equation} \label{globalgroth}
\begin{tikzpicture}[node distance=2cm]
\node (A) {$V^{L}_{G, b}$};
\node (B) [right of = A] {$V_G^{\widetilde{L}}$};
\node (C) [below of = A] {$\widehat{X}_{\mathfrak{B}_b}$};
\node (D) [right of = C] {$A_G^L$};
\node (E) [above of = A] {$\widetilde{V}^{L}_{G, b}$};
\node (F) [above of = B] {$\widetilde{V}^{\widetilde{L}}_G$};
\draw[->, below] (A) to node {\hspace{-1.3cm} $\mathlarger{\mathlarger{\mathlarger{\lrcorner}}}$} (B);
\draw[->, right] (B) to node {$\alpha_G$} (D);
\draw[->, above] (C) to node {$\theta$} (D);
\draw[->] (A) to node {} (C);
\draw[->] (E) to node {} (A);
\draw[->, right] (F) to node {$\pi_G$} (B);
\draw[->, below] (E) to node {\hspace{-1.3cm} $\mathlarger{\mathlarger{\mathlarger{\lrcorner}}}$} (F);
\end{tikzpicture}
\end{equation}
Just as in equation \eqref{isom}, $\widetilde{V}^{\widetilde{L}}_G \cong G \times_{B} V^{\widetilde{L}}_B$ and $\widetilde{V}^{L}_{G, b} \cong G \times_B V^{L}_{B, b}$. So, we define the arc space 
\begin{align*} 
L^+_{\mathcal{B}^{\ad}_X}(\widetilde{V}^L_G/A_G)(S) &:= \{(b, x): b \in \mathcal{B}^{\ad}_X(S), x \in \widetilde{V}^L_{G, b}(\widehat{X}_{\mathfrak{B}_b})\} = L^+_{\mathcal{B}^{\ad}_X} G\times_{I_{\mathcal{B}^{\ad}_X}} I_{\mathcal{B}^{\ad}_X}^{L} 
\end{align*}
For each $b \in \mathcal{B}^{\ad}_X(k)$, denote the fiber of $I_{\mathcal{B}^{\ad}_X}(k)$ over $b$ by $I_b(k) = \prod_{v \in X(k)} I_v$, where $I_v$ is the Iwahori subgroup of $\mathcal{O}_v.$ If $b$ corresponds to boundary divisor $\sum_{v \in X(k)}\lambda_v v,$ then the set-theoretic fiber $L^+_b(\widetilde{V}^L_{G, b})(k)$ over $b \in \mathcal{B}^{\ad}_X(k)$ is isomorphic to the product
$$L^+_bG(k) \times_{I_b(k)} \prod_{v \in X(k)} \left( \bigcup_{w \in \Adm(-w_0(\lambda_v), \lambda_v)} I_vwI_v \right).$$
\begin{defn}
The \emph{global admissible union associated with $V_G$} is the fpqc quotient
$$\widetilde{Q}^{\ad}_X := L^+_{\mathcal{B}^{\ad}_X}(\widetilde{V}^L_G/A_G)/L^+_{\mathcal{B}^{\ad}_X}G.$$
Its fiber over $b \in \mathcal{B}^{\ad}_X(k)$ is isomorphic to 
$$L^+_bG(k) \times_{I_b(k)} \prod_{v \in X(k)} I^{\leq \lambda_v}_{\ad} := L^+_bG(k)  \times_{I_b(k)} \prod_{v \in X(k)} \left( \bigcup_{w \in \Adm(\lambda_v)} I_{\ad, v}wI_{\ad, v} \right).$$
\end{defn}
Replacing $V^L_G$ with $V^{0, L}_G$ in \eqref{globalgroth} gives rise to the fpqc quotient 
$$\widetilde{Q}^{\ad, 0}_X := L^+_{\mathcal{B}^{\ad}_X}(\widetilde{V}^{0,L}_G/A_G)/L^+_{\mathcal{B}^{\ad}_X}G,$$
whose fiber over $b \in \mathcal{B}^{\ad}_X(k)$ is isomorphic to 
$$L^+_bG  \times_{I_b(k)} \prod_{v \in X(k)} I^{\lambda_v}_{\ad} := L^+_bG  \times_{I_b(k)} \prod_{v \in X(k)} \left( \bigcup_{w \in W} I_{\ad, v}t_v^{w(\lambda_v)}I_{\ad, v} \right).$$
In particular, if $\lambda_v = 0$ for some $v$, then $I^{\leq \lambda_v}_{\ad} = I^{\lambda_v}_{\ad} = I_{\ad, v}$. Let the pullback of $\widetilde{Q}^{\ad}_X \rightarrow \mathcal{B}^{\ad}_X$ to $\mathcal{B}_X$ be $\widetilde{Q}_X.$ The arc space $L^+_{\mathcal{B}_X} G$ acts on $\widetilde{Q}_X$ by left-translation.
\begin{defn} \label{globparahecke}
The pro-algebraic stack $[\widetilde{Q}_X]_G := [L^+_{\mathcal{B}_X} G \setminus\widetilde{Q}_X]$ is called  the \emph{global parabolic Hecke stack}. For a positive integer $N$, the stack $[\widetilde{Q}_X]_{G, N} := [L^+_{\mathcal{B}_{X, N}}G \setminus \widetilde{Q}_X]$ is called the \emph{$N$-truncated global parabolic Hecke stack}.
\end{defn}
More concretely, over the point $b \in \mathcal{B}_X(k)$ with boundary divisor $\lambda_b = \sum_{v \in X(k)} \lambda_v v,$
the fiber of $[\widetilde{Q}_X]_G$ is isomorphic to 
\begin{align*}
\left[L^+_b G \setminus \left(L^+_b G \times_{I_b} \prod_{v \in X(k)} I^{\leq \lambda_{\ad, v}}_{\ad, v}\right)\right] &= \left[ I_b \setminus \prod_{v \in X(k)} I_{\ad, v}^{\leq \lambda_{\ad, v}} \right]=  \prod_{v \in X(k)}\left[I_v \setminus I_{\ad, v}^{\leq \lambda_{\ad, v}} \right].
\end{align*}
\subsection{Local model of singularities}
The total parabolic multiplicative Hitchin stack $\mathcal{M}^{\para}_X$ is not smooth. The goal of this section is to construct its local model of singularities and show that the local evaluation map from $\mathcal{M}^{\para}_X$ to these models are (formally) smooth. 

We will first use cotangent complexes to compute the complex $\mathcal{T}^\bullet$ that controls the deformation of points in $\mathcal{M}_X$. By incorporating the parabolic structure at a point, we can compute the complex $\mathcal{K}^\bullet$ that controls the deformation of points in $\mathcal{M}^{\para}_X$. There is a distinguished triangle relating $\mathcal{T}^\bullet$ to $\mathcal{K}^\bullet$ and a long exact sequence relating their associated cohomology groups. If we can show that a certain cohomology group $H^1(X, H^0(\mathcal{K}^\bullet))$ is zero, then the global obstruction to deformation would be completely determined by its local obstructions. Using spectral sequences, we will see that $H^1(X, H^0(\mathcal{K}^\bullet)) = 0$ if $H^1(X, H^0(\mathcal{T}^\bullet))= 0$. This serves as a condition for the (formal) smoothness of the local evaluation map. Once we show that $\mathcal{M}^{\para}_X$ surjectively maps to the global admissible union $[\widetilde{Q}_X]_G$ on the level of tangent spaces, we would have all of the ingredients required to prove the main result, Theorem \ref{localmodel}.
\subsubsection{Cotangent complexes}
Let $L$ be a $T$-torsor. The map $\chi: V^L_G \rightarrow [V^L_G/G]$ is a $G$-torsor and it gives rise to a distinguished triangle of cotangent complexes:
$$\chi^*\mathbb{L}_{[V^L_G/G]} \rightarrow \mathbb{L}_{V^L_G} \rightarrow \mathbb{L}_{V^L_G/[V^L_G/G]} \rightarrow.$$
Since $\chi$ is a $G$-torsor, $\mathbb{L}_{V^L_G/[V^L_G/G]} \cong \mathcal{O}_{V^L_G} \otimes \mathfrak{g}^*[0]$ and $\mathbb{L}_{V^L_G}$ is $G$-equivariant. Descending to the base $[V^L_G/G]$ gives
$$\mathbb{L}_{[V^L_G/G]} \rightarrow (\chi_*\mathbb{L}_{V_G})^G \rightarrow {V^L_G} \times^G \mathfrak{g}^*[0].$$
Let $m$ be an $X$-point of $[V^L_G/G]$, which comes with a $G$-torsor $E$ over $X$ and a $G$-equivariant map $m': E \rightarrow V^L_G$. Applying $m^*$ to the distinguished triangle above gives 
$$m^*\mathbb{L}_{[V^L_G/G]} \rightarrow m^*(\chi_*\mathbb{L}_{V^L_G})^G = (m^*\chi_*\mathbb{L}_{V^L_G})^G \rightarrow m^*({V^L_G} \times^G \mathfrak{g}^*[0]).$$
In order to compute $m^*\mathbb{L}_{[V^L_G/G]}$, we will first compute $m^*(\chi_*\mathbb{L}_{V^L_G})^G$. Consider the commutative diagram:
\begin{center} 
\begin{tikzpicture}[node distance=2.5cm]
\node (A) {$E$};
\node (B) [right of = A] {$V^L_G$};
\node (C) [below of = A] {$X$};
\node (D) [right of = C] {$[V^L_G/G]$};
\node (E) [left of = A] {$E \times V^L_G$};
\node (F) [below of = E] {$E \times^G V^L_G$};
\node (G) [left of = E] {$V^L_G$};
\draw[->, above] (A) to node {$m'$} (B);
\draw[->, right] (B) to node {$\chi$} (D);
\draw[->, above] (C) to node {$m$} (D);
\draw[->, left] (A) to node {$\chi'$} (C);
\draw[->, above] (E) to node {$p_E$} (A);
\draw[->, left] (E) to node {$\chi''$} (F);
\draw[->] (F) to node {} (C);
\draw[->, above] (E) to node {$p_{V^L_G}$} (G);
\end{tikzpicture}
\end{center}
Here, $p_{V^L_G}$ and $p_E$ are the projections from $E \times V^L_G$ to $V^L_G$ and $E$ respectively. There is a map $\phi': E \rightarrow E \times V^L_G$ that sends $e \in E$ to the graph of $m'$, i.e. $\phi'(e) = (e, m'(e)).$ Since it is $G$-equivariant, it descends to a section $\phi:X \rightarrow E \times^G V^L_G.$ 
Since the right square is commutative, 
$$(m^*\chi_* \mathbb{L}_{V^L_G})^G = (\chi'_* (m')^* \mathbb{L}_{V^L_G})^G.$$ 
Since $m$ is $G$-equivariant, $m^*$ commutes with taking $G$-fixed points and
$$m^*(\chi_*\mathbb{L}_{V^L_G})^G = (m^*\chi_* \mathbb{L}_{V^L_G})^G = (\chi'_* (m')^* \mathbb{L}_{V^L_G})^G.$$ 
For the right hand side, it follows that $(m')^*\mathbb{L}_{V^L_G} = (\phi')^*(p^*_{V^L_G} \mathbb{L}_{V^L_G})$ by definition. Applying $\chi'_*$ to both sides of this yields
$$\chi'_*(m')^*\mathbb{L}_{V^L_G} = \chi'_*(\phi')^*(p^*_{V^L_G} \mathbb{L}_{V^L_G}) = \phi^*\chi''_*(p_{V^L_G}^*\mathbb{L}_{V^L_G}) = \phi^*\chi''_*(\mathbb{L}_{E \times V^L_G/E})$$
because $p_{V^L_G}^*\mathbb{L}_{V^L_G} = \mathbb{L}_{E \times V^L_G/E}$. Taking the $G$-fixed points gives
$$(\chi'_* (m')^* \mathbb{L}_{V^L_G})^G = (\phi^* \chi''_*(p_{V^L_G}^*\mathbb{L}_{V^L_G}))^G = (\phi^* \chi''_*(\mathbb{L}_{E \times V^L_G/E}))^G = \phi^*\mathbb{L}_{E \times^G V^L_G/X}.$$
So, $m^*\mathbb{L}_{[V^L_G/G]}$ is isomorphic to the cone
$$\mathrm{cone}((\phi^* \mathbb{L}_{E \times^G V^L_G/X} \rightarrow \ad(E)^*[0])[-1]).$$
Let $\mathcal{T}^\bullet = \underline{\RHom}_X(m^*\mathbb{L}_{[V^L_G/G]}, \mathcal{O}_X).$ Then, it is isomorphic to the cone 
$$\mathrm{cone}(\ad(E)[0] \rightarrow \underline{\RHom}_X(\phi^* \mathbb{L}_{E \times^G V^L_G}, \mathcal{O}_X)).$$
Since $\mathbb{L}_{V^L_G}$ is supported on degrees $(-\infty, 0]$, so is $\phi^* \mathbb{L}_{E \times^G V^L_G}.$ Since $\underline{\Hom}_X(\cdot, \mathcal{O}_X)$ is left exact, $\underline{\RHom}_X(\phi^* \mathbb{L}_{E \times^G V^L_G}, \mathcal{O}_X)$ is supported on $[0, \infty)$. Therefore, $\mathcal{T}^\bullet$ is the complex
$$\ad(E)[1] \rightarrow \underline{\RHom}_X(\phi^* \mathbb{L}_{E \times^G V^L_G} \mathcal{O}_X),$$
which is supported on $[-1, \infty).$ The deformation of $m$ in the mapping stack $\underline{\Hom}(X, [V^L_G/G])$ is controlled by $\mathcal{T}^\bullet.$ The obstruction space is $H^1(X, \mathcal{T}^\bullet),$ the tangent space of $\mathcal{M}_X/\mathcal{B}_X$ is $H^0(X, \mathcal{T}^\bullet)$ and the infinitesimal automorphism group is $H^{-1}(X, \mathcal{T}^\bullet).$

\subsubsection{Incorporating the parabolic structure}
In order to add a parabolic structure at a point, we need to consider the following commutative diagrams. 

\begin{center} 
\begin{tikzpicture}[node distance=2.5cm]
\node (A) {$\Spec(k)$};
\node (B) [right of = A] {$[V^L_B/B]$};
\node (C) [below of = A] {$X$};
\node (D) [right of = C] {$[V^L_G/G]$};
\node (E) [right of = B] {$\Spec(k)$};
\node (F) [right of = E] {$[V^L_B/B]$};
\node (G) [below of = E] {};
\node (H) [right of = G] {$X$};
\draw[->, above] (A) to node {$m_x$} (B);
\draw[->, right] (B) to node {$\pi$} (D);
\draw[->, above] (C) to node {$m$} (D);
\draw[right hook ->] (A) to node {} (C);
\draw[->, above] (E) to node {$m_x$} (F);
\draw[->, above] (F) to node {} (H);
\draw[->, above right] (E) to node {$x$} (H);
\end{tikzpicture}
\end{center}
By the same computations as in the previous section, the deformation of $m_x$ is controlled by 
$$\mathcal{T}^\bullet_x = \underline{\RHom}_{X}(m^*_x\mathbb{L}_{[V^L_B/B]}, \mathcal{O}_x) = [\ad(E^B_x) \rightarrow \underline{\RHom}_x(\phi^*_x\Omega^1_{E^B_x \times^B V^L_B}, \mathcal{O}_x)],.$$
Let $d\pi^*: \pi^*\mathbb{L}_{[V^L_G/G]} \rightarrow \mathbb{L}_{[V^L_B/B]}$ be the map between cotangent complexes induced by $\pi$ and let $\Gamma(x)$ be the graph of $x: \Spec(R) \rightarrow X.$ Then, have the following commutative diagram
\begin{center} 
\begin{tikzpicture}
\node (A) at (0, 0) {$\Spec(R)$};
\node (B) at (4, 0) {$[V_B^L/B]$};
\node (C) at (0, -2.5) {$\Spec(R) \times_{\Spec(k)} X$};
\node (D) at (4, - 2.5) {$[V_G^L/G]$};
\node (E) at (-3, -2.5) {$\Gamma(x)$};
\draw[->, above] (A) to node {$m_x$} (B);
\draw[->, right] (B) to node {$\pi$} (D);
\draw[->, above] (C) to node {$m$} (D);
\draw[right hook ->, left] (A) to node {$(\id, x)$} (C);
\draw[right hook ->, above] (E) to node {$i$} (C);
\end{tikzpicture}
\end{center}
Since it commutes, $m^*_x \pi^* = i^*m^*$. So, the deformations of points $(x, E, \phi, E^B_x) \in \mathcal{M}^{\para}_X$ are controlled by the complex $\mathcal{K}^\bullet$,
 which fits into the exact triangle 
\begin{align*}
\mathcal{K}^\bullet &\rightarrow \underline{\RHom}_X(m^*\mathbb{L}_{[V^L_G/G]}, \mathcal{O}_X) \oplus \underline{\RHom}_x(m_x^* \mathbb{L}_{[V^L_B/B]}, \mathcal{O}_x) \\
&\xrightarrow{(i^*, -\eta^*)} \underline{\RHom}_x(i^*m^*\mathbb{L}_{[V^L_G/G]}, \mathcal{O}_x) = \underline{\RHom}_x(m_x^*\pi^*\mathbb{L}_{[V^L_G/G]}, \mathcal{O}_x) \rightarrow \mathcal{K}^\bullet[1]
\end{align*}
where $\eta^*$ is the map defined as 
\begin{align*}
\eta^*: \underline{\RHom}_x(m_x^* \mathbb{L}_{[V^L_B/B]}, \mathcal{O}_x) &\rightarrow \underline{\RHom}_x(m_x^*\pi^*\mathbb{L}_{[V^L_G/G]}, I) \\
f &\mapsto f \circ m^*_x d\pi^*.
\end{align*}
We will express this exact triangle as 
\begin{equation}\label{exacttri}
    \mathcal{K}^\bullet \rightarrow \mathcal{T}^\bullet \oplus \mathcal{T}^\bullet_x \rightarrow i^*\mathcal{T}^\bullet \rightarrow \mathcal{K}^\bullet[1].
\end{equation}
Similar to $\mathcal{T}^\bullet$, the obstruction space is $H^1(X, \mathcal{K}^\bullet),$ the tangent space $\mathcal{M}^{\para}_X/\mathcal{B}_X$ is $H^0(X, \mathcal{K}^\bullet)$ and the infinitesimal automorphism group is $H^{-1}(X, \mathcal{K}^\bullet).$ Consider the following two-step truncations of $\mathcal{T}^\bullet$, $\mathcal{T}^\bullet_x$ and $i^*\mathcal{T}^\bullet:$
\begin{align*}
\mathcal{T}^{\leq 0} &= [\ad(E) \xrightarrow{D_{\ad}} \underline{\Hom}_X(\phi^*\Omega^1_{E \times^G V^L_G}, \mathcal{O}_X)],
\\
\mathcal{T}^{\leq 0}_x &= [\ad(E^B_x) \rightarrow \underline{\Hom}_x(\phi^*_x\Omega^1_{E^B_x \times^B V^L_B}, \mathcal{O}_x)],\\
i^*\mathcal{T}^{\leq 0} &= [i^*\ad(E) \rightarrow i^*\underline{\Hom}_X(\phi^*\Omega^1_{E \times^G V^L_G}, \mathcal{O}_X)].
\end{align*}
Since $\eqref{exacttri}$ is exact, $i^*\mathcal{T}^\bullet$ is homotopy equivalent to the cone of the map $\mathcal{K}^\bullet \rightarrow \mathcal{T}^\bullet \oplus \mathcal{T}^\bullet_x.$ Hence, we can work out that a two-step truncation $\mathcal{K}^{\leq 0}$ of $\mathcal{K}^\bullet$ is equivalent to $\mathcal{K}^{\leq 0} = [\mathcal{F} \rightarrow \mathcal{G}],$
where $\mathcal{F} = \ker(\ad(E) \twoheadrightarrow i_*(i^*\ad(E)/\ad(E^B_x))$ is in degree -1 and $$\mathcal{G} = \ker((\phi^*\Omega^1_{E \times^G V^L_G})^* \twoheadrightarrow i_*(i^*(\phi^*\Omega^1_{E \times^G V^L_G})^*/(\phi^*_x\Omega^1_{E^B_x \times^B V^L_B})^*)$$ is in degree 0. As subsheaves of locally free sheaves, $\mathcal{F}$ and $\mathcal{G}$ are locally free over $\mathcal{O}_X.$

\subsubsection{Computing cohomology groups} \label{subsec:spectral}
Let $\mathcal{X} = \{X_i\}_{i\in I}$ be a finite Zariski open affine covering of $X$. For each $q \geq 0,$ we can form a \v{C}ech double complex $C^\bullet(\mathcal{X}, H^q(\mathcal{T}^\bullet))$ and a spectral sequence $$E^{p, q}_2 = H^p(\mathrm{Tot}(C^\bullet(\mathcal{X}, H^q(\mathcal{T}^\bullet)))) \Longrightarrow H^{p+q}(X, \mathcal{T}^\bullet).$$ 
Since the intersections of the open affine subsets $X_i$ are affine and $\mathcal{T}^i$ is quasi-coherent for any $i$, the double complex $C^\bullet(\mathcal{X}, H^q(\mathcal{T}^\bullet))$ vanishes when $q > 0.$ Thus, the spectral sequence degenerates at $E_2$ and 
$$H^p(X, \mathcal{T}^\bullet) \cong H^p(\mathrm{Tot}(C^\bullet(\mathcal{X}, H^0(\mathcal{T}^\bullet)))).$$
We have the following two facts regarding $\mathcal{T}^\bullet$:
\begin{itemize}
\item Since $X$ is a curve, $H^i(X, H^j(\mathcal{T}^\bullet)) = 0$ for $i \geq 2$ and for all $j$.
\item Since $H^j(\mathcal{T}^\bullet)$ for $j \geq 1$ is supported on finitely many points, $H^1(X, H^j(\mathcal{T}^\bullet)) = 0$ for $j \geq 1.$
\end{itemize}
These show that the $E_2$-page of the spectral sequence is given by two non-zero rows:
\begin{center}
\begin{tikzpicture}
\node (A) at (0, 0) {$H^1(X, H^{-1}(\mathcal{T}^\bullet))$};
\node (B) at (4, 0) {$H^1(X, H^{0}(\mathcal{T}^\bullet))$};
\node (C) at (7, 0) {$0$};
\node (D) at (10, 0) {$0$};
\node (E) at (13, 0) {$\cdots$};
\node (F) at (0, -1) {$H^0(X, H^{-1}(\mathcal{T}^\bullet))$};    
\node (G) at (4, -1) {$H^0(X, H^{0}(\mathcal{T}^\bullet))$};
\node (H) at (7, -1) {$H^0(X, H^{1}(\mathcal{T}^\bullet))$};
\node (I) at (10, -1) {$H^0(X, H^{2}(\mathcal{T}^\bullet))$};
\node (J) at (13, -1) {$\cdots$};
\end{tikzpicture}
\end{center}
Therefore, the spectral sequence for $\mathcal{T}^\bullet$ degenerates at the second page, leaving us with isomorphisms
$$H^i(X, \mathcal{T}^\bullet) \cong H^0(X, H^i(\mathcal{T}^\bullet)), \quad i = -1 \text{ or } i \geq 2,$$
as well as exact sequences 
\begin{align*}
    0 \rightarrow H^1(X, H^{-1}(\mathcal{T}^\bullet)) &\rightarrow H^0(X, \mathcal{T}^\bullet) \rightarrow H^0(X, H^0(\mathcal{T}^\bullet)) \rightarrow 0, \\
    0 \rightarrow H^1(X, H^0(\mathcal{T}^\bullet)) &\rightarrow H^1(X, \mathcal{T}^\bullet) \rightarrow H^0(X, H^1(\mathcal{T}^\bullet)) \rightarrow 0.
\end{align*}
The same holds for $\mathcal{K}^\bullet.$

Since $x$ is a point, $H^i(\Gamma(x), H^j(\mathcal{T}^\bullet_x)) = 0$ for $i \geq 1$ and for all $j$. So, the second page of the spectral sequence for $\mathcal{T}^\bullet_x$ only has one non-zero row:
\begin{center} 
\begin{tikzpicture}
\node (A) at (0,0) {$H^0(\Gamma(x), H^{-1}(\mathcal{T}^\bullet_x))$};
\node (B) at (3.5, 0) {$H^0(\Gamma(x), H^0(\mathcal{T}^\bullet_x))$};
\node (C) at (7, 0) {$H^0(\Gamma(x), H^1(\mathcal{T}^\bullet_x))$};
\node (D) at (10.5, 0) {$H^0(\Gamma(x), H^2(\mathcal{T}^\bullet_x))$};
\node (E) at (13.5, 0) {$\cdots$};
\end{tikzpicture}
\end{center}
Therefore, the spectral sequence degenerates at the second page and we have canonical isomorphisms
$$H^i(\Gamma(x), \mathcal{T}^\bullet_x) \cong H^0(\Gamma(x), H^i(\mathcal{T}^\bullet_x)), \quad \text{ for all }i.$$
The same holds for $i^*\mathcal{T}^\bullet$.
\subsubsection{Obstruction spaces}
Let $X_v$ be the formal disk at a point $v \in X$. By \cite[\S 7.1.3]{Wang23}, there is an injective map 
$$H^0(X, H^1(\mathcal{T}^\bullet)) \hookrightarrow \prod_v H^0(X_v, H^1(\mathcal{T}^\bullet|_{X_v}),$$
where $v$ ranges over points on $X$ such that $m(v)$ is singular in $[V_G/G].$ The right hand side is the local obstruction space. The same holds when $\mathcal{T}^\bullet$ is replaced by $\mathcal{K}^\bullet.$ 
The obstruction to deforming points $(x, E, \phi, E^B_x)$ in $\mathcal{M}^{\para}_X$ is controlled by the complex $\mathcal{K}^\bullet.$ So if $H^1(X, H^0(\mathcal{K}^\bullet)) = 0,$ then the earlier exact sequence 
$$0 \rightarrow H^1(X, H^0(\mathcal{K}^\bullet)) \rightarrow H^1(X, \mathcal{K}^\bullet) \rightarrow H^0(X, H^1(\mathcal{K}^\bullet)) \rightarrow 0$$
tells us that 
$$H^1(X, \mathcal{K}^\bullet)) \cong H^0(X, H^1(\mathcal{K}^\bullet)) \hookrightarrow \prod_v H^0(X_v, H^1(\mathcal{K}^\bullet|_{X_v}).$$
In other words, it tells us that the global obstruction is completely determined by its local obstructions. We will now show that if $H^1(X, H^0(\mathcal{T}^\bullet) = 0,$ then $H^1(X, H^0(\mathcal{K}^\bullet)) = 0$. 

To compute
$H^1(X, H^0(\mathcal{K}^\bullet))$, we would like to compute $$H^1(X, H^0(\mathcal{K}^{\leq 0})) = H^1(X, \mathrm{coker}(\mathcal{F} \rightarrow \mathcal{G})).$$ 
The sheaf $\mathcal{F}$ is a subsheaf of $\ad(E)(-x_0)$ which contains sections that have a zero at $x_0$. Similarly, $\mathcal{G}$ is a subsheaf of $(\phi^*\Omega^1_{E \times^G V^L_G})^*(-x_0).$ Thus, 
$\mathcal{F}^*$ and $\mathcal{G}^*$ are subsheaves of $\ad(E)^*(x_0)$ and $(\phi^* \Omega^1_{E \times^G V^L_G})(x_0)$ respectively. By Serre Duality,
\begin{align*}
H^1(X, \mathrm{coker}(\mathcal{F} \rightarrow \mathcal{G}))
&=  H^0(X, \ker(\mathcal{G}^* \rightarrow \mathcal{F}^*) \otimes_{O_X} K_X)^*\\
&\subseteq H^0(X, \ker((\phi^* \Omega^1_{E \times^G V^L_G})(x_0) \rightarrow \ad(E)^*(x_0)) \otimes_{O_X} K_X)^* \\
&= H^1(X, \mathrm{coker}(\ad(E)(-x_0) \rightarrow (\phi^* \Omega^1_{E \times^G V^L_G})^*(-x_0)) \\
&= H^1(X, (\Lie(I^{\mathrm{sm}}_{(E, \varphi)})/\mathfrak{z}_G)^*(-x_0))
\end{align*}
because $H^1(X, \mathrm{coker}(D_{\ad})) = H^1(X, (\Lie(I^{\mathrm{sm}}_{(E, \varphi)})/\mathfrak{z}_G)^*)$ by \cite[\S 7.3.4]{Wang23}. By \cite[Proposition 7.3.6]{Wang23}, $H^1(X, (\Lie(I^{\mathrm{sm}}_{(E, \varphi)})/\mathfrak{z}_G)^*(-x_0)) = 0$ when $L(-x_0)$ is very $(G, \delta_a)$-ample, i.e. 
\begin{align*}
\deg(\omega(L(-x_0))) &> 2r(g-1) + 2 + \delta_a \\
\deg(\omega(L)) &> 2r(g-1) + 2 + \delta_a + \deg(\omega(\mathcal{O}(x_0))) \\
\deg(\omega(L)) &> 2r(g-1) + 2 + \delta_a + r.
\end{align*}
In other words, $H^1(X, H^0(\mathcal{K}^{\leq 0})) = 0$ when $L$ is very $(G, \delta_a + r)$-ample.
\subsubsection{Tangent spaces}
Applying the right derived functor $R\Gamma$ to \eqref{exacttri} gives a long exact sequence: 
\begin{align*} \cdots \rightarrow H^{-1}(\Gamma(x), i^*\mathcal{T}^\bullet) \rightarrow H^0(X, \mathcal{K}^\bullet) &\rightarrow H^0(X, \mathcal{T}^\bullet) \oplus H^0(\Gamma(x), \mathcal{T}^\bullet_x) \rightarrow \rightarrow H^0(\Gamma(x), i^*\mathcal{T}^\bullet) \rightarrow \cdots
\end{align*}
Restricting to formal discs $X_v$ where $v$ ranges over points in $X$ that are not sent to $[V_G^0/G \times T]$ gives a commutative diagram with exact rows:
\begin{align*}
\begin{tikzpicture}
\node (A) at (-1,0) {$H^{-1}(\Gamma(x), i^*\mathcal{T}^\bullet)$};
\node (B) at (2, 0) {$H^0(X, \mathcal{K}^\bullet)$};
\node (C) at (6, 0) {$H^0(X, \mathcal{T}^\bullet) \oplus H^0(\Gamma(x), \mathcal{T}^\bullet_x)$};
\node (D) at (10.5, 0) {$H^0(\Gamma(x), i^*\mathcal{T}^\bullet)$};
\node (F) at (-1, -1.5) {$H^{-1}(U_{x, v}, i^*\mathcal{T}^\bullet)$};
\node (G) at (2, -1.5) {$H^0(X_v, \mathcal{K}^\bullet)$};
\node (H) at (6, -1.5) {$H^0(X_v, \mathcal{T}^\bullet) \oplus H^0(U_{x, v}, \mathcal{T}^\bullet_x)$};
\node (I) at (10.5, -1.5) {$H^0(U_{x, v}, i^*\mathcal{T}^\bullet)$};
\draw[->, right] (B) to node {$q$} (G);
\draw[->, right] (C) to node {$p$} (H);
\draw[->, right] (D) to node {$h$}(I);
\draw[->, above] (A) to node {$a$} (B);
\draw[->, above] (B) to node {$b$} (C);
\draw[->, above] (C) to node {$c$}(D);
\draw[->, above] (F) to node {$d$} (G);
\draw[->, above] (G) to node {$e$} (H);
\draw[->, above] (H) to node {$f$}(I);
\draw[->, right] (A) to node {$g$} (F);
\end{tikzpicture}
\end{align*}
where $U_{x, v} = \Gamma(x) \cap X_v$. The goal is to show that $q$ is surjective. 

If $U_{x,v}$ is empty, then the cohomology groups over $U_{x, v}$ are 0. This means that 
$$H^0(X_v, \mathcal{K}^\bullet) \cong H^0(X_v, \mathcal{T}^\bullet) \oplus H^0(U_{x, v}, \mathcal{T}^\bullet_x).$$
Since the restriction $p: H^0(X, \mathcal{T}^\bullet)  \rightarrow H^0(X_v, \mathcal{T}^\bullet)$ is surjective by \cite[\S 7.3.7]{Wang23}, $q$ is surjective. On the other hand, if $U_{x, v}$ is non-empty, then $U_{x, v} = \Gamma(x)$ because $\Gamma(x) \subseteq X_v$. In this case, $g$ and $h$ are equalities and $p$ is surjective. By the Four lemma, $g$ and $p$ being surjective and $h$ being injective implies that $q$ is surjective. 

So in both cases, the restriction
\begin{equation} \label{surjectivetangents} q:T_m(\mathcal{M}^{\para}_X/\mathcal{B}_X) = H^0(X, \mathcal{K}^\bullet) \rightarrow  T_{\eva(m)}([\widetilde{Q}_X/\mathcal{B}_X])  = H^0(X_v, \mathcal{K}^\bullet)
\end{equation}
is surjective.
\subsubsection{Parabolic local models}
We are now prepared to prove the main theorem of this section.
\begin{thrm} \label{localmodel} Let $m = (L, x, E, \varphi, E^B_x) \in \mathcal{M}^{\para, \heart}_X.$ If $H^1(X, (\Lie(I^{\mathrm{sm}}_{(E, \varphi)})/\mathfrak{z}_G)^*(-x_0)) = 0$, then the local evaluation map
$$\eva^{\para}: \mathcal{M}^{\para, \heart}_X \rightarrow [\widetilde{Q}_X]_G,$$
is formally smooth at $m$. Meanwhile, 
$\eva^{\para}_N: \mathcal{M}^{\para, \heart}_X \rightarrow [\widetilde{Q}_X]_{G,N}$
is smooth at $m$. This is true in particular when $L$ is very $(G, \delta_a + r)$-ample, where $a = h_X(m).$
\end{thrm}
\begin{proof}
Let $I \subseteq R$ be a square-zero ideal of an Artinian local $k$-algebra $R$. Let $R \rightarrow R/I$ be a small extension of Artinian $k$-algebras. Fix $m_0 \in \mathcal{M}^{\heart, \para}_X(R/I)$ and let $\overline{m}_0$ be its image in $[\widetilde{Q}_X].$  Suppose that $\overline{m}_0$ lifts to $\overline{m}$ over $R$. Then, the local obstruction of deforming $\overline{m}$ vanishes. If $H^1(X, (\Lie(I^{\mathrm{sm}}_{(E, \varphi)})/\mathfrak{z}_G)^*(-x_0)) = 0$, then $H^1(X, H^0(\mathcal{K}^\bullet)) = 0$ by \S \ref{subsec:spectral}.  This means that the global obstruction of deforming $m_0$ also vanishes. By the surjectivity of $q$ in \eqref{surjectivetangents}, there is a lift of $m_0$ to $R$ lying over $\overline{m}.$ This gives us the desired theorem.
\end{proof}
Here, the local evaluation map sends a point $(x, E, \varphi, E^B_x) \in \mathcal{M}^{\para}_X$ with boundary divisor $\lambda = \sum_{v \in X(k)} \lambda_v v$ to the product $(\gamma_v)_{v \in X(k)} \in \prod_{v \in X(k)} I^{\leq \lambda_v}_{\ad}$ of trivializations of $\varphi$ over formal neighbourhoods of $v \in X(k)$. By construction, the local evaluation map $\eva^{\para}_N: \mathcal{M}^{\para, \heart}_X \rightarrow [\widetilde{Q}_X]_G$ factors through the limit $\eva_N: \mathcal{M}^{\heart}_X \rightarrow [\widetilde{Q}_X]_G$.

\section{Dimension formula and equidimensionality}
Compared to the linear case, the strategy for proving the equidimensionality of multiplicative affine Springer fibers is very different. Indeed, \cite{KL88} first proved the equidimensionality of parabolic affine Springer fibers using ``lines of type $\alpha$", where $\alpha$ is a simple root. Then, \cite{Bez96} used this result to conclude the equidimensionality of affine Springer fibers and subsequently, found their dimensions. For multiplicative affine Springer fibers, this is the reverse. \cite{Chi22} first found the dimensions of affine Springer fibers before proving their equidimensionality. He also resorted to global methods such as using the Product formula and the mutliplicative Hitchin fibration. 

Unfortunately, the notion of ``lines of type $\alpha$" does not generalize readily when passing to monoids \cite{BC18}. Instead, we will use the same strategy as \cite[\S 3, \S 5]{Chi22}. We will first compute the dimension of parabolic multiplicative affine Springer fibers for unramified, regular semisimple elements. This will involve computing the dimension of generalized Mirkovi\'{c}-Vilonen cycles in the affine flag variety that are indexed by two dominant coweights. We will then prove the dimension formula and equidimensionality of parabolic multiplicative affine Springer fibers for general regular semisimple elements using the Product formula and local constancy of the fibers.
\subsection{The case for unramified elements}
Let $\gamma \in G^{\rs}(F)$ be a split or unramified element, i.e. we can assume that $\gamma \in t^\mu T(\mathcal{O}) \cap G^{\rs}(F)$ with $\mu = \nu_\gamma \in X_*(T)^+$ after conjugation. Suppose that $\mu \leq \lambda$. We wish to find 
\begin{align*}
\dim X^{\lambda, \para}_\gamma &= \dim \left(\left\{ g \in \Fl_G: \Ad_{g^{-1}}(\gamma) \in \bigcup_{w \in W} It^{w(\lambda)}I\right\}\right) \\
&= \dim \left(\left\{ g \in \Fl_G: \Ad_{g^{-1}}(t^\mu) \in \bigcup_{w \in W} It^{w(\lambda)}I\right\}\right)\\
&= \dim \left(\bigcup_{w\in W} \{ g \in \Fl_G: \Ad_{g^{-1}}(t^\mu) \in It^{w(\lambda)}I \}\right)\\
&= \max_{w \in W} \dim ( \{ g \in \Fl_G: \Ad_{g^{-1}}(t^\mu) \in It^{w(\lambda)}I \}).
\end{align*}
For a fixed $w \in W$, we denote 
$$X_{t^{w(\lambda)}}(t^\mu) := \{ g \in \Fl_G: \Ad_{g^{-1}}(t^\mu) \in It^{w(\lambda)}I \},$$
which is an example of an \emph{affine Lusztig variety} from \cite{He23}. Let $U$ be the unipotent radical of $B$ with its opposite $U^-$. Since $T(F)$ does not permute the $U^-(F)$-orbits, $U^-(F)vI$ with $v \in \widetilde{W}$, in the affine flag variety transitively, we have the following formula: 
\begin{align*}
\dim X_{t^{w(\lambda)}}(t^\mu) = \sup_{v \in W} \dim(X_{t^{w(\lambda)}}(t^\mu) \cap U^-(F)vI).
\end{align*}
The left multiplication map by $v^{-1}$ gives an isomorphism 
$$X_{t^{w(\lambda)}}(t^\mu) \cap U^-(F)vI \cong X_{t^{w(\lambda)}}(t^{v^{-1}(\mu)}) \cap v^{-1}U^-(F)vI.$$
So, it suffices to describe $\dim X_{t^{w(\lambda)}}(t^\mu) \cap U^-(F)I.$ We will do so by relating them to certain intersections of $I$-orbits and $U^-(F)$-orbits in the affine flag variety. Consider the map
\begin{align*}
f_\mu: U^-(F) &\rightarrow U^-(F), \\
u&\mapsto u^{-1}t^\mu u t^{-\mu}.
\end{align*}
We have that
$$X_{t^{w(\lambda)}}(t^\mu) \cap U^-(F)I \cong \frac{f_{\mu}^{-1}(It^{w(\lambda)}It^{-\mu} \cap U^-(F))}{I \cap U^-(F)}.$$
We also have the isomorphism 
\begin{equation} \label{MVcycle} \frac{It^{w(\lambda)}I \cap U^-(F)t^{\mu}I}{I} \cong \frac{It^{w(\lambda)}It^{-\mu} \cap U^-(F)}{I \cap t^\mu U^-(F)t^{-\mu}}.\end{equation}
A Mirkovi\'{c}-Vilonen cycle in the affine Grassmannian is an irreducible component of the closure of the intersection between a $G(\mathcal{O})$-orbit and a $U^-(F)$-orbit. 
\begin{defn}
Denote the intersection of an Iwahori-orbit and a $U^-(F)$-orbit by  
$$I^w \cap S^v_{w_0} := \frac{IwI \cap U^-(F)vI}{I}, \quad w, v \in \widetilde{W}.$$
A \emph{generalized Mirkovi\'{c}-Vilonen cycle} in the affine flag variety is an irreducible component of $\overline{I^w \cap S^v_{w_0}}.$
\end{defn}
In particular, \eqref{MVcycle} is $I^{t^{w(\lambda)}} \cap S^{t^\mu}_{w_0}$. In summary, we have the diagram:
\begin{equation} \label{diagram}
\begin{tikzpicture}
\node (A) at (0,0) {$f_{\mu}^{-1}(It^{w(\lambda)}It^{-\mu} \cap U^-(F))$};
\node (B) at (8, 0) {$It^{w(\lambda)}It^{-\mu} \cap U^-(F)$};
\node (C) at (0, -2) {$X_{t^{w(\lambda)}}(t^\mu) \cap U^-(F)I = \frac{f_{\mu}^{-1}(It^{w(\lambda)}It^{-\mu} \cap U^-(F))}{I \cap U^-(F)}$};
\node (D) at (8, -2) {$I^{t^{w(\lambda)}} \cap S^{t^\mu}_{w_0} = \frac{It^{w(\lambda)}It^{-\mu} \cap U^-(F)}{I \cap t^\mu U^-(F)t^{-\mu}}$};
\draw[->, above] (A) to node {$f_\mu$} (B);
\draw[->, above] (A) to node {} (C);
\draw[->, right] (B) to node {} (D);
\end{tikzpicture}
\end{equation}
The left vertical map is a $(I \cap U^-(F))$-torsor. Since 
$$I \cap t^\mu U^-(F)t^{-\mu} = t^\mu(I \cap U^-(F))t^{-\mu},$$
the right vertical map is a $(t^\mu(I \cap U^-(F))t^{-\mu})$-torsor.
\subsubsection{Dimensions of generalized Mirkovi\'{c}-Vilonen cycles}
The dimension of generalized Mirkovi\'{c}-Vilonen cycles indexed by two general elements in $\widetilde{W}$ is unknown. However, the situation is easier where the cycles are indexed by two translation elements. For a dominant coweight $\lambda$, let $P_\lambda$ be the parabolic subgroup of $G$ generated by $T$ and the root subgroups $U_\alpha$ such that $\langle \lambda, \alpha \rangle \leq 0$. Let $W_\lambda$ be the Weyl group of $P_\lambda$ with longest element $w^{\lambda}_0$ and let $W^{\lambda} = W/W_{\lambda}.$
\begin{prop} \label{iwahorimvcycle} The intersection $I^{t^{w(\lambda)}} \cap S^{t^\mu}_{w_0}$ in $\Fl_G$ is equidimensional with dimension 
\begin{align*} \dim(I^{t^{w(\lambda)}} \cap S^{t^\mu}_{w_0}) &= \langle \lambda + \mu, \rho \rangle - \dim(G/P_\lambda) + \dim(X_w) \\
&= \langle \lambda + \mu, \rho \rangle - \ell(w^{\lambda}_0) + \ell(w^\lambda),
\end{align*}
where $w^\lambda$ is the image of $w \in W^\lambda.$
\end{prop}
\begin{proof}
This proof is inspired by \cite[Corollary 6.4]{Zho19}. By \cite[Thoerem 6.1]{Zho19}, the intersection $It^{\lambda}G(\mathcal{O}) \cap U^-(F)t^\mu G(\mathcal{O})$ has dimension 
$$\langle \lambda + \mu, \rho \rangle - \dim(G/P_\lambda) + \dim(X_w) = \langle \lambda + \mu, \rho \rangle - \ell(w^{\lambda}_0) + \ell(w^\lambda).$$
The projection $\Fl_G \rightarrow \Gr_G$ is a $G/B$-bundle. For any $x, y \in W$, 
both $t^{w(\lambda)}x$ and $t^\mu y$ are elements of $\widetilde{W}$. So, the intersection $I^{t^{w(\lambda)}x} \cap S^{t^\mu y}_{w_0}$ lies in a $G/B$-bundle above $It^{w(\lambda)}G(\mathcal{O}) \cap U^-(F)t^\mu G(\mathcal{O})$. Thus, 
$$\dim(It^{w(\lambda)}G(\mathcal{O}) \cap U^-(F)t^\mu G(\mathcal{O})) \leq \dim(I^{t^{w(\lambda)}x} \cap S^{t^\mu y}_{w_0}).$$
When $x \geq y,$ the intersection $I^{t^{w(\lambda)}x} \cap S^{t^\mu y}_{w_0}$ in the $G/B$-fiber above any point of $It^{w(\lambda)}G(\mathcal{O}) \cap U^-(F)t^\mu G(\mathcal{O})$ is a subset of the Richardson variety $BxB/B \cap B^-yB/B$, which has dimension $\ell(x) - \ell(y).$ Hence, 
$$\dim(I^{t^{w(\lambda)}x} \cap S^{t^\mu y}_{w_0}) \leq \dim(It^{w(\lambda)}G(\mathcal{O}) \cap U^-(F)t^\mu G(\mathcal{O})) + \ell(x) - \ell(y).$$
However in our case, $x = y = 1$. So we get equalities of dimensions, 
$$\dim(I^{t^{w(\lambda)}} \cap S^{t^\mu}_{w_0}) = \dim(It^{w(\lambda)}G(\mathcal{O}) \cap U^-(F)t^\mu G(\mathcal{O})).$$
These generalized Mirkovi\'{c}-Vilonen cycles are equidimensional because the base $I^{t^{w(\lambda)}} \cap S^{t^\mu}_{w_0}$ is equidimensional \cite[Theorem 6.1]{Zho19}.
\end{proof}
\subsubsection{Admissible subsets}
To relate the dimension of $X_{t^{w(\lambda)}}(t^\mu) \cap U^-(F)I$ with the dimension of $I^{t^{w(\lambda)}} \cap S^{t^\mu}_{w_0}$, we need to compute the dimension of the fibers of the maps in \eqref{diagram}. Doing so requires the dimension theory of admissible subsets \cite[\S 6]{GHKR10}. In this section, we will be proving an Iwahori-analogue of the results in \cite[\S 3.4]{Chi22}. 

Let $P = MN$ be the standard parabolic subgroup, where $N$ is the unipotent radical of $P$ and $M$ is the standard Levi subroup. Let $LN$ be the loop space of $N$, $L^+N$ be the arc space and $L^+_n N$ be the $n$-th jet-scheme. 
Let $\mathcal{G}$ be the Bruhat-Tits parahoric group scheme corresponding to the Iwahori subgroup $I$ such that $\mathcal{G}(\mathcal{O}) = I.$ For $n \geq 0$, let $I_n = \ker(\mathcal{G}(\mathcal{O}) \rightarrow \mathcal{G}(\mathcal{O}/t^n \mathcal{O}))$ be the $n$th principal congruence subgroups of $I$. These are normal subgroups of $I$.

Define the normal subgroup $N_n = I_n \cap N(\mathcal{O}) \subseteq I \cap N(\mathcal{O})$. Choose a Borel subgroup $B$ contained in $P$ and let $\delta_N$ be the sum of the $B$-fundamental coweights $\omega_\alpha$, where $\alpha$ ranges over the simple $B$-positive roots appearing in $\Lie(N).$ Let $N[i]$ be the product of the root subgroups $U_\beta \subset N$ for $\beta$ satisfying $\langle \beta, \delta_N\rangle \geq i.$ These subgroups $N[i]$ are stable under conjugation by $M$ and have abelian successive quotients $N\langle i \rangle := N[i]/N[i+1].$ Define $N_n[i] = N_n \cap N[i](\mathcal{O})$ and $N_n\langle i \rangle = N_n[i]/N_n[i+1].$ 

For each $\gamma \in M(F) \cap G(F)^{\rs},$ the map 
$$f_\gamma: LN \rightarrow LN, \quad u \mapsto u^{-1}\gamma u \gamma^{-1}$$
preserves each normal subgroup $N[i]$ and induces maps $f_\gamma[i]: LN[i] \rightarrow LN[i]$ and $f_\gamma \langle i \rangle: LN\langle i \rangle \rightarrow LN \langle i \rangle.$ There is an $M$-equivariant isomorphism $N\langle i \rangle \cong \Lie N\langle i \rangle.$ For each $i$, let 
\begin{align*} 
r_i &:= \mathrm{val}\det(f_\gamma \langle i \rangle ) = \mathrm{val} \det(\ad_\gamma: \Lie N\langle i \rangle(F) \rightarrow \Lie N\langle i \rangle(F) )\\
r_N(\gamma)&:= \mathrm{val}\det(\ad_\gamma: \Lie N(F) \rightarrow \Lie N(F)).
\end{align*}
Let $\ell = \max_{\alpha \in \Phi} \langle \delta_N, \alpha).$ Then,
$$r_N(\gamma) = \sum^\ell_{i = 1} r_i.$$
\begin{lem} \label{lem1} For any $1 \leq i \leq \ell + 1$ and 
any positive integer $n$ such that $n \geq \sum^{\ell + 1}_{j = 1} r_j,$ we have that $N_n[i] \subset f_\gamma(I \cap N[i](\mathcal{O})).$
\end{lem}
\begin{proof}
We follow \cite[Lemma 3.4.1]{Chi22} and prove this by descending induction on $i$. When $i = \ell + 1,$ we have $N[\ell + 1]$ is the trivial subgroup. Hence, $N_n[\ell + 1] = N_n \cap N[\ell + 1](\mathcal{O})$ is also trivial and it certainly lies in $f_\gamma(I \cap N[\ell+1](\mathcal{O})) = f_\gamma(I \cap N[\ell + 1](\mathcal{O})).$

Assume that the statement is true for $i+1$. Let $x \in N_n[i].$ Define an operation $x \ast u = u^{-1}x\gamma u \gamma^{-1}.$ To show that $x \in f_\gamma(I \cap N[i](\mathcal{O})),$ we need to find $u \in (I \cap N[i](\mathcal{O}))$ with $x \ast u = 1$. Indeed if $x \ast u = 1,$ then $x\gamma u \gamma^{-1} = u$ and 
\begin{align*}
f_\gamma(u^{-1}) &= u\gamma u^{-1} \gamma^{-1} = (x\gamma u \gamma^{-1})\gamma u^{-1} \gamma^{-1} = x.
\end{align*}
Let $x_i \in N_n\langle i \rangle$ be the image of $x$. Since $\mathrm{val} \det(f_\gamma \langle i \rangle) = r_i,$ we have that $t^{r_i} N\langle i \rangle (\mathcal{O}) \subset f_\gamma \langle i \rangle (N\langle i \rangle(\mathcal{O}))$. Let $$I_{N\langle i \rangle} := \frac{N[i] \cap I}{N[i+1]\cap I}.$$
Then, we also have $t^{r_i} (I_{N\langle i \rangle}) \subset f_\gamma \langle i \rangle (I_{N\langle i \rangle})$. There exists $u_i \in N_{n-r_i}[i]$ such that $x_i \ast u_i = 1$ in $I_{N\langle i \rangle}$ and hence, $x \ast u_i \in N_{n-r_i}[i+1]$. By the inductive hypothesis, there exists $v \in I \cap N[i+1](\mathcal{O})$ such that $(x \ast u_i) \ast v = 1.$ Then, $u = u_iv$ satisfies $x \ast u = 1.$
\end{proof}
\begin{defn}
A subset of $L^+N$ is \emph{admissible} if it is the pre-image of a locally closed subset of $L^+_nN$ for some $n$. A subset of $LN$ is \emph{admissible} if it is $G(F)$-conjugate to an admissible subset of $L^+N$.  
\end{defn}
For example when $N = U^-$, the sets $It^{w(\lambda)}It^{-\mu} \cap U^-(F)$ and $I \cap t^\mu U^-(F)t^{-\mu}$ are both admissible subsets of $U^-(F).$
\begin{lem} \label{f0admiss}
Let $V$ be an admissible subset of $L^+N$ and $f_0 = f_\gamma|_{I \cap L^+N}$. Let $n \geq r_N(\gamma)$ such that $V$ is right invariant under $N_n$ and suppose that $V \subset f_0(I \cap L^+N)$. Then, the set $f_0^{-1}(V)$ is admissible and right-invariant under $N_n$ and $f_0$ induces a smooth surjective map $$f_0^{-1}(V)/N_n \rightarrow V/N_n,$$
whose fibers are isomorphic to $\mathbb{A}^{r_N}(\gamma)$. \end{lem}
\begin{proof}
We will essentially follow \cite[Lemma 3.4.2]{Chi22}. Let $\overline{f}_0: I \cap L^+_nN \rightarrow I \cap L^+_nN$ be the map induced by $f_0$. Since $V$ is right invariant under $N_n$, so is $f_0^{-1}(V).$ Since $f_0^{-1}(V)/N_n = \overline{f}_0^{-1}(V/N_n)$ is a locally closed subset of $I \cap L^+_nN \subseteq L^+_nN$, we have that $f_0^{-1}(V)$ is admissible. Since $V \subset f_0(I \cap L^+N)$, the induced map $f_0^{-1}(V)/N_n \rightarrow V/N_n$ is surjective. It remains to show that the fibers are isomorphic to $\mathbb{A}^{r(\gamma)}.$

Let $H = N_n$ and define normal subgroups $H[i]_j = t^jH[i]$ (resp. $H\langle i \rangle_j = t^j H\langle i \rangle$). Define the product $v \ast u = u^{-1}v\gamma u\gamma^{-1}$ for $u, v \in N_n.$ In particular, $\overline{f}_0(u) = 1 \ast u$ and all fibers of $\overline{f}_0$ are isomorphic to the stabilizer $S:= \overline{f}_0^{-1}(1).$ 

The $\ast$-operation descends to $H[i]$ and $H\langle i \rangle$. Let $S[i]$ and $S\langle i\rangle$ be the corresponding stabilizers of 1. In order to describe $S$, we will describe each $S\langle i\rangle$. First, we claim that the homomorphism $S[i] \rightarrow S\langle i\rangle$ is surjective for all $i$. Indeed, let $s \in S\langle i\rangle$ and choose $h$ to be its representative in $H[i].$ Since 
$$S \langle i \rangle = \ker(\overline{f}_0\langle i \rangle) \subset t^{n-r_i}H\langle i \rangle,$$
we have that $h \in H[i]_{n-r_i}$ and $1 \ast h \in H[i+1]_{n-r_i}$. Since $n-r_i \geq \sum^{\ell + 1}_{j = i+1} r_j,$ we can apply the previous lemma to obtain an element $h' \in H[i+1]$ satisfying $1 \ast (hh') = 1$. Thus, $hh' \in S[i]$ maps to $s \in S\langle i \rangle.$ The kernel of this surjective homomorphism is $S[i+1]$ and $$S \langle i \rangle \cong (f_0 \langle i \rangle)^{-1}(t^nI_{N\langle i \rangle})/t^nI_{N\langle i \rangle} \cong \mathbb{A}^{r_i}. $$
From this, we see that $S \cong \mathbb{A}^{r_N(\gamma)}.$
\end{proof}
To extend this result from $f_0$ to $f_\gamma$, we need the following lemma.  
\begin{lem} \label{inward}
For any $n \geq r_N(\gamma)$, we have that $f_{\gamma}^{-1}(N_n) \subset N_{n- r_N(\gamma)}.$
\end{lem}
\begin{proof}
We will prove that if $u \in N(F)$ with $f_\gamma(u) \in N_n,$ then $u \in N[i](F) \cdot N_{n - \sum_{j < i} r_j}$ using induction. When $i = 1,$ the statement says $u \in N[1](F) = N(F),$ which holds by assumption. Suppose that the statement is true for all integers $k$ satisfying $1 \leq k \leq i$. By the inductive hypothesis, we have $u = u_iv$ with $u_i \in N[i](F)$ and $v \in N_{n- \sum_{j < i} r_j}$. By assumption, 
\begin{align*}
 f_\gamma(u) &= f_\gamma(u_iv) \\
 &= v^{-1}u_i^{-1} \gamma u_iv \gamma^{-1} \\
 &= v^{-1}u_i^{-1} \gamma u_i(\gamma^{-1} \gamma) v \gamma^{-1} \\
 &= v^{-1}(u_i^{-1} \gamma u_i\gamma^{-1}) \gamma v \gamma^{-1} \in N_n. 
\end{align*}
Hence, $u_i^{-1}\gamma u_i \gamma^{-1} \in N[i](F) \cap (v N_n \gamma v \gamma^{-1}) \in N_{n - \sum_{j < i} r_j}[i].$ Let $\overline{u}_i \in N\langle i \rangle$ be the image of $u_i$. Then, $f_{\gamma}\langle i \rangle (\overline{u}_i) \in N_{n - \sum_{j < i}r_j}\langle i \rangle.$ Since $\val \det(f_\gamma \langle i \rangle ) = r_i,$ we have that $\overline{u}_i \in N_{n - \sum_{j < i+1} r_j} \langle i \rangle.$ Hence, 
$$u = u_iv \in N[i+1](F) \cdot N_{n - \sum_{j < i+1}r_j}.$$
This proves the statement by induction. Now, if we set $i = \ell + 1$, then the lemma follows because $\sum^{\ell + 1}_{j = 1} r_j = r_N(\gamma)$ and $N[\ell + 1]$ is trivial. 
\end{proof}
\begin{prop} \label{admissiblesubset}
Let $Z$ be an admissible subset of the loop space $LN$. Then, $f^{-1}_\mu(Z)$ is admissible and there exists a positive integer $m$ such that for all $n \geq m$, the subsets $f^{-1}_\mu(Z)$ and $Z$ are right invariant under $N_n$. Moreover, the map
$$f^{-1}_\mu(Z)/N_n \rightarrow Z/N_n$$
induced by $f_\mu$ is smooth and surjective with irreducible geometric fibers of dimension $r_N(\gamma)$
\end{prop}
\begin{proof}
The proof is essentially \cite[Proposition 3.4.4]{Chi22}. Fix a positive integer $n_0 \geq r_N(\gamma)$ and choose $\mu \in X_*(Z(M)^0)$ such that $\Ad_{t^{\mu_0}}(Z) \subset N_{n_0}$. By Lemma \ref{inward}, $$f_{\gamma}^{-1}(\Ad_{t^{\mu_0}}(Z)) \subset f_{\gamma}^{-1}(N_{n_0}) \subset N_{n_0-r_N(\gamma)}.$$
Hence, $\Ad_{t^{\mu_0}}(f^{-1}_\gamma(Z)) = f^{-1}_\gamma(\Ad_{t^{\mu_0}}(Z)) = f_0^{-1}(\Ad_{t^{\mu_0}}(Z)).$ Since $\Ad_{t^{\mu_0}}(Z)$ is an admissible subset of $L^+N,$ we have that $f_0^{-1}(\Ad_{t^{\mu_0}}(Z))$ is an admissible subset by Lemma \ref{f0admiss}.

Let $n_1 > n_0$ be a positive integer such that $\Ad_{t^{\mu_0}}(Z)$ and $f^{-1}_\gamma(\Ad_{t^{\mu_0}}(Z))$ are invariant under $N_{n_1}.$ For all $n \geq n_1$, the subsets $Z$ and $f_{\gamma}^{-1}(Z)$ are right-invariant under the group $t^{-\mu_0} N_n t^{\mu_0}$. Thus, we have the following commutative diagram 
\begin{center} 
\begin{tikzpicture}
\node (A) at (0,0) {$f_\gamma^{-1}(Z)/t^{-\mu_0} N_n t^{\mu_0}$};
\node (B) at (4, 0) {$Z/t^{-\mu_0} N_n t^{\mu_0}$};
\node (C) at (0, -2) {$f_{\gamma}^{-1}(\Ad_{t^{\mu_0}}(Z))/N_n$};
\node (D) at (4, -2) {$\Ad_{t^{\mu_0}}(Z)/N_n$};
\draw[->, above] (A) to node {} (B);
\draw[->, left] (A) to node {$\cong$} (C);
\draw[->, right] (B) to node {$\cong$} (D);
\draw[->, above] (C) to node {} (D);
\end{tikzpicture}
\end{center}
where the vertical arrows are isomorphisms induced by $\Ad_{t^{\mu_0}}.$ By Lemma \ref{lem1}, $\Ad_{t^{\mu_0}}(Z) \subset N_{n_0} \subset f_\gamma(I\cap L^+N).$ By Lemma \ref{f0admiss}, the bottom horizontal arrow is smooth, surjective and has fibers isomorphic to $\mathbb{A}^{r_N(\gamma)}.$

Let $m$ be a positive integer such that
for all $n \geq m,$ $N_n$ contains 
$t^{-\mu_0} N_{n'} t^{\mu_0}$ for some $n' \geq n_1.$ Consider the following diagram 
\begin{center} 
\begin{tikzpicture}
\node (A) at (0,0) {$f_\gamma^{-1}(Z)/t^{-\mu_0} N_{n'} t^{\mu_0}$};
\node (B) at (4, 0) {$Z/t^{-\mu_0} N_{n'} t^{\mu_0}$};
\node (C) at (0, -2) {$f_{\gamma}^{-1}(Z)/N_n$};
\node (D) at (4, -2) {$Z/N_n$};
\draw[->, above] (A) to node {} (B);
\draw[->, left] (A) to node {} (C);
\draw[->, right] (B) to node {} (D);
\draw[->, above] (C) to node {} (D);
\end{tikzpicture}
\end{center}
Both vertical maps are smooth and surjective with fibers isomorphic to the irreducible scheme $N_n/t^{-\mu_0} N_{n'} t^{\mu_0}$ while the top horizontal map is smooth and surjective with fibers $\mathbb{A}^{r_N(\gamma)}.$ Hence, the lower map is also smooth and surjective with fibers $\mathbb{A}^{r_N(\gamma)}$.
\end{proof}
\subsubsection{Dimension formula for unramified elements}
We will now apply the dimension theory of admissible subsets to prove the equidimensionality and dimension formula of parabolic multiplicative affine Springer fibers for unramified elements. In particular, we set $N = U^-$ and define $U^-_n$ similarly to $N_n$ from the previous section.   
\begin{thrm} \label{unramdim}
The parabolic multiplicative affine Springer fiber $X^{\lambda, \para}_{\gamma}$ of an unramified element $\gamma \in G^{\rs}(F)$ is an equidimensional, quasi-projective variety of dimension $$\langle \rho, \lambda \rangle + \frac{1}{2}d(\gamma).$$  
\end{thrm}
\begin{proof}
Suppose that $\gamma \in t^\mu T(\mathcal{O}) \cap G^{\rs}(F)$ with $\mu = \nu_\gamma \in X_*(T)^+$ after conjugation. By Proposition \ref{admissiblesubset}, there exists a sufficiently large positive integer $n$ such that the map
$$(f_{\mu}^{-1}(It^{w(\lambda)}It^{-\mu} \cap U^-(F)))/U^-_n \rightarrow (It^{w(\lambda)}It^{-\mu} \cap U^-(F))/U^-_n$$
has fibers of dimension $r_{U^-}(t^\mu)$. Since $f_\mu^{-1}(Z)$ and $Z$ are right invariant under $U^-_n$, they are also right invariant under $U^-_n.$ Taking the quotient of \eqref{diagram} by $U^-_n$ gives the following:
\begin{equation}  \label{thediagram}
\begin{tikzpicture}
\node (A) at (0,0) {$f_{\mu}^{-1}(It^{w(\lambda)}It^{-\mu} \cap U^-(F))/U^-_n$};
\node (B) at (8, 0) {$(It^{w(\lambda)}It^{-\mu} \cap U^-(F))/U^-_n$};
\node (C) at (0, -2) {$X_{t^{w(\lambda)}}(t^\mu) \cap U^-(F)I$};
\node (D) at (8, -2) {$I^{t^{w(\lambda)}} \cap S^{t^\mu}_{w_0}$};
\draw[->, above] (A) to node {$f_\mu$} (B);
\draw[->, above] (A) to node {} (C);
\draw[->, right] (B) to node {} (D);
\end{tikzpicture}
\end{equation}
The left (resp. right) vertical map is smooth and surjective whose fibers are isomorphic to the  irreducible scheme $(I \cap U^-(F))/U^-_n$ (resp. $(t^\mu(I \cap U^-(F))t^{-\mu}/U^-_n)$). Since $U^-_n$ lies in $t^\mu(I \cap U^-(F))t^{-\mu}$ and $I \cap U^-,$ the schemes $X_{t^{w(\lambda)}}(t^\mu) \cap U^-(F)I$ and $I^{t^{w(\lambda)}} \cap S^{t^\mu}_{w_0}$ remain the same after taking the quotient by $U^-_n.$ By Propositions \ref{iwahorimvcycle} and \ref{admissiblesubset}, we have that 
\begin{align*}
\dim(X_{t^{w(\lambda)}}(t^\mu) \cap U^-(F)I) &= \dim(I^{t^{w(\lambda)}} \cap S^{t^\mu}_{w_0}) + \dim(t^\mu(I \cap U^-(F))t^{-\mu}) + r_{U^-}(t^\mu) \\
&- \dim(I \cap U^-(F)) \\
&= \langle \lambda + \mu, \rho \rangle - \ell(w^{\lambda}_0) + \ell(w^\lambda) + \dim(I \cap U^-(F))  - \langle \mu, 2\rho \rangle \\ &+ r_{U^-}(t^\mu)
- \dim(I \cap U^-(F)) \\
&= \langle \lambda + \mu, \rho \rangle - \ell(w^{\lambda}_0) + \ell(w^\lambda)  - \langle \mu, 2\rho \rangle + \frac{1}{2}d(t^\mu) + \langle \mu , \rho \rangle \\
&= \langle \lambda, \rho \rangle - \ell(w^{\lambda}_0) + \ell(w^\lambda) + \frac{1}{2}d(t^\mu) 
\end{align*}
For $\gamma \in t^\mu T(\mathcal{O}) \cap G^{\rs}(F),$ taking the supremum and maximum gives 
\begin{align*}
\dim X^{\lambda,\para}_\gamma &= \max_{w \in W} \dim X_{t^{w(\lambda)}}(\gamma) \\
&= \max_{w\in W} \sup_{v \in W} \dim(X_{t^{w(\lambda)}}(t^\mu) \cap U^-(F)vI) \\
&= \max_{w \in W} \left(\langle \lambda, \rho \rangle - \ell(w^{\lambda}_0) + \ell(w^\lambda) + \frac{1}{2}d(t^\mu) \right)\\
&=  \langle \lambda, \rho \rangle + \frac{1}{2}d(\gamma).
\end{align*}
Since the cycles $I^{t^{w(\lambda)}} \cap S^{t^\mu}_{w_0}$ are equidimensional and the maps in \eqref{thediagram} are smooth, surjective and their fibers are irreducible schemes, $X^{\lambda,\para}_\gamma$ is equidimensional as well.
\end{proof}
\subsection{Local constancy of parabolic multiplicative affine Springer fibers}
In this section, we will show that for two ``close enough" points $a$ and $a'$ in the Hitchin base, their associated parabolic multiplicative Hitchin fibers are isomorphic to each other. 
We first need a description of the fibers in terms of the regular centralizer. Fix $w \in \Cox(W, S)$ and let $\gamma = \varepsilon^w_+(a)$. By Proposition \ref{jscheme}, we have the isomorphism
$$J_a \cong I_{\gamma} := \gamma^*I.$$
\begin{lem}[\cite{Chi22}, Lemma 5.1.2] \label{lemma5.1.2}
For any $g \in G(F_v)$, 
$\Ad_{g^{-1}}(\gamma) \in V_G(\mathcal{O}_v)$ if and only if $\Ad_{g^{-1}}(\gamma I_\gamma(\mathcal{O}_v))\subset V_G(\mathcal{O}_v).$
\end{lem}
Let $I$ be the Iwahori monoid of $V_G(\mathcal{O}_v),$ i.e. $I$ is the pre-image $\eva^{-1}_0(V_B(k))$ of $V_B(k)$ under the evaluation map. 
\begin{cor} \label{iwahoriver}
For any $g \in G(F_v)$, $\Ad_{g^{-1}}(\gamma) \in I$ if and only if $\Ad_{g^{-1}}(\gamma (I_\gamma(\mathcal{O}_v) \cap I)) \subset I$.
\end{cor}
Consider $a \in \mathfrak{C}_+(\mathcal{O}_v) \cap \mathfrak{C}^{\rs, \times}_+(F_v)$ with its associated cameral cover $\widetilde{X}_{a, v}$. By \cite[Lemma 3.5.2]{Ngo08}, there exists an integer $N$ such that for all $a' \in \mathfrak{C}_+(\mathcal{O}_v) \cap \mathfrak{C}^{\rs, \times}_+(F_v)$ satisfying $a \cong a'$ mod $t^N_v$, there is a $W$-equivariant isomorphism between the covers $\widetilde{X}_{a, v}$ and $\widetilde{X}_{a', v}$.
\begin{lem}[\cite{Chi22}, Lemma 5.1.3] \label{lemma5.1.3}
Let $a, a'\in \mathfrak{C}_+(\mathcal{O}_v) \cap \mathfrak{C}^{\rs, \times}_+(F_v)$ with $a \equiv a'$ mod $t^N_v$. Suppose that there exists a $W$-equivariant isomorphism between the cameral covers $\widetilde{X}_{a, v}$ and $\widetilde{X}_{a', v}.$ Then, there exists $g \in G(\mathcal{O}_v)$ such that $\Ad_{g^{-1}}(\gamma I_\gamma(\mathcal{O}_v)) = \gamma' I_{\gamma'}(\mathcal{O}_v)$, where $\gamma = \varepsilon^w_+(a)$ and $\gamma' = \varepsilon^w_+(a')$.
\end{lem}
We will now prove the local constancy of parabolic multiplicative affine Springer fibers using these lemmas.
\begin{thrm} \label{localconstancy}
There exists an integer $N$ such that for all $a' \in \mathfrak{C}_+(\mathcal{O}_v) \cap \mathfrak{C}^{\rs, \times}_+(F_v)$ with $a \cong a'$ mod $ t^N_v$, the parabolic multiplicative affine Springer fiber $M^{\para}_v(a)$, equipped with the action of $P_v(a),$ is isomorphic to $M^{\para}_v(a')$, equipped with the action of $P_v(a')$. 
\end{thrm}
\begin{proof} There exists an integer $N$ such that for all $a' \in \mathfrak{C}_+(\mathcal{O}_v) \cap \mathfrak{C}^{\rs, \times}_+(F_v)$ satisfying $a \cong a'$ mod $t^N_v$, the cover $\widetilde{X}_{a', v}$ is isomorphic to the cover $\widetilde{X}_{a, v}$ with the actions by $W$. By Lemma \ref{lemma5.1.3}, there exists $g \in G(\mathcal{O}_v)$ such that $\Ad_{g^{-1}}(\gamma I_\gamma(\mathcal{O}_v)) = \gamma' I_{\gamma'}(\mathcal{O}_v).$ 

Let $I$ be the Iwahori monoid of $V_G(\mathcal{O}),$ i.e. $I$ is the pre-image $\eva^{-1}_0(V_B(k))$ of $V_B(k)$ under the evaluation map. Then, $g$ satisfies $\Ad_{g^{-1}}(\gamma (I_\gamma(\mathcal{O}_v)\cap I)) = \gamma'(I_{\gamma'}(\mathcal{O}_v) \cap I) \subset I$ as well. By Corollary \ref{iwahoriver}, such a $g$ satisfies $\Ad_{g^{-1}}(\gamma) \in I$ and $\Ad_g(\gamma') \in I.$ This proves the isomorphism between  $M^{\para}_v(a)$ and $M^{\para}_v(a')$. 
\end{proof}

\subsection{Dimension formula}
In this section, we prove the dimension formula of parabolic multiplicative affine Springer fibers. Its proof is based on \cite[Theorem 5.2.1]{Chi22}. 
\begin{prop} \label{paradim}
Let $(a, x) \in \mathcal{A}^{\para, \heart}_X \times X$ be a point with a very $(G, \delta_a)$-ample boundary divisor $b \in \mathcal{B}_X$. Then,  
$$\dim M^{\para}_x(a_x) \cong \dim P_x(a_x).$$
\end{prop}
\begin{proof}
Let $\mathcal{O}_x$ be the completed local ring at $x$ with fraction field $F_x$ and uniformizer $t_x$ at $x$. With $a \in \mathfrak{C}_+(\mathcal{O}_x)$, form the Cartesian diagram 
\begin{center} 
\begin{tikzpicture}[node distance=2cm, auto]
\node (A) {$X_a$};
\node (B) [right of = A] {$\overline{T}_+$};
\node (C) [below of = A] {$\Spec(\mathcal{O}_x)$};
\node (D) [below of = B] {$\mathfrak{C}_+$};
\draw[->, below] (A) to node {\hspace{-1.5cm} $\mathlarger{\mathlarger{\mathlarger{\lrcorner}}}$}(B);
\draw[->, left] (A) to node {$\pi_a$} (C);
\draw[->] (B) to node {$\pi$} (D);
\draw[->] (C) to node {$a$} (D);
\end{tikzpicture}
\end{center}
where $X_a = \Spec R_a$ for a finite flat $\mathcal{O}_x$-algebra $R_a.$ Since $\charac(k)$ is coprime to $|W|$ by assumption, the product $F_a = R_a \otimes_{\mathcal{O}_x} F_x$ is a finite tamely ramified extension of $F_x$ of degree $e$. For each $s \in k,$ define $$a_s:= a(st_x + (1-s)t^e_x) \in \mathfrak{C}_+(\mathcal{O}_x) \cap \mathfrak{C}_+^{\times, \rs}(F_x).$$ It interpolates between $a_1 = a$ and the unramified conjugacy class $a_0 = a(t^e_x).$ Also, $M^{\para}_x(a_s) \cong M^{\para}_x(a)$ for $s \neq 0$ because we can obtain $a_s$ from $a$ by changing the uniformizer.

Let $N$ be a non-zero positive integer such that $M^{\para}_v(a)$ and $M^{\para}_v(a_0)$ only depends $a$ and $a_0$ respectively modulo $t^N_x$. Choose a $T$-torsor $L$ on $X$ that is trivialized on the formal neighbourhood of $x$ such that there exists a $T$-torsor $L'$ satisfying $(L')^{\otimes c} \cong L$ and for all $y \in X\setminus \{x\},$ the local evaluation map
\begin{equation} \label{localeval} \mathcal{A}_{X, L} = H^0(X, \mathfrak{C}^L_+) \rightarrow \mathfrak{C}_+(\mathcal{O}_x/t^N_x) \times \mathfrak{C}_+(\mathcal{O}_y/t^2_y)\end{equation}
is surjective. 

The family $\{a_s\}_s$ defines a curve $C$ in $\mathfrak{C}_+(\mathcal{O}_x/t^N_x \mathcal{O}_x)$. So, let $L_C$ be the inverse image of $C$ under the map \eqref{localeval}, which is closed in $\mathcal{A}_{X, L}.$ By \cite[Lemma 5.2.6]{Chi22}, there exists a constructible subset $Z_C$ of $L_C$ whose elements $a_+$ lie in $\mathcal{A}^{\ani}_{X, L}$ and $a_+(X \setminus \{x\})$ intersect the discriminant divisor transversally. In particular, $Z_C$ is fiberwise dense with respect to the projection $L_C \rightarrow C$ and it contains a fiberwise dense open subset $U_C$ of $L_C$. Thus, we can choose a section $\sigma$ of the map \eqref{localeval} such that $\sigma(C) \cap U_C \subseteq \mathcal{A}^{\ani}_{X, L}$ is non-empty and contains $\sigma(a_0).$

By the Product formula in Proposition \ref{paraprod}, we have 
$$ \dim M^{\para}_x(\sigma(a_0)_x) - \dim P_x(\sigma(a_0)_x) = \dim \mathcal{M}^{\para}_{\sigma(a_0), x} - \dim \mathcal{P}_{\sigma(a_0)}.$$
Since $\sigma(a_0)_x$ is unramified, we have that $\dim M^{\para}_x(\sigma(a_0)_x) = \dim P_x(\sigma(a_0)_x)$ by Theorem \ref{unramdim}. Thus, $\dim \mathcal{M}^{\para}_{\sigma(a_0), x} = \dim \mathcal{P}_{\sigma(a_0)}.$ Since $\sigma(C) \cap U_C \subseteq \mathcal{A}^{\ani}_{X, L}$, the restriction of the multiplicative Hitchin fibration to $\sigma(C) \cap U_C$ is proper. By upper semicontinuity of the fiber dimension, 
$$\dim \mathcal{M}^{\para}_{\sigma(a_s), x} \leq \dim \mathcal{M}^{\para}_{\sigma(a_0), x} = \dim \mathcal{P}_{\sigma(a_0)}$$
for all $\sigma(a_s) \in \sigma(C) \cap U_C$ with $s \neq 0.$ Since $\mathcal{P}$ is smooth over $\mathcal{A}_{X, L}$, we have $\dim \mathcal{P}_{\sigma(a_s)} = \mathcal{P}_{\sigma(a_0)}.$ Hence, $\dim \mathcal{M}^{\para}_{\sigma(a_s), x} = \dim \mathcal{P}_{\sigma(a_s)}.$
Applying the Product formula again gives
$$ \dim M^{\para}_x(\sigma(a_s)_x) - \dim P_x(\sigma(a_s)_x) = \dim \mathcal{M}^{\para}_{\sigma(a_s), x} - \dim \mathcal{P}_{\sigma(a_s)} = 0,$$
which shows that $\dim M^{\para}_x(\sigma(a_s)_x) = \dim P_x(\sigma(a_s)_x).$ Since $s \neq 0$ and $\sigma(a_s)_x = a_s$, the isomorphism $$M^{\para}_x(\sigma(a_s)_x) = M^{\para}_x(a_s) \cong M^{\para}_x(a_1) = M^{\para}_x(a)$$ implies that $\dim M^{\para}_x(a_x) = \dim P_x(a_x).$
\end{proof}

\subsection{Equidimensionality}
In this section, we will show that the parabolic multiplicative affine Springer fibers are equidimensional. First, we need to know the geometry of the local model of singularities, which is an admissible union of Iwahori cells. The latter arises as the special fiber of the global affine Schubert scheme constructed by \cite{Zhu14}. Here, Zhu considered the global affine Grassmannian $\Gr_{\mathcal{G}}$ of a Bruhat-Tits group scheme $\mathcal{G}$ over a curve $X$, equipped with isomorphisms
$$\Gr_{\mathcal{G}}\big|_{X^\times} \cong \Gr_G \times X^\times, \quad (\Gr_{\mathcal{G}}) \big|_0 \cong \Fl_G.$$
For $\lambda \in X_*(T)^+$, define $\overline{\Gr}^\lambda_{\mathcal{G}}$ as the reduced closure of $\Gr^{\leq \lambda}_G \times X^\times$ in $\Gr_{\mathcal{G}}$. 
\begin{thrm}[\cite{Zhu14}, Theorem 3] The reduced fiber $(\overline{\Gr}^\lambda_{\mathcal{G}})|_0$ is isomorphic to $\bigcup_{w\in \Adm(\lambda)} IwI.$
\end{thrm}
By \cite[\S 9.1]{PZ13} and \cite[\S 2.1]{HR19}, the scheme $\overline{\Gr}^\lambda_{\mathcal{G}}$ is normal, integral and Cohen-Macaulay. Hence, we have the following corollary.
\begin{cor}[\cite{Cass21}, Corollary 2.3.1] \label{cass}
The fiber $(\overline{\Gr}^\lambda_{\mathcal{G}})|_0 = \bigcup_{w\in \Adm(\lambda)} IwI$ is Cohen-Macaulay, connected and equidimensional of dimension $\dim \Gr^{\leq \lambda}_G.$
\end{cor}
This allows us to conclude that some parabolic multiplicative Hitchin fibers are Cohen-Macaulay and that the parabolic multiplicative Hitchin fibration is flat when restricted to certain very $(G, N)$-ample boundary divisors.
\begin{cor} \label{cohen}
If $b$ is very $(G, \delta_a)$-ample, the stack $\mathcal{M}^{\heart, \para}_{b, x}$ is Cohen-Macaulay.
\end{cor}
\begin{proof}
Over $b \in \mathcal{B}_X$, the local evaluation map to $[\widetilde{Q}_X]_{G, N}$ is smooth for some $N$ in Theorem \ref{localmodel}. Since the admissible union $\bigcup_{w\in \Adm(\lambda)} IwI$ is Cohen-Macaulay by Corollary \ref{cass}, the result follows. \end{proof}
\begin{cor} \label{flatflat}
Let $(a, x) \in \mathcal{A}^{\heart}_X \times X$ be a point with a very $(G, \delta_a)$-ample divisor $b \in \mathcal{B}_X.$ The restricted parabolic multiplicative Hitchin fibration $$h_{b, x}^{\para}: \mathcal{M}^{\ani, \para}_{b, x} \rightarrow \mathcal{A}^{\ani}_b \times \{x\}$$
is flat. 
\end{cor}
\begin{proof}
By the Product formula and Proposition \ref{paradim}, we have that $\dim \mathcal{M}^{\para}_{a, x} = \dim \mathcal{P}_a$ for each $(a, x) \in \mathcal{A}^{\ani}_b \times X.$ Since $\mathcal{P}$ is a smooth Deligne-Mumford stack over $\mathcal{A}^\heart_X$, the fiber dimension $h^{\ani}_{b, x}$ is constant. Since $\mathcal{M}^{\ani, \para}_{b, x}$ is Cohen-Macaulay by Corollary \ref{cohen}, $h^{\ani}_{b, x}$ is flat.
\end{proof}
\begin{prop} Let $\gamma \in G^{\rs}(F)$ and $\lambda \in X_*(T)^+$. The parabolic multiplicative affine Springer fiber $X^{\lambda, \para}_\gamma = M^{\para}_x(a)$ with $a = \chi_+(\gamma_\lambda)$ is equidimensional.
\end{prop}
\begin{proof}
Pick two distinct points $x, x_0 \in X$ and consider the boundary divisor 
$$\lambda_b = \lambda x + \lambda_0 x_0,$$
where $\lambda_0$ is chosen so that $\lambda_b$ is very $(G, \delta_a)$-ample and it satisfies the inequalities in \cite[\S 5.3.1]{Chi22}. By Theorem \ref{localconstancy}, there exists a positive integer $N$ such that $M^{\para}_x(a_x)$, equipped with the action of $P_x(a_x)$, only depends on $N$. By \cite[Lemma 5.3.4]{Chi22}, there exists a point $a_+ \in \mathcal{A}^{\ani}_b \subseteq \mathcal{A}^\heart_b$ such that $x_0 \notin \Supp(\mathfrak{D}_{a_+})$, $a_+(X \setminus x)$ is transversal to the discriminant divisor and $a_+ \equiv a$ mod $t^N_x.$ By the Product formula in Proposition \ref{paraprod}, there is a homeomorphism of stacks
$$[\mathcal{M}^{\para}_{a_+, x}/\mathcal{P}_{a_+}] \simeq [M^{\para}_x(a_{+, x})/P_x(a_{+, x})] \simeq [M^{\para}_x(a_x)/P_x(a_x)].$$
Since $\mathcal{M}^{\para}_{a_+, x}$ is equidimensional by Corollary \ref{flatflat}, $M^{\para}_x(a_x)$ is equidimensional. 
\end{proof}
\begin{cor} \label{parcor}
Let $\gamma \in G^{\rs}(F)$ and $\lambda \in X_*(T)^+$. Then, $\dim X^{\lambda, \para}_\gamma = \dim X^{\lambda}_\gamma.$ 
\end{cor}
\begin{proof}
We follow \cite[\S 2.5.14]{Yun16}. Since the projection $X^{\lambda, \para}_\gamma \rightarrow X^{\lambda}_\gamma$ is surjective, we have that $\dim X^{\lambda, \para}_\gamma \geq \dim X^{\lambda}_\gamma.$ Let $X^{\lambda, \para, \reg}_\gamma$ be the pre-image of $X^{\lambda, \reg}_\gamma$ under this projection map. Then, $X^{\lambda, \para, \reg}_\gamma \rightarrow X^{\lambda, \reg}_\gamma$ is an isomorphism. Since $X^{\lambda, \para}_\gamma$ is equidimensional, 
$$\dim X^{\lambda, \para}_\gamma = \dim X^{\lambda, \para, \reg}_\gamma = \dim X^{\lambda, \reg}_\gamma \leq \dim X^{\lambda}_\gamma.$$
\end{proof}
\printbibliography
\end{document}